\documentclass[a4paper,12pt]{amsart}
\usepackage[a4paper,margin=2.4cm]{geometry}
\usepackage{color}
\usepackage{amssymb}
\usepackage{amsmath}
\usepackage{mathrsfs}
\usepackage{cancel}
\usepackage{tikz-cd}
\usepackage{comment}
\usepackage{stmaryrd} 
\usepackage{hyperref}

\newtheorem{Theorem}{Theorem}[section]

\newtheorem{lemma}[Theorem]{Lemma}
\newtheorem{prop}[Theorem]{Proposition}
\newtheorem{cor}[Theorem]{Corollary}
\newtheorem{thm}[Theorem]{Theorem}
\theoremstyle{definition}
\newtheorem{defn}[Theorem]{Definition}

\theoremstyle{remark}
\newtheorem{rem}[Theorem]{Remark}

\newcounter{numl}
\newcommand{\labelnuml}{\textup{(\roman{numl})}}

\DeclareSymbolFont{script}{U}{eus}{m}{n}
\DeclareSymbolFontAlphabet{\amathscr}{script}
\DeclareMathSymbol{\Wedge}{0}{script}{"5E}
\DeclareMathAlphabet{\mathrmsl}{OT1}{cmr}{m}{sl}

\newcommand{\R}{\mathbb{R}}
\newcommand{\T}{\mathbb{T}}

\newcommand{\Pol}{P}
\newcommand{\Aut}{\mathrm{Aut}}
\newcommand{\tor}{{\mathfrak t}}

\newcommand{\A}{\mathcal{A}}

\newcommand{\E}{\mathcal{E}}
\newcommand{\End}{\mathrm{End}}
\newcommand{\Id}{\mathrm{Id}}
\newcommand{\db}{\bar{\partial}}
\renewcommand{\d}{\partial}
\newcommand{\tr}{\mathrm{tr}}
\newcommand{\rk}{\mathrm{rk}}
\newcommand{\Vol}{\mathrm{Vol}}
\newcommand{\pt}{\mathrm{pt}}

\newcommand{\Hess}{\mathrm{Hess}}
\newcommand{\Herm}{\mathrm{Herm}}
\newcommand{\eq}{\mathrm{eq}}
\newcommand{\F}{\mathcal{F}}
\newcommand{\Hom}{\mathrm{Hom}}
\newcommand{\Ric}{\mathrm{Ric}}
\newcommand{\Todd}{\mathrm{Todd}}
\newcommand{\ch}{\mathrm{ch}}
\newcommand{\Scal}{\mathrm{Scal}}

\begin{document}
	
	\title{The weighted Hermite--Einstein equation}
	\author{Michael Hallam and Abdellah Lahdili}
	\address{M. Hallam}\email{hallam.math@gmail.com}
	\address{A.\,Lahdili\\ D{\'e}partement de Math{\'e}matiques\\ Universit\'e du Qu\'ebec \`a Montr\'eal}\email{lahdili.abdellah@gmail.com}
	\date{\today}
	
	\begin{abstract}
		We introduce a new weighted version of the Hermite--Einstein equation, along with notions of weighted slope (semi/poly)stability, and prove that a vector bundle admits a weighted Hermite--Einstein metric if and only if it is weighted slope polystable. The new equation encompasses several well-known examples of canonical Hermitian metrics on vector bundles, including the usual Hermite--Einstein metrics, K{\"a}hler--Ricci solitons, and transversally Hermite--Einstein metrics on certain Sasaki manifolds.  We prove that the equation arises naturally as a moment map, that solutions to the equation are unique up to scaling, and demonstrate a weighted Kobayashi--L{\"u}bke inequality satisfied by vector bundles admitting a weighted Hermite--Einstein metric. As an application of our techniques, we extend a bound of Tian on the Ricci curvature \cite{Tia92} to a bound on a modified Ricci curvature, related to the existence of K{\"a}hler--Ricci solitons. Along the way, we introduce a new weighted vortex equation, as well as a weighted analogue of Gieseker stability. A key technical point is the application of a new extension of Inoue's equivariant intersection numbers \cite{Ino20} to arbitrary weight functions on the moment polytope of a K{\"a}hler manifold with Hamiltonian torus action.
	\end{abstract}
	
	\maketitle
	
	\section{Introduction}
	
	The interplay between the existence of canonical metrics and stability conditions arising from algebraic geometry has become a central theme in modern complex geometry. The most well-known and classical such result is the Kobayashi--Hitchin correspondence, proved by Donaldson and Uhlenbeck--Yau \cite{Don85, UY86}, which states that a holomorphic vector bundle over a compact K{\"a}hler manifold admits a Hermite--Einstein metric if and only if the bundle is slope polystable.
	
	Inspired by the recent advent of weighted cscK metrics and weighted K-stability \cite{Lah19, Ino22, Ino20} we introduce a new weighted analogue of the Hermite--Einstein equation. The new equation encompasses many known examples of canonical metrics on vector bundles, including the usual Hermite--Einstein metrics, K{\"a}hler--Ricci solitons, and transversally Hermite--Einstein metrics on vector bundles over Sasaki manifolds studied by Biswas--Schumacher \cite{BS10} and Baraglia--Hekmati \cite{BH22}. Following \cite[Section 5]{AJL}, we also show that the weighted Hermite--Einstein equation with polynomial weights reduces to the Hermite--Einstein equation over a semi-simple fibration, distinct from the original manifold.
	
	To briefly introduce the equation, let \((E,h)\) be a holomorphic vector bundle with Hermitian metric over a compact K{\"a}hler manifold \((X,\omega)\). We assume a compact real torus \(\T\) acts on this setup, so \(\T\) acts by holomorphic isometries on both \((E,h)\) and \((X,\omega)\), and the map \(E\to X\) is \(\T\)-equivariant. Under these conditions, a moment map \(\Phi_h\in\Gamma(\mathfrak{t}^*\otimes \End(E))\) for the curvature \(F_h\in\Omega^2(\End(E))\) of \((E,h)\) exists. We further assume the \(\T\)-action on \((X,\omega)\) is Hamiltonian, and let \(\mu:X\to\mathfrak{t}^*\) be a fixed moment map with moment polytope \(P:=\mu(X)\). Given a smooth \emph{weight function} \(v:P\to\mathbb{R}_{>0}\), we define the \emph{\(v\)-weighted contraction} of \(F_h+\Phi_h\) to be \[\Lambda_{\omega,v}(F_h+\Phi_h):=v(\mu)\Lambda_\omega F_h + \langle \Phi_h, dv(\mu)\rangle,\] where \(dv(p)\) is considered as an element of \(\mathfrak{t}^{**}\cong\mathfrak{t}\) for \(p\in P\), and \(\langle-,-\rangle\) denotes the canonical pairing \(\mathfrak{t}^*\otimes\mathfrak{t}\to\mathbb{R}\) extended to \(\Gamma(\mathfrak{t}^*\otimes\End(E))\otimes\mathfrak{t}\to\Gamma(\End(E))\). The \emph{\(v\)-weighted Hermite--Einstein equation} is then \[\frac{i}{2\pi}\Lambda_{\omega,v}(F_h+\Phi_h) = c_v\Id_E,\] where \(c_v\) is a suitable equivariant cohomological constant. For more details, see Section \ref{sec:wHE}.
	
	In this work, we give a thorough study of these metrics. Alongside the equation, we introduce notions of weighted slope (semi/poly)stability, and prove the following weighted form of the Kobayashi--Hitchin correspondence:
	
	\begin{thm}\label{thm:main_wHK}
		A vector bundle admits a weighted Hermite--Einstein metric if and only if it is weighted slope polystable.
	\end{thm} The proof that weighted slope polystability implies existence follows the approach of Uhlenbeck--Yau \cite{UY86}, namely a continuity method. Analytically there is little change in the proof. However, we highlight here the geometric importance of Lemma \ref{lem:codim2integral}, which states that one may compute the weighted slope of an equivariant subsheaf by differential geometric means, provided the singular locus of the sheaf has codimension at least 2. The proof of this lemma exemplifies the equivariant approach we take, and is an indispensable ingredient in the proof of the correspondence.
	
	A key technical tool that we develop and use extensively throughout the paper is a new extension of Inoue's equivariant intersection theory \cite{Ino20} to arbitrary weight functions on the moment polytope \(P\). These intersection numbers, defined via Fourier transforms, provide a powerful and convenient framework for expressing the complicated equivariant cohomological quantities that appear frequently in the weighted theory. The simplest examples of such intersection numbers are \[(v(\alpha_\T)) = \int_Xv(\mu)\omega^{[n]},\quad\quad (c_1(E)_\T\cdot v(\alpha_\T)) = \int_X\frac{i}{2\pi}\tr\Lambda_{\omega,v}(F_h+\Phi_h)\omega^{[n]},\] where \(\alpha_\T:=[\omega+\mu]\) denotes the \(\T\)-equivariant K{\"a}hler class, and \(\omega^{[n]}:=\omega^n/n!\). Indeed, the second number here provides a suitable notion of \(v\)-weighted degree of \(E\), allowing us to define weighted slope stability. While the notion of applying \(v\) to \(\alpha_\T\) might seem perverse at first, the notation is perfectly natural from the viewpoint of our Fourier transform definition.
	
	As a further application of the weighted intersection numbers, we state and prove a weighted version of the Kobayashi--L{\"u}bke inequality for vector bundles admitting weighted Hermite--Einstein metrics, providing an equivariant cohomological obstruction to the existence of such metrics.
	
	\begin{thm}
		Suppose that \(E\) admits a \(v\)-weighted Hermite--Einstein metric, and that \(v\) satisfies the inequality \begin{equation}\label{eq:weight-ineq-intro}
			\Hess(v) - \frac{n+1}{n}\frac{dv\otimes dv}{v} \leq 0,
		\end{equation} where \(n:=\dim_{\mathbb{C}}X\). Then \[(\rk(E)-1)(c_1(E)_\T^2\cdot v(\alpha_\T)) \leq 2\,\rk(E)(c_2(E)_\T\cdot v(\alpha_\T)).\] Furthermore, equality holds if and only if \(E\) is projectively flat.
	\end{thm}
	
	Note the unexpected emergence of inequality \eqref{eq:weight-ineq-intro}. While we show that \eqref{eq:weight-ineq-intro} is satisfied for all weight functions encompassing the examples listed above, the condition remains somewhat mysterious to us.
	
	Alongside these results, we show that the weighted Hermite--Einstein equation has a natural interpretation as a moment map on the space of connections, a result which in the unweighted setting goes back to Atiyah--Bott \cite{AB83}. The approach taken is that of Dervan--Hallam \cite{DH23}, namely by integrating equivariant forms over the fibres of a universal family. The argument requires an extension of Inoue's equivariant intersection 2-forms \cite{Ino20} to the case of arbitrary weight functions on the moment polytope.
	
	We prove the following uniqueness result.
	
	\begin{thm}
		A weighted Hermite--Einstein metric on a simple bundle is unique up to scaling.
	\end{thm}
	
	The proof by now is very standard---we introduce a weighted version of Donaldson's functional \cite{Don85}, which is the ``log-norm functional" in this context, and show that it is convex along geodesic rays in the space of metrics. Since critical points of this functional are precisely weighted Hermite--Einstein metrics, uniqueness is immediate.
	
	As an application of our techniques, we adapt certain results of Tian  \cite{Tia92} in the setting of K{\"a}hler--Einstein metrics to the setting of K{\"a}hler--Ricci solitons. In particular, given a Fano manifold \((X,\omega)\) with \(\T\)-action and \(\xi\in\mathfrak{t}\), we define a certain invariant \(\beta_\xi(X)\leq1\) in terms of extensions of \(TX^{1,0}\), and prove:
	
	\begin{thm}
		There is no \(\T\)-invariant K{\"a}hler metric in \(c_1(X)\) satisfying \(\Ric(\omega)-\frac{i}{2\pi}\d\db\mu^\xi\geq t\omega\) and \(t>\beta_\xi(X)\).
	\end{thm}
	
	In particular, if \(\beta_\xi(X)<1\), a K{\"a}hler--Ricci soliton with vector field \(\xi\) does not exist. The proof, similarly to Tian's considers extensions of the holomorphic tangent bundle of \(X\). While we cannot prove the existence of an exponentially weighted Hermite--Einstein metric on such extensions, the existence of a weaker kind of canonical metric is sufficient for our needs.
	
	In the course of the paper, we introduce two tangentially related notions: weighted vortices, and weighted Gieseker stability. The \(v\)-solitons of Han--Li \cite{HL23} are shown to induce weighted vortices, and weighted Gieseker stability is related to weighted slope stability. It is clear from the evidence presented here, along with the already existing weighted cscK theory, that these notions form part of an entire ``weighted ecosystem" of canonical metrics and stability conditions. We claim that the framework developed here, in particular the extension of Inoue's intersection theory to arbitrary weight functions, provides the correct tool for further analysis in this field.

	\subsection*{Outline} In Section \ref{sec:wHE}, we define the weighted Hermite--Einstein equation, after reviewing necessary background material on vector bundles, connections, and vector bundle valued moment maps. 
	
	In Section \ref{sec:examples}, we give examples of the weighted Hermite--Einstein equation for various choices of weight functions. We prove that line bundles admit weighted Hermite--Einstein metrics for any choice of weight function, and that existence of solutions to the weighted Hermite--Einstein equation is preserved under small deformations of the weight function. We further introduce the weighted vortex equation, and show that \(v\)-solitons are special examples of weighted vortices. 
	
	In Section \ref{sec:eq-intersections}, we review the weighted intersection theory introduced by Inoue \cite{Ino20}, and extend it to the case of an arbitrary smooth function on the moment polytope using Fourier transform techniques. 
	
	In Section \ref{sec:moment-map-and-stability}, the theory of Section \ref{sec:eq-intersections} is applied to show that the weighted Hermite--Einstein equation arises naturally as a moment map, and to define the weighted slope of a vector bundle along with weighted slope stability. We then prove that a weighted Hermite--Einstein vector bundle is weighted slope polystable, and introduce the notion of weighted Gieseker stability. As an aside, we also apply the full weighted Kobayashi--Hitchin correspondence to study transverse Hermite--Einstein metrics on Sasaki manifolds, and polynomially weighted Hermite--Einstein metrics. 
	
	In Section \ref{sec:Lubke}, we state and prove the weighted analogue of the Kobayashi--L{\"u}bke inequality. 
	
	In Section \ref{sec:extensions} we extend a bound of Tian on the Ricci curvature \cite{Tia92} to a bound on a weighted Ricci curvature, obstructing the existence of K{\"a}hler--Ricci solitons. 
	
	In Section \ref{sec:uniqueness}, we introduce the weighted Donaldson functional, and use its convexity to prove uniqueness of weighted Hermite--Einstein metrics up to rescaling. 
	
	In Section \ref{sec:DUY} we prove that a weighted slope polystable vector bundle admits a weighted Hermite--Einstein metric, completing the proof of the weighted Kobayashi--Hitchin correspondence.
	
	\renewcommand{\abstractname}{Acknowledgements}
	\begin{abstract}
		The authors thank Zakarias Sj{\"o}str{\"o}m Dyrefelt for early discussions on the project. They also thank Ruadha{\'i} Dervan for suggesting the idea of weighted Gieseker stability and Eiji Inoue and Carlo Scarpa for stimulating discussions. The first named author thanks Lars Sektnan, Annamaria Ortu, and John McCarthy for their interest in and comments on the project. The second named author would like to thank Vestislav Apostolov, Yoshinori Hashimoto, Julien Keller and Eveline Legendre for their insightful discussions. The second named author also thanks CIRGET-UQAM for hospitality and financial support during the preparation of this work.
	\end{abstract}
	
	\section{The weighted Hermite--Einstein equation}\label{sec:wHE}
	
	In this section, we define the weighted Hermite--Einstein equation, after first giving the necessary background on vector bundles, connections, and vector bundle valued moment maps.
	
	\subsection{Background} Let \(E\) be a holomorphic vector bundle over a compact K{\"a}hler manifold \((X,\omega)\) of dimension \(n\). We recall that for any choice of Hermitian metric \(h\) on \(E\), there is a canonically induced connection \(\nabla^h\) on \(E\), called the \emph{Chern connection}---it is the unique Hermitian connection whose \((0,1)\)-part is the del-bar operator of \(E\).
	
	Suppose we choose a local holomorphic frame for \(E\), and identify \(h\) with a smoothly varying Hermitian matrix. Then the curvature of the Chern connection in this frame is computed as \[F_h=\db(h^{-1}\d h),\] where the \(\d\) and \(\db\) operators are those on \(X\).
	
	The metric \(h\) is called a \emph{Hermite--Einstein metric} if \[\frac{i}{2\pi}\Lambda_\omega F_h = \lambda\, \Id_E,\] where \(\Lambda_\omega F_h:=\frac{n F_h\wedge\omega^{n-1}}{\omega^n}\) is the contraction of \(F_h\) with \(\omega\), and \(\lambda := \frac{n\deg(E, \omega)}{\rk(E)\Vol(X,\omega)}\) is the Einstein constant.
	
	Suppose now that a compact torus \(\T\) acts on \((X,\omega)\) by Hamiltonian isometries. In particular, the action is holomorphic and there exists a moment map \(\mu:X\to\mathfrak{t}^*\) for the action. Suppose we are given a linear lift of the \(\T\)-action to \(E\) by biholomorphisms, and that \(h\) is a \(\T\)-invariant Hermitian metric on \(E\).

	\begin{defn}
		A \emph{moment map} for the curvature \(F_h\) of a \(\T\)-invariant metric \(h\), is a \(\T\)-invariant section \(\Phi_h\in\Gamma(\mathfrak{t}^*\otimes\End(E))\) such that \[F_h(-,\xi) = {\nabla^h}\langle\Phi_h,\xi\rangle\] for all \(\xi\in\mathfrak{t}\). Here on the left-hand side we identify \(\xi\) with the vector field it generates on \(X\). On the right-hand side, \({\nabla^h}\) denotes the canonical extension of the Chern connection to \(\End(E)\), and \(\langle-,-\rangle\) denotes the dual pairing \(\mathfrak{t}^*\otimes\mathfrak{t}\to\mathbb{R}\).
	\end{defn}
	
	When \(E = L\) is an ample line bundle with Hermitian metric \(h\) whose curvature is a K{\"a}hler form \(\omega\), a familiar fact is that a choice of lift of the \(\T\)-action to \(L\) canonically induces a moment map for \(\omega\). In the vector bundle setting we have a similar fact, which requires the notion of the infinitesimal action of \(\T\) on sections of \(E\).
	
	First, note that a lift of the \(\T\)-action to \(E\) induces an action on sections of \(E\) by \[(t\cdot s)|_{x}:= t(s|_{t^{-1}x}),\] where \(t\in\T\), \(s\in\Gamma(E)\) and \(x\in X\). This action then induces an infinitesimal action on sections, defined by \begin{equation}\label{eq:infinitesimal_action}
		\mathcal{L}^E_\xi(s):=\left.\frac{d}{dt}\right|_{t=0}(\exp(-t\xi)\cdot s)
	\end{equation}
	for all \(s\in\Gamma(E)\). It is straightforward to see that \begin{equation}\label{eq:infinitesimal_action_on_sections}
		\mathcal{L}_\xi^E(fs)=\xi(f)s+f\mathcal{L}_\xi^E(s)
	\end{equation} for all \(f\in C^\infty(X)\) and \(s\in\Gamma(E)\).  Given this definition, we have:
	\begin{prop}[{\cite[Definition 7.5]{BGV04}}]\label{prop:canonical_moment_map}
		Given a lift of the \(\T\)-action to \((E, h)\), there exists a canonically defined moment map \(\Phi_h\) for the curvature \(F_h\) of the Chern connection, defined by \[\langle\Phi_h, \xi\rangle := \nabla^h_\xi - \mathcal{L}^E_\xi\] for each \(\xi\in\mathfrak{t}\).
	\end{prop}
	
	The fact that \(\Phi_h\) is a well-defined endomorphism can be observed from \eqref{eq:infinitesimal_action_on_sections}, and that \(\Phi_h\) satisfies the moment map property is a consequence of the equivariant Bianchi identity \cite[Proposition 7.4]{BGV04}. We can also easily see that the endomorphism \(\Phi_h\) is skew-adjoint: for $s,s'\in \Gamma(E)$ we have
	\begin{equation}
		\begin{split}\label{computation}
			h(\nabla^h_\xi s,s')+ h(s,\nabla^h_\xi s')=&\mathcal{L}_\xi(h(s,s'))\\
			=&(\mathcal{L}^E_\xi h)(s,s')+ h(\mathcal{L}^E_\xi s,s')+h(s,\mathcal{L}^E_\xi s')\\
			=& h(\mathcal{L}^E_\xi s, s')+h(s,\mathcal{L}^E_\xi s'),
		\end{split}
	\end{equation}
	where in the final line we use that $h$ is $\T$-invariant, so $\mathcal{L}^E_\xi h=0$. Rearranging the above and using the definition of \(\Phi_h\) yields
	\[
	h(\Phi_h^\xi (s),s')=-h(s,\Phi_h^\xi(s')).
	\]
	
	Thus, given a lift of the \(\T\)-action to \((E, h)\), we will always take \(\Phi_h\) to be the canonically induced moment map for \(F_h\) from Proposition \ref{prop:canonical_moment_map}. Note that if we change Hermitian metrics from \(h\) to \(h'\), this formula tells us how the canonically induced moment maps change. Indeed, let \(\theta := \nabla^{h'}-\nabla^h\) be the \(\End(E)\)-valued \((1,0)\)-form on \(X\) describing the change of connections. Then \begin{equation}\label{eq:moment-map-change}
		\langle \Phi_{h'}, \xi \rangle - \langle \Phi_h, \xi \rangle = \theta(\xi)
	\end{equation} for all \(\xi\in \mathfrak{t}\); we will sometimes abuse notation by writing this relations as \[\Phi_{h'}-\Phi_h = \theta.\]
	
	It will be convenient to introduce the equivariant \(\db\)-operator \(\db^\eq\) of \(\End(E)\), which takes \(\chi\in\Omega^{1,0}(\End(E))^\T\) and maps it to the equivariant \(\End(E)\)-valued \((1,1)\)-form defined by \[\langle\db^{\eq}\chi, \xi\rangle := \db\chi + \chi(\xi),\] where \(\xi\in\mathfrak{t}\) and on the right-hand side we conflate \(\xi\) with the vector field it generates on \(X\). If \(\theta:=\nabla^{h'}-\nabla^h\), the relationship between equivariant curvatures may be then expressed succinctly as \begin{equation}\label{eq:change-of-curvature}
		F_{h'}+\Phi_{h'} = F_h + \Phi_h + \db^\eq\theta.
	\end{equation}
	
	\begin{rem}
		In the case where \(E=L\) is a line bundle, \(\frac{i}{2\pi}F_h=\omega\) is a K{\"a}hler metric, and \(h' = e^{-\varphi}h\) for some \(\T\)-invariant K{\"a}hler potential \(\varphi\), equation \eqref{eq:moment-map-change} reduces to the familiar formula for the change in moment maps \[\mu_\varphi - \mu = d^c\varphi;\] see \cite[Lemma 1]{Lah19}.

	\end{rem}

	\subsection{Definition of the equation}
	
	As in the previous section, let \((E, h)\) be a holomorphic Hermitian vector bundle on a compact K{\"a}hler manifold \((X, \omega)\), with curvature \(F_h\). We assume that a compact torus \(\T\) acts on \((X, \omega)\) by Hamiltonian isometries, and fix a moment map \(\mu:X\to\mathfrak{t}^*\) for the action. We further assume a fixed linear lift of the \(\T\)-action to \(E\) such that \(\T\) acts on \(E\) by biholomorphisms preserving \(h\). We write \(\Phi_h\in\Gamma(\mathfrak{t}^*\otimes\End(E))\) for the canonically induced moment map for \(F_h\).
	
	Denote by \(P:=\mu(X)\) the moment polytope of \(\mu\), and let \(v,w:P\to\mathbb{R}_{>0}\) be positive smooth functions.
	
	\begin{defn}
		The \emph{\(v\)-weighted contraction} of \(F_h+\Phi_h\) is defined by \[\Lambda_{\omega,v}(F_h+\Phi_h):=v(\mu)\Lambda_\omega F_h + \langle \Phi_h, dv(\mu)\rangle,\] where \(dv:T^*P\to\mathbb{R}\) is considered as a map \(P\to \mathfrak{t} \cong \mathfrak{t}^{**}\).
	\end{defn}
	
	\begin{defn}
		A \(\T\)-invariant Hermitian metric \(h\) on \(E\) is called a \emph{\((v,w)\)-weighted Hermite--Einstein metric} if it satisfies the equation \[\frac{i}{2\pi}\Lambda_{\omega,v}(F_h+\Phi_h) = c_{v,w}w(\mu)\Id_E,\] for a real constant \(c_{v,w}\).
	\end{defn}
	
	When the weight functions are understood or irrelevant, we shall refer to such a metric as a \emph{weighted Hermite--Einstein metric}. The constant \(c_{v,w}\) will be shown in Lemma \ref{constants-in-wHE} to depend only on the equivariant cohomology of the manifold and the vector bundle.
	
	This definition is modelled on the weighted cscK metrics of \cite{Lah19}. However, we will now show that for the purposes of solving the equation, the second weight function \(w\) is superfluous.
	
	\begin{lemma}\label{vw-weighted-is-v-weighted}
		A \((v,w)\)-weighted Hermite--Einstein metric exists on \(E\) if and only if a \((v,1)\)-weighted Hermite--Einstein metric exists.
	\end{lemma}
	
	A \((v,1)\)-weighted Hermite--Einstein satisfies the equation \begin{equation}\label{eq:v-weighted-HE}
		\frac{i}{2\pi}\Lambda_{\omega, v}(F_h + \Phi_h) = c_v \Id_E,
	\end{equation} and we shall refer to such a metric simply as a \emph{\(v\)-weighted Hermite--Einstein metric}.
	
	\begin{proof}
		Let \(h\) be a Hermitian metric on \(E\). Given a \(\T\)-invariant smooth function \(f:X\to\mathbb{R}\), let \(h':=e^fh\). The Chern connections of these metrics satisfy \[\nabla^{h'}-\nabla^h = \d f\cdot\Id_e.\] Hence by \eqref{eq:change-of-curvature}, the equivariant curvatures change according to \[F_{h'}+\Phi_{h'} = F_h + \Phi_h + \db^\eq\d f\cdot\Id_E,\] using that \(\db(\Id_E) = 0\); here \(\db^\eq\) is the equivariant del-bar operator on \(\Omega^{1,0}(X)^\T\). Hence,\begin{align*}
			i\Lambda_{\omega,v}(F_{h'}+\Phi_{h'}) :=&
			i\Lambda_{\omega,v}(F_h+\Phi_h) + i\Lambda_{\omega,v}(\db^\eq\d f)\Id_E \\
			=& i\Lambda_{\omega,v}(F_h+\Phi_h) + \Delta_vf\cdot\Id_E,
		\end{align*} where \[\Delta_v f := i\Lambda_{\omega,v}\db^\eq\d f = i \Lambda_\omega \db(v(\mu)\d f) = \d^*(v(\mu)\d f) = v(\mu)\Delta(v)+i\langle \d f, dv(\mu)\rangle\] is the weighted Laplacian; in the final equality \(\Delta f := \d^*\d f\). Thus, if we can solve \begin{equation}\label{eq:Laplace-equation}
			\Delta_vf = 2\pi c_{v,w}w(\mu) - 2\pi c_v,
		\end{equation} then \(h'\) is a \((v,1)\)-weighted Hermite--Einstein metric if and only if \(h\) is a \((v,w)\)-weighted Hermite--Einstein metric. The operator \(\Delta_v\) is self-adjoint, elliptic, and has kernel the constant functions, so we can solve the equation \(\Delta_v f=g\) for \(f\) if and only if the integral of \(g\) vanishes. The right-hand side of \eqref{eq:Laplace-equation} has vanishing integral by the definitions of \(c_{v,w}\) and \(c_v\), so we are done.
	\end{proof}

	\section{Examples}\label{sec:examples}
	
	In this section we give several examples of weighted Hermite--Einstein metrics. We also show in Corollary \ref{Cor:LineBund}that line bundles admit weighted Hermite--Einstein metrics for arbitrary weight functions, and in Proposition \ref{prop:deformation} that the weight function of a weighted Hermite--Einstein metric may be deformed to produce new weighted Hermite--Einstein metrics.
	
	\subsection{The Hermite--Einstein equation}
	
	Choosing \(v=w=1\), we recover the usual Hermite--Einstein equation. Note that by Donaldson's uniqueness of Hermite--Einstein metrics \cite{Don85}, given a holomorphic lift of the \(\T\)-action to a simple bundle \(E\), any Hermite--Einstein metric must be \(\T\)-invariant, otherwise we could generate new examples by pulling back the metric.\footnote{Uniqueness is only up to constant factors, but one may easily check using compactness of \(\T\) that pulling back by \(\T\) does not rescale the metric.} Thus, we do not lose anything by restricting to \(\T\)-invariant metrics.
	
	\subsection{Weighted Hermite--Einstein metrics on line bundles}
	
	The following is a simple corollary of the proof of Lemma \ref{vw-weighted-is-v-weighted}.
	
	\begin{cor}\label{Cor:LineBund}
		Let \(L\) be a \(\T\)-equivariant holomorphic line bundle. Then for any choice of weight function \(v\), \(L\) admits a \(v\)-weighted Hermite--Einstein metric.
	\end{cor}
	
	\begin{proof}
		For any \(\T\)-invariant metric \(h\) on \(L\), \[\frac{i}{2\pi}\Lambda_{\omega, v}(F_h+\Phi_h) = g\,\Id_L\] for some \(\T\)-invariant smooth function \(g:X\to\mathbb{R}\). Let \(f\) solve the equation \(\Delta_v f = -g + c_v\), where \(c_v\) is the appropriate constant. Then \(h':=e^fh\) satisfies \[\frac{i}{2\pi}\Lambda_{\omega, v}(F_{h'}+\Phi_{h'}) = g\Id_L - g\Id_L + c_v\Id_L = c_v\Id_L.\qedhere\]
	\end{proof}
	
	\subsection{The exponentially weighted Hermite--Einstein equation and K{\"a}hler--Ricci solitons}\label{sec3.3}
	
	Let $(X,\omega,\xi)$ be  a compact K{\"a}hler--Ricci soliton: \begin{equation}\label{eq:KRsoliton}
		{\rm Ric}(\omega)-\omega=\frac{i}{2\pi}\d\db\langle\mu,\xi\rangle,
	\end{equation} where \(\Ric(\omega)\in c_1(X)\), and let $D$ be the corresponding Chern connection on the holomorphic tangent bundle $TX^{1,0}$. Using the Hermitian metric \(g\), we may transpose the \((0,1)\)-form part of the equation \eqref{eq:KRsoliton} to a \((1,0)\)-vector field to get \[\frac{-1}{2\pi}\Lambda_\omega R_\omega - i\,\Id_{TX^{1,0}} = \frac{1}{2\pi}\langle\Phi_\omega,\xi\rangle,\]
	where \(R_\omega\) is the full curvature tensor of \((X,\omega)\) and \(\Phi_\omega:=g^{-1}i\d\db\mu\) is a moment map for \(R_\omega\); see \cite[Proposition 4.1]{DH23}.\footnote{In \cite{DH23} the moment map is written as \(g^{-1}i\db\d\mu\), where \(\db\) is the del-bar operator of the holomorphic vector bundle \(T^*X^{1,0}\). Here, we instead work with the del and del-bar operators on forms, and \emph{define} \(g^{-1}(dz^j\wedge d\bar{z}^k):=g^{\bar{k}\ell}dz^j\otimes\frac{\d}{\d z^\ell}\). With this convention, the endomorphism \(g^{-1}i\d\db\mu\) used here agrees with that of \cite{DH23}; in local coordinates they are both written \(ig^{\bar{k}\ell}\d_j\d_{\bar{k}}\mu \,dz^j\otimes\frac{\d}{\d z^\ell}\).} Multiplying by \(ie^{\langle\mu,\xi\rangle}\), we see \[\frac{i}{2\pi}\Lambda_{\omega,e^{\langle-,\xi\rangle}}(R_\omega+\Phi_\omega) = e^{\langle\mu,\xi\rangle}\Id_{TX^{1,0}},\] hence a K{\"a}hler--Ricci soliton is an exponentially weighted Hermite--Einstein metric with weight functions \(v(\mu)=w(\mu)=e^{\langle\mu,\xi\rangle}\). Notice that $\Phi_\omega$ is {\it canonically normalized} in the sense of \cite{BGV04}:
	\[
	\langle \Phi_\omega,\xi\rangle=g^{-1}i\partial\db\langle\mu,\xi\rangle=D\xi=-\mathcal{L}^{TX^{1,0}}_\xi+ D_\xi,
	\]
	where $D$ is the Chern connection of $g$ and $\mathcal{L}^{TX^{1,0}}_\xi(V):=[\xi,V]=D_\xi V-D_V\xi$ is the naturally induced infinitesimal action. The second equality follows from the straightforward identity
	\[
	D\left(i(\db-\partial)\langle\mu,\xi\rangle\right)= i\partial\db\langle\mu,\xi\rangle+\frac{1}{2}\mathcal{L}_\xi g=i\partial\db \langle\mu,\xi\rangle,
	\]
	together with $g^{-1}\left(i(\db-\partial)\langle\mu,\xi\rangle\right)=\xi$ and the fact that \(g^{-1}\) commutes with \(D\).
	
	\subsection{\(v\)-solitons and weighted vortices}
	
	Recall the \(v\)-soliton equation of Han--Li \cite{HL23} defined by \[\mathrm{Ric}(\omega)-\omega=\frac{i}{2\pi}\d\db\log v(\mu).\] A natural question is whether a \(v\)-soliton on \(TX^{1,0}\) induces a weighted Hermite--Einstein metric on \(TX^{1,0}\). While we cannot ascertain the existence of a weighted Hermite--Einstein metric on \(X\) in this more general setting, we instead show that the existence of a \(v\)-soliton implies a suitably related ``weighted vortex" exists on \(TX^{1,0}\).
	
	\begin{lemma}\label{Lem:WKRS-Stab}
		Let $(X,\omega)$ be a $v$-soliton ${\rm Ric}(\omega)-\omega=\frac{i}{2\pi}\d\db\log v(\mu)$ with weight function $v\in C^\infty(P,\mathbb{R}_{>0})$. Then the induced Hermitian metric $h_\omega:=\frac{1}{2}(g-\sqrt{-1}\omega)$ on $TX^{1,0}$ solves a $v$-weighted vortex equation
		\[
		\begin{cases}
			\db^{TX^{1,0}}\varphi=0\\
			i(\Lambda_\omega R_{\omega}+\langle \Phi_{\omega}, d(\log v)(\mu)\rangle)- \left\langle \varphi^\star\otimes\varphi, \Hess(\log v)(\mu)\right\rangle= 2\pi\,\Id_{TX^{1,0}}.
		\end{cases}
		\]
		where $\varphi\in \Gamma(TX^{1,0})\otimes \tor^{*}$ is given by $\varphi(\xi):=\xi^{(1,0)}$, and \(\varphi^\star\in\Gamma(T^*X^{1,0})\otimes\tor^*\) is defined by \(\varphi^\star(\xi):=h_\omega(-,\xi^{1,0})\).
	\end{lemma}
	\begin{proof}
		Let $(\xi_a)_{a=1,\cdots,\ell}$ be a basis of $\mathbb{S}^1$ generators of $\tor$. Expanding the $v$-soliton equation gives
		\[
		\begin{split}
			{\rm Ric}(\omega)- \omega=& \frac{i}{2\pi}\left(\sum_{a=1}^\ell\frac{v_{,a}(\mu)}{v(\mu)}\partial\db\mu^{\xi_a}+\sum_{a,b=1}^\ell \left(\frac{v_{,b}(\mu)}{v(\mu)}\right)_{,a}\partial \mu^{\xi_a}\wedge \db\mu^{\xi_b}\right)\\
			=& \frac{i}{2\pi}\left(\langle  \partial\db\mu, d(\log v)(\mu)\rangle +  \left\langle \partial\mu \wedge \db\mu, \Hess(\log v)(\mu) \right\rangle\right)
		\end{split}
		\]
		Hence the Hermitian metric on $TX^{1,0}$ induced by the $v$-soliton $(g,\omega)$ satisfies the equation
		\begin{equation}\label{eq:awHE}
			-\frac{1}{2\pi} \Lambda_\omega R_{\omega}-\frac{1}{2\pi}\langle \Phi_{\omega}, d(\log v)(\mu)\rangle-i\,\Id_{TX^{1,0}}= \frac{i}{2\pi}\left\langle
			\Upsilon, \Hess(\log v)(\mu)\right\rangle
		\end{equation}
		where $\Upsilon_{ab}$ is the endomorphism  of $TX^{1,0}$ given by,
		\[
		\Upsilon_{ab}=g^{-1}(\partial\mu^{\xi_a}\wedge \db\mu^{\xi_b}).
		\] With the conventions of Section \ref{sec3.3}, one readily computes \[\Upsilon_{ab} = h_\omega(-,\xi_a^{1,0})\otimes \xi_b^{1,0} = \varphi^\star(\xi_a)\otimes\varphi(\xi_b).\]
	Hence, the Hermitian metric $h_\omega$ corresponding to the weighted soliton $(g,\omega)$ satisfies the following equation
	\[
	i(\Lambda_\omega R_{\omega}+\langle \Phi_{\omega}, d(\log v)(\mu)\rangle)- \left\langle \varphi^\star\otimes\varphi, \Hess(\log v)(\mu) \right\rangle= 2\pi\,\Id_{TX^{1,0}}.
	\]
	Lastly, we have $\db^{TX^{1,0}}\varphi=0$ as $\T$ acts holomorphically on $X$.
\end{proof}

More generally, for a $\T$-equivariant holomorphic vector bundle $E$ on a K\"ahler $\T$-manifold $(X,\omega)$ with weight functions $v,w\in C^{\infty}(\Pol, \mathbb{R}_{>0})$, we say that a $\T$-invariant Hermitian metric $h$ on $E$ together with a section $\varphi\in\Gamma(E)^{\T}\otimes\tor^*$ form a $(v,w)$-weighted vortex if the pair \((h,\varphi)\) solves the equation
\begin{equation}
	\begin{cases}\label{eq:v-vortex}
		\db^E\varphi=0,\\
		i\frac{\Lambda_{\omega,v}( F_{h}+\Phi_{h})}{v(\mu)} - \left\langle  \varphi^\star\otimes\varphi, \Hess(\log w)(\mu)\right\rangle=\tau\Id_E
	\end{cases}
\end{equation}
for some \(\tau\in\mathbb{R}\), where $\varphi^\star\otimes\varphi$ is the $\tor^{*}\otimes\tor^{*}$-valued section of $\End(E)$ given by
\begin{equation}\label{eq:phi*phi}
	(\varphi^\star\otimes\varphi)(s):=h(s, \varphi)\varphi,
\end{equation}
for all $s\in E$. The equation \eqref{eq:v-vortex} is a weighted version of the vortex equation of Bradlow \cite{Brad91}. A Kobayashi--Hitchin--Bradlow type correspondence for the weighted vortex equation with applications to the weighted Ricci solitons will appear in a subsequent paper \cite{HL2}.

\subsection{Transverse Hermite--Einstein metrics on Sasaki manifolds}\label{sec:transverse}

In this section, we show that for a certain choice of weight functions, one recovers a special case of the transverse Hermite--Einstein metrics studied by Biswas--Schumacher in the Sasaki setting \cite{BS10}, and Baraglia--Hekmati on more general foliated manifolds \cite{BH22}. The result is analogous to the correspondence between weighted cscK metrics with weight functions \(v(\mu) = (\langle\mu,\xi\rangle +a)^{-n-1}\), \(w(\mu) = (\langle\mu,\xi\rangle+a)^{-n-3}\) and constant scalar curvature Sasaki metrics explored in \cite{AC21, ACL21, AJL, Lah19}.

Let \((X,\omega)\) be a K{\"a}hler manifold whose metric \(\omega\) is \(i\) times the curvature of a Hermitian metric on an ample line bundle \(L\to X\). Denote by \(S\) the unit circle bundle of \(L^{*}\), which is a principal \(S^1\)-bundle on \(X\). The Chern connection on \(L\) induces a principal bundle connection \(A'\in\Omega^1(S,i\mathbb{R})\) on \(S\), whose curvature is \(i\omega\); we define \(A:=\frac{1}{i}A'\in \Omega^1(S,\mathbb{R})\). Denote by \(u\) the fundamental vector field on \(S\) generated by the \(S^1\)-action. Then \(A(u)=1\), and \(u\) is invariant under the \(S^1\)-action. We write \(D:=\ker A\) for the horizontal subspace of \(TS\).

Let \(E\to X\) be a holomorphic vector bundle with Hermitian metric \(h\). We write \(P\to X\) for the associated principal \(U(r)\)-frame bundle. The Chern connection on \(E\) induces a connection \(\theta\) on \(P\), which is a \(\mathfrak{u}(r)\)-valued 1-form on \(P\) satisfying:\begin{enumerate}
	\item \(g^*\theta = \mathrm{Ad}(g^{-1})\theta\) for all \(g\in U(r)\);
	\item \(\theta(a) = a\) for all \(a\in\mathfrak{u}(r)\), where on the left-hand side \(a\) is identified with the fundamental vector field it generates on \(P\).
\end{enumerate}

Now let \(\xi\) be a holomorphic Killing field on \(M\) admitting a Hamiltonian function \(f\), which we take to be strictly positive. Following \cite{AC21}, we twist the contact structure on \(S\) as follows. On \(S\), consider the contact form \(A_\xi:=\frac{1}{f}A\), and vector field \[u_\xi = fu + \xi_D,\] where \(\xi_D\) denotes the horizontal lift of \(\xi\) to \(S\). Since \(dA_\xi = -\frac{1}{f^2}df\wedge A + \frac{1}{f}dA\), these satisfy the following properties \begin{enumerate}
	\item \(A_\xi(u_\xi)=1\);
	\item \(dA_\xi(u_\xi,-)=0\);
	\item the restriction of \(dA_\xi\) to \(D\subset TS\) coincides with \(\frac{1}{f}\pi^*\omega|_D\).
	
\end{enumerate}

We assume that \(\xi\) has an \(F_h\)-Hamiltonian, i.e. a section \(\Phi_h^\xi\) of \(\mathrm{End}(E)\) satisfying \[\iota_\xi F_h = -\nabla^h\Phi_h^\xi,\] where \(\nabla^h\) is the Chern connection of \((E,h)\) extended to \(\End(E)\) and \(F_h\) is the curvature of \(\nabla^h\). In terms of the principal bundle \(P\), the section \(\Phi_h^\xi\) may be considered as a smooth map \(\Phi:P\to\mathfrak{u}(r)\) satisfying \[g^*\Phi = \mathrm{Ad}(g^{-1})\Phi\] for all \(g\in U(r)\). The moment map condition may then be written on the principal bundle as \[\iota_{\bar{\xi}}(d\theta) = -d_{\mathcal{H}}\Phi,\] where \(d_{\mathcal{H}}\Phi\) denotes the restriction of \(d\Phi\) to the horizontal subspace \(\mathcal{H}:=\ker\theta\), and \(\bar{\xi}\in\mathcal{H}\) denotes the horizontal lift of \(\xi\) from \(X\) to \(P\).

Now we further pull back the bundles \(E\) and \(P\) to \(S\); we shall denote these pullbacks via the symbols \(\tilde{E}\) and \(\tilde{P}\) respectively. Similarly, the pullback of the connection \(\theta\) is denoted by \(\tilde{\theta}\), etc. Our aim is to modify \(\tilde{\theta}\) to \(\tilde{\theta}'\), so that \(\tilde{\theta}'\) is basic with respect to the horizontal lift of \(u_\xi\) to \(\tilde{P}\) with respect to \(\tilde{\theta}'\). To modify the connection form, we may add to it an arbitrary Ad-invariant horizontal 1-form on \(\tilde{P}\) with values in \(\mathfrak{u}(r)\). We thus define \[\tilde{\theta}':=\tilde{\theta} - A_\xi \tilde{\Phi},\] where \(A_\xi\) is pulled back from \(S\). This is a connection on \(\tilde{P}\); denote by \(\tilde{\mathcal{H}}'\) its horizontal subspace, and let \(\bar{u}_\xi\in\tilde{\mathcal{H}}\) (resp. \(\bar{u}_\xi'\in\tilde{\mathcal{H}}'\)) be the horizontal lift of \(u_\xi\) to \(\tilde{P}\) with respect to \(\tilde{\theta}\) (resp. \(\tilde{\theta}'\)).

\begin{lemma}\label{lem:basic}
	The form \(\tilde{\theta}'\) is basic with respect to \(\bar{u}_\xi'\).
\end{lemma}

\begin{proof} Let us compute the horizontal subspace \(\tilde{\mathcal{H}}'\) determined by \(\tilde{\theta}'\). The tangent space to \(S\) decomposes as \(TS = D\oplus (\mathbb{R}\cdot u_\xi)\), where \(D=\ker A = \ker A_\xi\). Denote by \(\bar{D}\subset T\tilde{P}\) the horizontal lift of \(D\) to \(\tilde{P}\) with respect to the connection \(\tilde{\theta}\). Note this is still horizontal for the connection \(\tilde{\theta}'\), since \(A_\xi\) vanishes on \(D\). Hence the horizontal subspaces \(\tilde{\mathcal{H}}\) and \(\tilde{\mathcal{H}}'\) differ only in the direction of \(u_\xi\). Since \(\bar{u}_\xi\) and \(\bar{u}_\xi'\) are both horizontal lifts of \(u_\xi\), \[\bar{u}_\xi' = \bar{u}_\xi+v,\] for some vertical vector field \(v\) on \(\tilde{P}\). Computing, \begin{align*}
		0 &= \tilde{\theta}'(\bar{u}_\xi') \\
		&= \tilde{\theta}(\bar{u}_\xi)+\tilde{\theta}(v)-A_\xi(\bar{u}_\xi)\tilde{\Phi}-A_\xi(v)\tilde{\Phi} \\
		&= \tilde{\theta}(v)-\tilde{\Phi}.
	\end{align*} Now, \(\Phi:P\to\mathfrak{u}(r)\) generates a tautological vertical vector field \(v_\Phi\) on \(P\) defined by letting \(v_\Phi(p)\) be the vertical tangent vector at \(p\) generated by the Lie algebra element \(\Phi(p)\). It is thus clear that the relation \(\tilde{\theta}(v)-\tilde{\Phi}=0\) implies \(v=\tilde{v}_\Phi\), so we have computed \(\bar{u}_\xi'=\bar{u}_\xi+\tilde{v}_\Phi\).
	
	By definition of the horizontal lift, \(\tilde{\theta}'(\bar{u}_\xi')=0\). To show \(\tilde{\theta}'\) is basic with respect to \(\bar{u}_\xi'\), it therefore remains to show that \(\iota_{\bar{u}_\xi'}d\tilde{\theta}'=0\). We compute \begin{align*}
		\iota_{\bar{u}_\xi'}d\tilde{\theta}' &= \iota_{\bar{u}_\xi'}(d\tilde{\theta}-dA_\xi \tilde{\Phi} +A_\xi\wedge d\tilde{\Phi}) \\
		&= \iota_{\bar{u}_\xi}d\tilde{\theta} + d\tilde{\theta}(\tilde{v}_\Phi,-) - dA_\xi(\bar{u}_\xi',-)\tilde{\Phi} +A_\xi(\bar{u}_\xi')d\tilde{\Phi} - A_\xi d\tilde{\Phi}(\bar{u}_\xi') \\
		&= \iota_{\bar{u}_\xi}d\tilde{\theta} + d\tilde{\theta}(\tilde{v}_\Phi,-)+d\tilde{\Phi} - A_\xi d\tilde{\Phi}(\bar{u}_\xi) - A_\xi d\tilde{\Phi}(\tilde{v}_\Phi).
	\end{align*} First note \(d\tilde{\Phi}(\bar{u}_\xi) = -d\tilde{\theta}(\bar{u}_\xi,\bar{u}_\xi) = 0\) by the moment map condition, since \(\bar{u}_\xi\) projects to \(\bar{\xi}\) under the map \(\tilde{P}\to P\). Next, we claim that \(d\tilde{\Phi}(\tilde{v}_\Phi)\) vanishes. To see this, note we may restrict the computation to a fibre of \(\tilde{P}\), in particular dropping the tilde notation. By choosing an arbitrary point of the fibre, we may identify the fibre (non-canonically) with \(U(r)\). The restricted moment map can be considered as \(\Phi:U(r)\to\mathfrak{u}(r)\), satisfying \(R_g^*\Phi = \mathrm{Ad}(g^{-1})\Phi\) for all \(g\in U(r)\), which implies \[\Phi(g) = \mathrm{Ad}(g^{-1})\Phi(e) = g^{-1}xg\] for all \(g\in U(r)\), where we define \(x:=\Phi(e)\), and \(e\) is the identity of \(U(r)\). To compute \(v_\Phi\), note that if we have an arbitrary element \(y\in\mathfrak{u}(r)\), the vector field at \(g\) generated by \(y\) via the right action is computed as \[\left.\frac{d}{dt}\right|_{t=0}g\exp(ty) = gy.\] Thus, the vector field generated by \(\Phi(g)\) at \(g\) is given by \[v_\Phi(g) = g\Phi(g) = xg.\] We now differentiate \(\Phi\) in the direction of \(v_\Phi\), to get \[v_\Phi(\Phi)(g) = \left.\frac{d}{dt}\right|_{t=0}\Phi(\exp(tx)g)= \left.\frac{d}{dt}\right|_{t=0}g^{-1}\exp(-tx)x\exp(tx)g=0.\]
	
	Returning to \(\tilde{P}\), we have now shown that \[\iota_{\bar{u}_\xi'}d\tilde{\theta}' = \iota_{\bar{u}_\xi}d\tilde{\theta} + d\tilde{\theta}(\tilde{v}_\Phi,-) + d\tilde{\Phi}.\] Note that \(\iota_{\bar{u}_\xi}d\tilde{\theta} + d_{\mathcal{H}}\tilde{\Phi}=0\) by the moment map condition, so we have \[\iota_{\bar{u}_\xi'}d\tilde{\theta}' = d\tilde{\theta}(\tilde{v}_\Phi,-) + d_{\mathcal{V}}\tilde{\Phi},\] and our aim is to show the right-hand side vanishes. The right-hand side is pulled back from \(P\), so we may compute there. Note that \(d\theta+\frac{1}{2}[\theta,\theta]=F\), where \(F\) is the curvature of the principal bundle \(P\). Isolating the vertical components of this equation, we have \(d\theta(v_\Phi,-) = -[\theta(v_\Phi),\theta]=-[\Phi,\theta]\). This is purely vertical, so we may again reduce to computing on a fibre, which we identify non-canonically with \(U(r)\). We now wish to show that \[d\Phi = [\Phi,\theta],\] where \(\theta\) is now the Maurer--Cartan form (for the right action) of \(U(r)\). Let \(a\in\mathfrak{u}(r)\) be an arbitrary element and denote by \(\bar{a}\) the right-invariant vector field on \(U(r)\) generated by \(a\), then we have \begin{align*}
		d\Phi(\bar{a})(g) &= \left.\frac{d}{dt}\right|_{t=0}\Phi(g\exp(ta)) \\
		&= \left.\frac{d}{dt}\right|_{t=0}\exp(-ta)g^{-1}xg\exp(ta) \\
		&= -ag^{-1}xg + g^{-1}xga \\
		&= [g^{-1}xg,a] \\
		&= [\Phi,\theta(\bar{a})]. \qedhere
	\end{align*}
\end{proof}

Now, let us compute the curvature of \(\tilde{\theta}'\) and take its horizontal trace, i.e. its trace over the distribution \(D\subset TS\). On \(\tilde{P}\), the curvature of \(\tilde{\theta}'\) is the restriction to \(\tilde{\mathcal{H}}'=\bar{D}\oplus(\mathbb{R}\cdot\bar{u}_\xi')\) of \[d\tilde{\theta}-dA_\xi \tilde{\Phi} +A_\xi\wedge d\tilde{\Phi}.\] Restricting this to \(\bar{D}\) yields \(d\tilde{\theta}|_{\bar{D}} - \frac{1}{f}\omega\tilde{\Phi}\). The first term is the curvature of \((\tilde{P},\tilde{\theta})\) in the \(D\)-directions. Computing the trace of this with respect to \(dA_\xi|_D = \frac{1}{f}\omega\) and translating back to vector bundles, we get \[f\Lambda_\omega F_h - n \Phi_h.\] Suppose this is a constant multiple of the identity, so that by Lemma \ref{lem:basic} we have a transverse Hermite--Einstein metric in the sense of \cite{BS10, BH22}. Then \[\frac{i}{2\pi}\left(f\Lambda_\omega F_h - n\Phi_h^\xi\right) = \lambda\Id_E,\] so multiplying by \(f^{-n-1}\), \[\frac{i}{2\pi}\left(f^{-n}\Lambda_\omega F_h - nf^{-n-1}\Phi_h^\xi\right) = \lambda f^{-n-1}\Id_E.\] Defining \(v(\mu) = f^{-n}\) and \(w(\mu)=f^{-n-1}\), this is the \((v,w)\)-weighted Hermite--Einstein equation.

\begin{prop}
	The metric \(h\) on \(E\to X\) is \((v,w)\)-weighted Hermite--Einstein with weights \(v(\mu) = (\langle \mu,\xi\rangle+a)^{-n}\) and \(w(\mu) = (\langle\mu,\xi\rangle+a)^{-n-1}\) if and only if the induced connection \(\tilde{\theta}\) on \(\tilde{E}\to S\) is transversally Hermite--Einstein with respect to the vector field \(u_\xi\).
\end{prop}

Suppose that \(\tilde{E}\to S\) admits a transversally Hermite--Einstein connection with respect to the vector field \(u_\xi\). This connection may not be induced by a connection on \(E\) in the above manner, so we cannot immediately conclude that \(E\) admits a weighted Hermite--Einstein metric. We will show in Proposition \ref{prop:transverse-converse} that it is indeed the case that the existence of a transversally Hermite--Einstein connection on \(\tilde{E}\to S\) implies the existence of a \((v,w)\)-weighted Hermite--Einstein metric on \(E\to X\).

\subsection{Polynomially weighted Hermite--Einstein metrics}\label{sec:poly-weights}

In this section, after reviewing the semi-simple fibration construction for a K\"ahler manifold $(X,\omega,\T)$, we use the associated bundle construction to construct from a $\T$-equivariant holomorphic vector bundle $E$ over $X$, a $\T$-equivariant holomorphic vector bundle $E_Q$ over the total space $X_Q$ of the semi-simple fibration. We show that under the associated connection construction, the weighted Hermite--Einstein equation on \(E\to X\) corresponds to the Hermite--Einstein equation on \(E_Q\to X_Q\).

Let $(X, \omega)$ be a compact K\"ahler manifold of dimension \(n\),  endowed with a Hamiltonian isometric action of an $\ell$-dimensional torus $\T$ and $\mu : X \to \Pol  \subset \tor^*$ a fixed moment map. We let $\pi_B : Q \to B$ be a principal $\T$-bundle over an $m$-dimensional manifold $B= B_1 \times \cdots \times B_N$, where for each $\alpha=1, \ldots, N$, $(B_\alpha, \omega_{B_\alpha})$ is a compact projective manifold of dimension \(n_\alpha\). We further assume that \(Q\) is the reduction to \(\T\) of a holomorphic principal bundle with structure group \(\T^{\mathbb{C}}\), and let \(\eta\in\Omega^1(Q,\mathfrak{t})\) be a principal bundle connection on \(Q\) with curvature
\begin{equation}\label{eq:cond-c1Q}
	d\eta = \sum_{\alpha=1}^N \pi_B^* \omega_{B_\alpha}\otimes p_\alpha, \qquad p_\alpha \in \Lambda \subset \tor,
\end{equation} where \(\Lambda\subset\mathfrak{t}\) denotes the integer lattice in \(\mathfrak{t}\); see \cite[Remark 5.1]{AJL} for a construction of such a principal bundle. 
Following \cite[Section 5]{AJL}, we consider the associated fibration
\[
X_Q:=Q\times_\T X\overset{\pi}{\to} B,
\]
where $Q\times_\T X:=(Q\times X)/\T$ with $\T$-action on the product given by $t\cdot(q, x)=(q\cdot t^{-1}, t\cdot x)$ for $x\in X$, $q\in Q$, and $t\in\T$. As per \cite{AJL}, this fibration is endowed with a natural integrable complex structure making \(X_Q\to B\) a holomorphic submersion.

Fixing constants $c_\alpha \in \R$ such
that for each $\alpha = 1,\cdots,N$, the affine linear function $\langle\mu,p_\alpha\rangle+c_\alpha>0$ is strictly positive on the momentum image $\Pol$ of $X$, we define the $2$-form $\omega_Q$ on $Q\times X$ by
\begin{equation}\label{Y-Kahler}
	\begin{split}
		\omega_Q  &:= \omega  + \sum_{\alpha=1}^N (\langle \mu, p_\alpha\rangle + c_\alpha)\pi_B^*\omega_{B_\alpha} +  \langle d\mu \wedge  \eta \rangle \\
		&= \omega + \sum_{\alpha=1}^N c_\alpha (\pi_B^*\omega_{B_\alpha}) + d\left(\langle \mu, \eta \rangle\right).
	\end{split}
\end{equation}
The form \(\omega_Q\) is $\T$-basic, so is the pull back of a K\"ahler form, still denoted $\omega_Q$, on the quotient $X_Q$. The K\"ahler manifold $(X_Q,\omega_Q)$ is called a \emph{semi-simple $(X, \T)$-fibration} associated to the K\"ahler manifold $(X, \T)$ and the torus bundle $Q\to B$ satisfying \eqref{eq:cond-c1Q}. Any K\"ahler metric $\omega_Q$ on $Y$ constructed in this way from a $\T$-invariant K\"ahler metric $\omega$ on $X$ and fixed K\"ahler metrics $\omega_\alpha$ on $B_\alpha$ is called \emph{bundle compatible}.

Now, let \(E\to X\) be a $\T$-equivariant holomorphic vector bundle. We consider the vector bundle 
\[
E_Q:=Q\times_\T E\to X_Q
\]
where $Q\times_\T E:=(Q\times E)/\T$ with $\T$-action on the product given by $t\cdot(q, e)=(q\cdot t^{-1}, t\cdot e)$ for $e\in E$, $q\in Q$, $t\in\T$. Using the  $\T$-bundle structure of $Q\times E\to E_Q$, the total space $E_Q$ inherits a complex structure $J_{E_Q}:=J_E\oplus \pi^*J_B$.

For any $\T$-invariant connection $\nabla$ on $E$ such that $\nabla^{0,1}$ induces the holomorphic structure $J_E$ of $E$, we write $\mathcal{H}^\nabla\subset TE$ the corresponding horizontal distribution inducing the splitting $TE=\mathcal{H}^\nabla\oplus VE$. By the Koszul--Malgrange theorem, $J_E$ is given by $\pi^*J_X\oplus J_{\mathbb{C}^r}$ where $\pi^*J_X$ is the complex structure on $\mathcal{H}^\nabla$ induced from the base and $J_{\mathbb{C}^r}$ is the standard complex structure on the fibers of $E\to X$. The relation $J_{E_Q}=J_{X_Q}\oplus J_{\mathbb{C}^r}$, shows that $J_{E_Q}$ is compatible with the vector bundle structure $E_Q\to X_Q$. 

Let $\Phi_\nabla$ be $\T$-momentum map for the curvature $F_\nabla$ of $\nabla$. We define the $\T$-invariant connection on the vector bundle $\pi_E\times\Id_Q:Q\times E\to Q\times X$ by
\begin{equation}\label{eq:nabla-bundle}
	\nabla^Q:=\nabla+\langle\Phi_\nabla\otimes \eta\rangle.
\end{equation}
In a basis of $\mathbb{S}^1$-generators $(\xi_a)$ of $\tor$, writing $\eta=\sum_a\eta_a\otimes\xi^Q_a$ where the 1-forms $\eta_a$ on $Q$ are such that $\eta_a(\xi_b^Q)=\delta_{ab}$, then
\[
\langle\Phi_\nabla\otimes \eta\rangle=\sum_{a=1}^\ell\Phi_\nabla^{\xi_a}\otimes\eta_a.
\]

To show that $\nabla^Q$ is the pullback of a connection (also denoted $\nabla^Q$) on $E_Q\to X_Q$ we need to verify the $\nabla^Q$ is $\T$-basic. First, it is straightforward to observe that \(\nabla^Q\) is \(\T\)-invariant; note this implies the corresponding connection \(\theta^Q\) on the frame bundle of \(E_Q\) satisfies \(\theta^Q(\xi)=0\) for all \(\xi\in\mathfrak{t}\), since the kernel of \(\theta^Q\) consists of parallel transported frames, and \(\nabla^Q\) being \(\T\)-invariant implies the infinitesimal \(\T\)-action on a frame induces parallel transport. To prove that \(d\theta^Q(\xi,-)=0\) for all \(\xi\in\tor\), we note the curvature of \(\theta^Q\) is \(d\theta^Q+\frac{1}{2}[\theta^Q,\theta^Q]\), and since \(\theta^Q(\xi)=0\) it suffices to show that the contraction of the curvature with any \(\xi\in\mathfrak{t}\) vanishes. Since the curvature on the frame bundle of \(E_Q\) descends to \(X_Q\), we instead compute the curvature of \(\nabla^Q\) as
\begin{equation}
	\begin{split}\label{eq:F_nablaQ}
		F_{\nabla^Q}=&F_\nabla+ \nabla^{\End(E)}( \langle\Phi_\nabla\otimes \eta\rangle) + \langle\Phi_\nabla\otimes \eta\rangle \wedge \langle\Phi_\nabla\otimes \eta\rangle\\
		=& F_\nabla+ \langle \nabla^{\End(E)}\Phi_\nabla\wedge\eta\rangle+ \langle\Phi_\nabla\otimes d\eta\rangle + \langle\Phi_\nabla\otimes \eta\rangle \wedge \langle\Phi_\nabla\otimes \eta\rangle\\
		=&F_\nabla+ \langle \nabla^{\End(E)}\Phi_\nabla\wedge \eta\rangle+ 
		\sum_{\alpha=1}^N(\pi_B^*\omega_{B_\alpha})\langle\Phi_\nabla,p_\alpha\rangle + \langle\Phi_\nabla\otimes \eta\rangle \wedge \langle\Phi_\nabla\otimes \eta\rangle \\
		=& F_\nabla+ 
		\sum_{\alpha=1}^N(\pi_B^*\omega_{B_\alpha})\langle\Phi_\nabla,p_\alpha\rangle+\sum_{a=1}^\ell \nabla^{\End(E)}\Phi^{\xi_a}_\nabla\wedge \eta_a + \sum_{a,b=1}^\ell\Phi_\nabla^{\xi_a}\Phi_\nabla^{\xi_b}\otimes\eta_a\wedge\eta_b.
	\end{split}    
\end{equation}
For $\xi_a\in\tor$, we have the vector field $\xi_a^X-\xi_a^Q$ on \(Q\times X\) and 
\[
\begin{split}
	&F_{\nabla^Q}(\xi_a^X-\xi_a^Q,-)\\
	=& F_\nabla(\xi_a^X-\xi_a^Q,-)+ \sum_{b} (\nabla^{\End(E)}\Phi_\nabla^{\xi_b}\wedge \eta_b)(\xi_a^X-\xi_a^Q,-)+\langle\Phi_\nabla\otimes\eta\rangle\wedge\langle\Phi_\nabla\otimes\eta\rangle(\xi_a^X-\xi_a^Q,-) \\
	=&F_\nabla(\xi_a^X,-)+ \sum_{b} (\nabla^{\End(E)}\Phi_\nabla^{\xi_b})(\xi_a^X)\otimes\eta_b+ \sum_{b} \delta_{ab}\nabla^{\End(E)}\Phi_\nabla^{\xi_b}-\sum_{b,c}\Phi_\nabla^{\xi_b}\Phi_\nabla^{\xi_c}\otimes(\delta_{ab}\eta_c-\delta_{ac}\eta_b)\\
	=&\sum_{b} F_{\nabla}(\xi_a^X,\xi_b^X)\otimes\eta_b -\sum_{b}^\ell[\Phi_\nabla^{\xi_a},\Phi_\nabla^{\xi_b}]\otimes\eta_b \\
	=& \,0,
\end{split}
\]
where in the final line we use \[[\Phi_\nabla^{\xi_a},\Phi_\nabla^{\xi_b}] = [\nabla_{\xi_a}-\mathcal{L}_{\xi_a},\nabla_{\xi_b}-\mathcal{L}_{\xi_b}] = [\nabla_{\xi_a},\nabla_{\xi_b}] = F_\nabla(\xi_a,\xi_b).\]  

It follows that the connection \(\nabla^Q\) is basic. We call the induced connection $\nabla^Q$ on $E_Q\to X_Q$ a \emph{bundle compatible connection}.

\begin{prop}\label{prop:polynomial}
	The connection \(\nabla\) on \(E\to (X,\omega)\) is a \((p,p)\)-weighted Hermite--Einstein metric, where
	\[p(\mu):=\prod_{\alpha=1}^{N}(c_\alpha+\langle\mu, p_\alpha\rangle)^{n_\alpha}, \] if and only if the connection \(\nabla^Q\) on \(E_Q\to (X_Q,\omega_Q)\) is Hermite--Einstein.
\end{prop}

\begin{proof}
	First, we compute the mean curvature of $(E,\nabla)\to(X,\omega)$ in in terms of the $p$-weighted mean curvature of $(E,\nabla)\to(X_Q,\omega_Q)$. Using \eqref{Y-Kahler}, one can easily check that the forms $\omega_Q^{[m+n]}$ and $\omega_Q^{[m+n-1]}$ on $X_Q$ are the push-forwards of the following basic forms on $Q\times X$ (see \cite[Section 5]{AJL}):
	\[
	\begin{split}
		\omega_Q^{[m+n]}\wedge\eta^{\wedge \ell} &= \left( p(\mu) \omega^{[n]}\wedge  \bigwedge_{\alpha=1}^{N} \pi^{*}_B\omega_\alpha^{[n_\alpha]}\right) \wedge\eta^{\wedge \ell},\\
		\omega_Q^{[m+n-1]}\wedge\eta^{\wedge \ell}=& \sum_{\alpha=1}^{N}\left(\frac{p(\mu)}{(\langle \mu,p_\alpha\rangle + c_\alpha)}\omega^{[n]}\wedge(\pi^{*}_B\omega_\alpha)^{[n_\alpha-1]}\wedge \bigwedge_{\substack{\beta=1\\\beta\neq \alpha}}^{N} (\pi^{*}_B\omega_\beta)^{[n_\beta]}\wedge\eta^{\wedge \ell}\right)\\
		&+ p(\mu) \omega^{[n-1]}\wedge \bigwedge_{\alpha=1}^{N} (\pi^{*}_B\omega_\alpha)^{[n_\alpha]}\wedge\eta^{\wedge \ell}.
	\end{split}
	\]
	where $\eta^{\wedge \ell}:=\eta_1\wedge\cdots\wedge\eta_\ell$. The mean curvature of $(E_Q,\nabla^Q)\to(X_Q,\omega_Q)$ is given by
	\[
	\begin{split}
		\Lambda_{\omega_Q} F_{\nabla^Q}=&\frac{i}{2\pi}\frac{F_{\nabla^Q}\wedge \omega_Q^{[n+m-1]}}{\omega_{Q}^{[n+m]}}\quad \text{(on $X_Q$)}\\
		=&\frac{F_{\nabla^Q}\wedge \omega_Q^{[n+m-1]}\wedge \eta^{\wedge\ell}}{\omega_{Q}^{[n+m]}\wedge \eta^{\wedge\ell}}\quad \text{(on $Q\times X$)}.
	\end{split}
	\]
	Using \eqref{eq:F_nablaQ}, the numerator of the above ratio satisfies:
	\[
	\begin{split}
		&F_{\nabla^Q}\wedge \omega_Q^{[n+m-1]}\wedge \eta^{\wedge\ell}\\
		=&\left(\sum_{\alpha=1}^N(\pi_B^*\omega_{B_\alpha})\langle\Phi_\nabla,p_\alpha\rangle\right) \wedge \sum_{\beta=1}^{N}\left(\frac{p(\mu)}{(\langle \mu,p_\beta\rangle + c_\beta)}\omega^{[n]}\wedge(\pi^{*}_B\omega_\beta)^{[n_\beta-1]}\wedge \bigwedge_{\substack{\gamma=1\\\gamma\neq \beta}}^{N} (\pi^{*}_B\omega_\gamma)^{[n_\gamma]}\wedge\eta^{\wedge\ell}\right)\\
		&+ p(\mu) F_\nabla\wedge\omega^{[n-1]}\wedge \bigwedge_{\alpha=1}^{N} (\pi^{*}_B\omega_\alpha)^{[n_\alpha]}\wedge\eta^{\wedge\ell}\\
		=& \left( \Lambda_\omega F_\nabla+\sum_{\alpha=1}^N  
		\frac{n_\alpha\langle\Phi_\nabla,p_\alpha\rangle}{(\langle \mu,p_\alpha\rangle + c_\alpha)}
		\right)p(\mu)\omega^{[n]}\wedge \bigwedge_{\alpha=1}^{N} (\pi^{*}_B\omega_\alpha)^{[n_\alpha]}\wedge\eta^{\wedge\ell}\\
		=&\left( \Lambda_\omega F_\nabla+\langle\Phi_\nabla,d(\log p)(\mu)\rangle
		\right)\omega^{[n+m]}_Q\wedge\eta^{\wedge\ell}\\
		=&\frac{\Lambda_{\omega,p}(F_\nabla+\Phi_\nabla)}{p(\mu)}\omega^{[n+m]}_Q\wedge\eta^{\wedge\ell},
	\end{split}
	\]
	where we use that $d(\log p)(\mu)=\sum_{\alpha=1}^N  
	\frac{n_\alpha p_\alpha}{\langle \mu,p_\alpha\rangle + c_\alpha}$. Substituting back we obtain
	\begin{equation}\label{eq:p-MeanCurv}
		\Lambda_{\omega_Q} F_{\nabla^Q}=\frac{\Lambda_{\omega,p}(F_\nabla+\Phi_\nabla)}{p(\mu)}.    
	\end{equation}
	Hence, $(E_Q,\nabla^Q)\to(X_Q,\omega_Q)$ is Hermite--Einstein if and only if $(E,\nabla)\to(X,\omega)$ is $(p,p)$-weighted Hermite--Einstein. 
\end{proof} 

Thus, if \(E\) admits a \((p,p)\)-weighted Hermite--Einstein metric, so too does \(E_Q\). However, similarly to the Sasaki case, the result does not imply that if \(E_Q\) admits a Hermite--Einstein metric then \(E\) admits a \((p,p)\)-weighted Hermite--Einstein metric, namely because the Hermite--Einstein metric on \(E\) may not be of the form \(\nabla^Q\) for some connection \(\nabla\) on \(E\). We will show in Proposition \ref{prop:semi-simple-fibration} that this implication is in fact true, using the weighted Kobayashi--Hitchin correspondence.

\subsection{Deformations}

Following \cite[Section 6.1.1]{HalThesis} in the weighted cscK case, one may consider deformations of the weight function in the weighted Hermite--Einstein setting. In particular, we have:

\begin{prop}\label{prop:deformation}
	Suppose that a \(E\) is simple and admits a \(v\)-weighted Hermite--Einstein metric. Let \(v_t:P\to\mathbb{R}_{>0}\) be a 1-parameter family of deformations of \(v\), so that \(t\in(-\epsilon,\epsilon)\) for some \(\epsilon>0\), the \(v_t\) vary smoothly with \(t\), and \(v_0=v\). Then for all \(t\) sufficiently close to \(0\), \(E\) admits a \(v_t\)-weighted Hermite--Einstein metric.
\end{prop}

The proof generalises to higher-dimensional families of deformations, but we state the one-dimensional case for simplicity.

\begin{proof}
	We use freely the notations and conventions introduced in Section \ref{sec:DUY}; the reader may wish to consult these before reading this proof. Taking the \(v\)-weighted Hermite--Einstein metric \(h_0\) as a reference metric, any other \(\T\)-invariant Hermitian metric may be identified with an element \(f\in \Herm^+(E,h_0)^{\T}\), the space of \(\T\)-invariant positive self-adjoint endomorphisms with respect to \(h_0\). Consider the vector bundle \(\mathcal{V}\to\Herm^+(E,h_0)^\T\) whose fibre over \(f\in\Herm^+(E,h_0)^\T\) is \[\mathcal{V}_f:=\Herm(E,h_0f)^\T,\] i.e. the \(\T\)-invariant endomorphisms of \(E\) that are Hermitian with respect to the metric \((h_0f)(e_1,e_2):=h_0(fe_1,e_2)\).\footnote{The bundle \(\mathcal{V}\) is of course the tangent bundle of \(\Herm^+(E,h_0)^\T\).} The section \[L_t:\Herm^+(E,h_0)^{\T}\to \mathcal{V}\] is defined by \[L_t(f):=K^0_{v_t}+\frac{i}{2\pi}\Lambda_{\omega,v_t}(\db^{\eq}(f^{-1}\d_0(f))),\] where \(K^0_{v_t}:=\frac{i}{2\pi}\Lambda_{\omega,v_t}(F_{h_0}+\Phi_{h_0}) - c_{v_t}\Id_E\), and \(f\) corresponds to a \(v_t\)-weighted Hermite--Einstein metric if and only if \(L_t(f)\) is a constant multiple of the identity.
	
	The bundle \(\mathcal{V}\) has a trivialisation \(\mathcal{V}\cong \Herm^+(E,h_0)^\T\times \Herm(E,h_0)^\T\), given by \[\mathcal{V}_f\ni e\mapsto fe \in\Herm(E,h_0)^\T.\] Under this trivialisation, the section \(L_0\) is identified with the map \[\hat{L}_0:\Herm^+(E,h_0)^\T\to\Herm(E,h_0)^\T,\] given by \[\hat{L}_0(f) = fL_0(f).\] One readily computes the linearisation of \(\hat{L}_0\) at the identity endomorphism of \(E\) is given by the weighted \(\d_0\)-Laplacian of \(\End E\), where \(\d_0\) denotes the \((1,0)\)-part of the Chern connection of \(h_0\). In particular, this is an elliptic operator with index \(0\), whose kernel on Hermitian endomorphisms consists of the constant multiples of \(\Id_E\), by simplicity of \(E\).
	
	By the calculation of Lemma \ref{det1}, if \(f\in\Herm^+(E,h_0)^\T\) satisfies \(\det(f) = 1\), then \(\tr L_0(f) = 0\). Denoting by \(\Herm_1^+(E, h_0)^\T\) the set of \(f\in \Herm^+(E,h_0)^\T\) satisfying \(\det(f)=1\), and by \(\mathcal{V}^0\subset\mathcal{V}\) the vector subbundle of trace-free endomorphisms, we see that \(L_0\) restricts to a section \[L_0:\Herm_1^+(E, h_0)^\T\to\mathcal{V}^0.\] For general \(t\), let \(p_0:\mathcal{V}\to\mathcal{V}_0\) denote the projection to the trace-free endomorphisms, and consider \[p_0\circ L_t:\Herm_1^+(E, h_0)^\T\to\mathcal{V}^0.\] We remarked the linearisation of \(L_0\) at the identity is the weighted \(\d_0\)-Laplacian of \(\End E\). The tangent space to \(\{\det(f)=1\}\) at the identity is the trace-free endomorphisms, on which the weighted Laplacian is an isomorphism, by simplicity of \(E\). Since \(L_0= p_0\circ L_0\), the linearisation of \(p_0\circ L_0\) is also an isomorphism at the identity. Taking the appropriate Banach approximations of these vector bundles, the Banach space implicit function theorem implies that for all \(t>0\) sufficiently small, there exists a unique smooth solution to the equation \(p_0(L_t(f)) = 0\). Since \(p_0\) merely subtracts the trace-part of an endomorphism, \(L_t(f)\) is a constant multiple of the identity endomorphism, and hence \(h_0f\) is \(v_t\)-weighted Hermite--Einstein.
\end{proof}

Using this deformation result, we can easily produce interesting examples of weighted Hermite--Einstein metrics. For example, suppose \(E\) admits a Hermite--Einstein metric, which is automatically \(\T\)-invariant by uniqueness of Hermite--Einstein metrics. Suppose further that \(\omega\in c_1(L)\) lies in the first Chern class of an ample line bundle \(L\). Deforming the weight function \(v(\mu):=1\) to \(v_t(\mu) :=(1+t\langle\mu,\xi\rangle)^{-n}\) for arbitrary \(\xi\in\tor\), we see that the circle bundle \(S\) of \(L^*\) admits a transverse Hermite--Einstein metric with respect to any sufficiently small vector field \(\xi\in\tor\).

Another example is the deformation \(v_t(\mu):=e^{t\langle\mu,\xi\rangle}\), which deforms Hermite--Einstein metrics to exponentially weighted Hermite--Einstein metrics.

\section{Equivariant intersections}\label{sec:eq-intersections}

Before developing more advanced theory of the weighted Hermite--Einstein equation, we first review the equivariant calculus introduced by Inoue \cite{Ino20}, and extend the theory from weight functions of the form \(\tilde{v}(\langle \mu,\xi\rangle)\) for some fixed \(\xi\in\mathfrak{t}\) and real analytic \(\tilde{v}\), to arbitrary smooth weight functions \(v(\mu)\) on the moment polytope. There are two contexts in which we will use the theory: \begin{enumerate}
	\item Given a \(\T\)-equivariant manifold \(X\), we may produce \emph{equivariant intersection numbers} from equivariant data on \(X\);
	\item Let \(K\) be an auxiliary compact Lie group. Given a \((\T\times K)\)-equivariant holomorphic submersion \(X\to B\), where \(\T\) acts trivially on \(B\), we may produce closed \(K\)-equivariant differential forms on \(B\) of degree 2 from equivariant data on \(X\), in a way that is compatible with the projection of such forms to equivariant cohomology. We will call such forms \emph{equivariant intersection 2-forms}.
\end{enumerate}

The equivariant intersection numbers will be used to show that quantities such as \(\int_X v(\mu)\omega^n\) and \(\int_X\tr(\Lambda_{\omega, v}(F_h+\Phi_h))\omega^n\) are independent of the choices of equivariant representative \(\omega+\mu\) for \([\omega+\mu]\) and \(\T\)-invariant metric \(h\). In particular, we can use the theory to define the \emph{weighted slope} of a sheaf, which will in turn be used to define weighted slope stability of vector bundles. We will further use such equivariant intersection numbers to phrase the correct analogue of the Kobayashi--L{\"u}bke inequality in the weighted setting. The equivariant intersection 2-forms will be used to demonstrate that the weighted Hermite--Einstein equation arises naturally as a moment map.

\subsection{Equivariant intersection numbers}

\subsubsection{Weight functions of one variable}

We first recall the theory of Inoue which produces equivariant intersection numbers. Let \(f:\mathbb{R}\to\mathbb{R}\) be a real analytic function; denote by \(\sum_{k=0}^\infty\frac{a_k}{k!}x^k\) its convergent power series centred at the origin of \(\mathbb{R}\). Let \(X\) be a compact complex manifold of dimension \(n\), equipped with a \(\T\)-action. Given an equivariant cohomology class \(\alpha_{\T}\) on \(X\), along with an auxiliary class \(\beta_{\T}\in H^{2k}_{\T}(X;\mathbb{R})\), we produce a convergent power series \[(\beta_{\T}\cdot f(\alpha_{\T}))\in\mathbb{R}\llbracket\mathfrak{t}\rrbracket\] as follows. The map \(\pi:X\to\mathrm{pt}\) induces a map \(\pi_*:H^*_{\T}(X;\mathbb{R})\to H^{*-2n}_{\T}(\pt)\); on the level of Cartan representatives, this is simply the integration map. There is an isomorphism \[\hat{H}^{*,\mathrm{even}}_{\T}(\pt;\mathbb{R}) \cong \mathbb{R}\llbracket \mathfrak{t}\rrbracket,\] where \(\hat{H}^{*,\mathrm{even}}_{\T}(\pt;\mathbb{R})\) denotes the completion of the even-degree equivariant cohomology ring of a point, and a degree \(2j\) class in this ring is mapped to a homogeneous polynomial of degree \(j\) on \(\mathfrak{t}\) under the isomorphism.

Given this information, we define \((\beta_{\T}\cdot f(\alpha_{\T}))\) by first taking the class \[\beta_{\T}\cdot f(\alpha_{\T}):=\sum_{k=0}^\infty\frac{a_k}{k!}\beta_{\T}\cdot\alpha_{\T}^k \in \hat{H}^{*,\mathrm{even}}_{\T}(X;\mathbb{R})\] and then applying the pushforward map \(\pi_*:H^*_{\T}(X;\mathbb{R})\to H^{*-2n}_{\T}(\pt)\) term-by-term in the series to get \[(\beta_{\T}\cdot f(\alpha_{\T})) := \pi_*(\beta_{\T}\cdot f(\alpha_{\T}))\in \hat{H}^{*,\mathrm{even}}_{\T}(\pt;\mathbb{R}) \cong \mathbb{R}\llbracket \mathfrak{t}\rrbracket. \]

Now, by \cite[Theorem 1.1]{Ino20}, if \(f\) has a convergent power series on all of \(\mathbb{R}\), then the formal power series \((\beta_{\T}\cdot f(\alpha_{\T}))\) is in fact convergent on \(\mathfrak{t}\). In particular, given an element \(\xi\in\mathfrak{t}\), we may evaluate the power series at \(\xi\) to produce an equivariant intersection number \[(\beta_{\T}\cdot f(\alpha_{\T}))(\xi)\in\mathbb{R}.\]

The connection to the weighted theory is as follows. Suppose that the weight function \(v:P\to\mathbb{R}\) is of the form \(v(p) = \tilde{v}(\langle p, \xi\rangle)\) for some fixed element \(\xi\in\mathfrak{t}\) and a fixed real analytic function \(\tilde{v}\) on \(\mathbb{R}\). Let \(f\) be an order \(n\) antiderivative for \(\tilde{v}\), so that \(f^{(n)}=\tilde{v}\). Let \(\alpha_\T:=[\omega+\mu]\). Then taking \(\beta_\T = 1\), \[(f(\alpha_{\T}))(\xi) = \int_X v(\mu)\omega^{[n]}.\] Here we use the notation \(\omega^{[j]}:=\frac{1}{j!}\omega^j\). In particular, since the left-hand side is independent of the choice of equivariant representative for \(\alpha_\T\), this implies that the right-hand side \(\int_Xv(\mu)\omega^n\) depends only on the equivariant cohomology class \([\omega+\mu]\), and not the choice of equivariant representative. Taking \(\beta_{\T} = c_1(E)_{\T}\) gives \[(c_1(E)_{\T} f'(\alpha_{\T}))(\xi) = \int_X \tr\left(\frac{i}{2\pi}\Lambda_{\omega,v}(F_h+\Phi_h)\right)\omega^{[n]},\] again showing that the right-hand side depends only on the equivariant cohomology classes \(c_1(E)_{\T}\) and \(\alpha_{\T}\); note here we take a derivative of the function \(f\), to account for the class \(c_1(E)_\T\) having degree 2.

\begin{rem}\label{rem:complex-extension}
	Since the power series \((f(\alpha_\T))\) converges on all of \(\tor\), we may extend it to an entire function on the complex vector space \(\mathfrak{t}\otimes\mathbb{C}\). In particular \((\beta_{\T}\cdot f(\alpha_{\T}))(i\xi)\) for \(\xi\in\tor\) is well-defined. In the special case \(f = \exp\), the function \(\xi\mapsto (\beta_{\T}\cdot e^{\alpha_{\T}})(i\xi)\) satisfies \begin{equation}\label{eq:bound-on-intersection}
		|(\beta_{\T}\cdot e^{\alpha_{\T}})(i\xi)| \leq C(1+|\xi|)^k
	\end{equation} for some fixed \(C>0\), where \(\beta_{\T}\) has degree \(2k\) as an equivariant form. This is easily seen from the fact that \(\beta_{\T}\) is polynomial in \(\xi\) of degree \(k\), and that \(e^{i\langle\mu,\xi\rangle}\) has modulus 1.
\end{rem}

\subsubsection{Arbitrary weight functions}

We now show how the above theory may be extended to arbitrary weight functions \(v:P\to\mathbb{R}\) on the moment polytope, that are not necessarily of the form \(v(p) = \tilde{v}(\langle p,\xi\rangle)\). The key is to use Fourier transforms to reduce to the equivariant intersection numbers of Inoue.

First, given an arbitrary function \(u:P\to\mathbb{R}\), we may extend \(u\) to a compactly supported smooth function on all of \(\tor^*\), with support containing \(\Pol\). Recall, that the Fourier transform of $u\in C^\infty_c(\tor^*)$ is given by \[\hat{u}(\xi)=\int_{\tor^*} u(\mu)e^{-i\langle\mu,\xi\rangle}\underline{d\mu},\] where \(\underline{d\mu}\) denotes the Lebesgue measure on \(\tor^*\), for a fixed choice of metric on \(\mathfrak{t}^*\) that will not matter. The Fourier transform defines a bijective map $\mathcal{F}:S(\tor^*)\to S(\tor)$ on the Schwartz class of rapidly decreasing functions, and the Fourier inversion formula yields \[u(\mu)=\int_\tor \hat{u}(\xi)e^{i\langle\mu,\xi\rangle}\underline{d\xi},\] where \(\underline{d\xi}\) denotes the Lebesgue measure on \(\tor\) (dual to \(\underline{d\mu}\)) divided by \((2\pi)^{\dim\T}\).

\begin{defn}
	Let \(\alpha_{\T}\in H^2_\T(X;\mathbb{R})\) be an equivariant K{\"a}hler class with moment polytope \(P\), and let \(\beta_{\T}\in H^{2k}_\T(X;\mathbb{R})\) be an auxiliary equivariant class. Let \(u:P\to\mathbb{R}\) be an arbitrary smooth function on the moment polytope, and pick a smooth compactly supported extension of \(u\) to all of \(\tor^*\). We define \[(\beta_{\T}\cdot u(\alpha_{\T})) := \int_{\tor}(\beta_{\T}\cdot e^{\alpha_{\T}})(i\xi)\hat{u}(\xi)\underline{d\xi},\] where \((\beta_{\T}\cdot e^{\alpha_{\T}})(i\xi)\) was defined in Remark \ref{rem:complex-extension}.
\end{defn}

The quantity \((\beta_{\T}\cdot u(\alpha_{\T}))\) manifestly depends only on the equivariant cohomology classes involved, and not on any choice of equivariant representative.

\begin{lemma}
	Let \(\alpha_\T\in H^2_\T(X;\mathbb{R})\) be an equivariant K{\"a}hler class with moment polytope \(P\), \(u:P\to\mathbb{R}\) a smooth function, and \(\beta_\T\in H^{2k}_\T(X;\mathbb{R})\). The number \((\beta_{\T}\cdot u(\alpha_{\T}))\) is well-defined in the sense that the integral converges and is independent of the choice of extension of \(u\) from \(\Pol\) to \(\tor^*\). In the case where \(u(p) = \tilde{u}(\langle p,\zeta\rangle)\) for a real analytic function \(\tilde{u}\) and a fixed \(\zeta\in\mathfrak{t}\), \[(\beta_{\T}\cdot u(\alpha_{\T})) = (\beta_{\T}\cdot f^{(k)}(\alpha_\T))(\zeta),\] where \(f\) satisfies \(f^{(n)} = \tilde{u}\) and \((\beta_{\T}\cdot f^{(k)}(\alpha_\T))(\zeta)\) is defined as in the previous section.
\end{lemma}

\begin{proof}
	
	That the integral converges can be seen easily from \eqref{eq:bound-on-intersection}, together with the fact that \(\hat{u}\) is a Schwartz class function.
	
	To see that the quantity does not depend on the extension of \(u\) from \(\Pol\) to \(\tor^*\), we choose an equivariant representative \(\sigma\) for \(\beta_{\T}\) and compute \begin{align}
		\int_{\tor}(\beta_{\T}\cdot e^{\alpha_{\T}})(i\xi)\hat{u}(\xi)\underline{d\xi}
		&= \int_{\tor} \int_X \langle\sigma,i\xi\rangle\wedge e^{\langle \omega+\mu, i\xi\rangle}\hat{u}(\xi)\underline{d\xi} \nonumber \\
		&= \int_X \int_{\tor} \langle \sigma, i\xi\rangle e^{i\langle\mu,\xi\rangle}\hat{u}(\xi)\underline{d\xi}\wedge e^\omega. \label{last-line}
	\end{align} Next, note that \(\sigma:\mathfrak{t}\to\Omega^*(X)\) is an equivariant differential form of degree \(2k\), so \(\langle\sigma,\xi\rangle\) is a sum of terms of the form \(\eta\cdot\xi_1^{a_1}\cdots \xi_{\ell}^{a_\ell}\) where \(\eta\in \Omega^{2j}(X)\) and \(a_1+\cdots+a_\ell = k-j\), for some \(0\leq j\leq k\). Replacing \(\sigma\) in \eqref{last-line} with such a term yields \begin{align*}
		&\int_X \int_{\tor} \eta\cdot i^{k-j}\xi_1^{a_1}\cdots\xi_\ell^{a_\ell} e^{i\langle\mu,\xi\rangle}\hat{u}(\xi)\underline{d\xi}\wedge e^\omega \\
		=&\int_X\left( \int_{\tor} e^{i\langle\mu,\xi\rangle}\mathcal{F}(\d_1^{a_1}\cdots\d_\ell^{a_\ell}u)(\xi)\underline{d\xi}\right) \eta\wedge e^\omega \\
		=& \int_X (\d_1^{a_1}\cdots\d_\ell^{a_\ell}u)(\mu) \eta\wedge e^\omega
	\end{align*} and this final line is clearly independent of the choice of extension of \(u\) from \(\Pol\) to \(\tor^*\).
	
	Finally, we show that the two definitions agree. We may assume \(\zeta \neq 0\), as the case \(\zeta = 0\) is easily observed. With this assumption, we may choose a basis \(\xi_1,\ldots,\xi_\ell\) for \(\mathfrak{t}\) so that \(\zeta = \xi_1\), and take the Lebesgue measure on \(\mathfrak{t}\) determined by this basis. Once again choosing an equivariant representative \(\sigma\) for \(\beta_{\T}\) and isolating a particular term \(\eta\cdot\xi_1^{a_1}\cdots\xi_\ell^{a_\ell}\), note we have \[\langle\eta\cdot\xi_1^{a_1}\cdots\xi_\ell^{a_\ell},\zeta\rangle = \begin{cases}
		\eta & \text{if } a_l = 0 \text{ for all } l > 1 \\
		0 & \text{otherwise.}
	\end{cases}\] Considering only the case \(\eta\cdot\xi_1^{k-j}\), \begin{align*}
		\int_X \langle\eta\cdot\xi_1^{k-j},\zeta\rangle f^{(k)}(\omega+\langle\mu,\zeta\rangle)
		&= \int_X \eta f^{(k)}(\omega+\langle\mu,\zeta\rangle) \\
		&= \int_X f^{(n+k-j)}(\langle\mu,\zeta\rangle)\eta\wedge\omega^{[n-j]}.
	\end{align*} Note that similarly to the above calculation of \((\beta_{\T}\cdot u(\alpha_{\T}))\), in the case \(u(\mu) = \tilde{u}(\langle\mu,\zeta\rangle)\) we have \(\d_1^{a_1}\cdots\d_\ell^{a_\ell}u = 0\) unless \(a_l=0\) for all \(l>0\), and \(\d_1^{k-j}u(\mu) = f^{(n+k-j)}(\langle\mu,\zeta\rangle)\). It follows that the two terms agree, and we are done.
\end{proof}

\begin{lemma}\label{constants-in-wHE}
	The following formulae hold:
	\begin{enumerate}
		\item \(\int_X v(\mu)\omega^{[n]} = (v(\alpha_\T))\);
		\item \(\int_X\tr(\frac{i}{2\pi}\Lambda_{\omega,v}(F_h+\Phi_h))\omega^{[n]} = (c_1(E)_\T\cdot v(\alpha_\T))\).
	\end{enumerate} In particular, the constant \(c_v\) of equation \eqref{eq:v-weighted-HE} is an equivariant topological constant.
\end{lemma}

\begin{proof}
	In the case \(\beta_{\T} = 1\), the Fourier inversion formula gives \begin{align*}
		(v(\alpha_{\T})) &=  \int_{\tor} (e^{\alpha_{\T}})(i\xi)\hat{v}(\xi)\underline{d\xi} \\
		&= \int_{\tor}\left(\int_X e^{i\langle \mu,\xi\rangle}\omega^{[n]}\right)\hat{v}(\xi)\underline{d\xi} \\
		&= \int_X\int_{\tor}\hat{v}(\xi)e^{i\langle \mu,\xi\rangle} \underline{d\xi} \omega^{[n]} \\
		&= \int_X v(\mu)\omega^{[n]}.
	\end{align*}
	
	To compute the case \(\beta_{\T} = c_1(E)_{\T}\), and later the case \(\beta_{\T} = c_2(E)_\T\), we will need the following straightforward formulae:
	\begin{align}
		\begin{split} \label{eq:FourierD}
			&(\partial_kv)(\mu)=i\int_{\tor} \hat{v}(\xi) \xi_k e^{i\langle\xi,\mu\rangle}\underline{d\xi}\implies
			dv(\mu)=i\int_{\tor} \hat{v}(\xi) \xi e^{i\langle\xi,\mu\rangle}\underline{d\xi}\\
			&(\partial_{jk}^2 v)(\mu)=i^2\int_{\tor} \hat{v}(\xi) \xi_j \xi_k e^{i\langle\xi,\mu\rangle}\underline{d\xi}\implies {\rm Hess}(v)(\mu)= i^2\int_{\tor} \hat{v}(\xi) (\xi\otimes\xi) e^{i\langle\xi,\mu\rangle}\underline{d\xi}.
		\end{split}
	\end{align}
	In the case \(\beta_{\T} = c_1(E)_{\T}\), we get \begin{align*}
		\big(c_1(E)_\T\cdot v(\alpha_\T)\big) &= \int_\tor \big(c_1(E)_{\T} \cdot e^{\alpha_\T}\big)(i\xi)\hat{v}(\xi) \underline{d\xi} \\
		&= \int_\tor\int_X\frac{i}{2\pi}\tr (F_h+\langle\Phi_h,i\xi\rangle)\wedge e^{\omega+i\langle\mu,\xi\rangle}\hat{v}(\xi) \underline{d\xi}\\
		&= \int_X\int_\tor\frac{i}{2\pi} \tr\big(F_h\wedge\omega^{[n-1]}+
		\langle \Phi_h, i\xi \rangle\omega^{[n]}\big)\hat{v}(\xi) e^{i\langle\mu,\xi\rangle} \underline{d\xi} \\
		&= \int_X\frac{i}{2\pi}\tr(v(\mu)\Lambda_\omega F_h+\langle \Phi_h, dv(\mu)\rangle)\omega^{[n]} \\
		&= \int_X \frac{i}{2\pi}\tr(\Lambda_{\omega, v}(F_h+\Phi_h))\omega^{[n]},
	\end{align*} where we used \eqref{eq:FourierD} in going from the third line to the fourth. Taking the trace of equation \eqref{eq:v-weighted-HE} and integrating over \(X\) with respect to \(\omega^{[n]}\) gives \[c_v = \frac{(c_1(E)_\T\cdot v(\alpha_\T))}{\rk(E)\Vol(X,\omega)}.\qedhere\]
\end{proof}

\subsection{Equivariant intersection 2-forms}\label{sec:eq2forms}

\subsubsection{Weight functions of one variable} Once again, we begin by recalling the theory established by Inoue in \cite{Ino20}, letting \(f:\mathbb{R}\to\mathbb{R}\) be our real analytic function with power series expansion \(\sum_{k=0}^\infty\frac{a_k}{k!}x^k\) about the origin. In this case, we assume we have a proper holomorphic submersion \(\pi:X\to B\). We further assume there exists a compact group \(K\) so that \(\T\times K\) acts on both \(X\) and \(B\) making the submersion \(\pi\) equivariant, and so that the \(\T\)-action on \(B\) is trivial.

Let \(\alpha_{\T\times K} = [\omega+\mu]\) be a \(\T\times K\)-equivariant \(2\)-form on \(X\), and let \(\beta_{{\T\times K}}=[\sigma]\) be an auxiliary \(\T\times K\)-equivariant \(2k\)-form on \(X\). We will define \[(\beta_{{\T\times K}}\cdot f(\alpha_{{\T\times K}}))_B(\xi),\] which will be a degree 2 \(K\)-equivariant cohomology class on \(B\). Recall we have a pushforward map \[\pi_*:H^{2d}_{\T\times K}(X;\mathbb{R}) \to H^{2d-2n}_{\T\times K}(B;\mathbb{R}),\] where \(n:=\dim X - \dim B\). Since the torus action on \(B\) is trivial, the equivariant cohomology of \(B\) splits as \[H^{2l}_{\T\times K}(B;\mathbb{R}) = \bigoplus_{j+k=l}H^{2j}_\T(\pt;\mathbb{R})\otimes H^{2k}_K(B;\mathbb{R}).\] Since \(H^{2j}_\T(\pt;\mathbb{R}) = S^j(\mathfrak{t}^*)\), i.e. the degree \(j\) homogeneous polynomials on \(\tor\), evaluating at \(\xi\in\tor\) the component in \(H^{2j}_\T(\pt;\mathbb{R})\) we get a map \[\mathrm{el}_\xi : H^{2l}_{\T\times K}(X;\mathbb{R}) \to \bigoplus_{k=0}^l H^{2k}_K(B;\mathbb{R}).\] Denote by \(p_2:H^{*,\mathrm{even}}_K(B;\mathbb{R})\to H^2_K(B;\mathbb{R})\) the projection to the degree 2 component of equivariant cohomology. Combining all of this information, we get a map \[p_2\circ \mathrm{el}_\xi\circ \pi_*: \hat{H}^{*,\mathrm{even}}_{\T\times K}(X;\mathbb{R}) \to H^2_K(B;\mathbb{R}),\] and we denote the image of \((\beta_{{\T\times K}}\cdot f(\alpha_{{\T\times K}}))\) under this map by \[(\beta_{{\T\times K}}\cdot f(\alpha_{{\T\times K}}))_B(\xi).\]

In terms of the equivariant representatives \(\omega+\mu\) and \(\sigma\) for \(\alpha_{{\T\times K}}\) and \(\beta_{{\T\times K}}\) respectively, an equivariant representative for \((\beta_{{\T\times K}}\cdot f(\alpha_{{\T\times K}}))_B(\xi)\) may be computed naturally as \[\left[\int_{X/B}\langle\sigma,\xi\rangle f(\omega+\langle\mu,\xi\rangle)\right]_{\deg = 2}.\] Here, note that for example \(\mu = \mu_{\mathfrak{t}} + \mu_{\mathfrak{k}}\) and \(\langle\mu,\xi\rangle := \langle\mu_{\tor},\xi\rangle + \mu_{\mathfrak{k}}\), so the result of the fibre integral is a \(K\)-equivariant differential form on \(B\), of which we take the degree 2 component. The result is the sum of a \(K\)-invariant differential 2-form on \(B\) and a \(K\)-equivariant ``moment map" \(\nu:B\to\mathfrak{k}^*\) for that form, although the 2-form may not be non-degenerate in general.

\subsubsection{Arbitrary weight functions} Similarly as before, we use Fourier transforms to extend the above to general weight functions. In particular, let \(u\) be an arbitrary weight function on the moment polytope. Then we define \[(\beta_{{\T\times K}}\cdot u(\alpha_{\T\times K}))_B := \int_{\tor} (\beta_{{\T\times K}}\cdot e^{\alpha_{{\T\times K}}})_B(i\xi)\hat{u}(\xi)\underline{d\xi},\] which is a degree 2 \(K\)-equivariant cohomology class on \(B\).

A natural equivariant representative for this class is computed as \[\left[\int_{\tor}\left(\int_{X/B} \langle\sigma,i\xi\rangle e^{\omega+i\langle\mu,\xi\rangle}\right)\hat{u}(\xi)\underline{d\xi}\right]_{\deg = 2}.\] For the purposes of this paper, we will only need to know that the above expression is an equivariantly closed \(K\)-equivariant 2-form, and from this we will deduce that the weighted Hermite--Einstein equation occurs as a moment map. One may similarly prove that the above formulae are well-defined and independent of the choice of extension of \(u\) to \(\mathfrak{t}^*\), and agree with the original definition in the case \(u(\mu)=\tilde{u}(\langle\mu,\xi\rangle)\); will shall omit the proofs as they are essentially unchanged from the case of equivariant intersection numbers.

\section{Moment map property and stability}\label{sec:moment-map-and-stability}

In this section, we use the equivariant intersection theory to develop two fundamental aspects related to the weighted Hermite--Einstein equation: the moment map property, and the notion of weighted slope stability of a \(\T\)-equivariant sheaf. We then prove one direction of the weighted Kobayashi--Hitchin correspondence, namely that existence of a \(v\)-weighted Hermite--Einstein metric implies \(v\)-weighted slope polystability. Assuming the full weighted Kobayashi--Hitchin correspondence (to be proved in Section \ref{sec:DUY}), we revisit the examples in Sections \ref{sec:transverse} and \ref{sec:poly-weights}. As a further application, we introduce a suitable notion of weighted Gieseker stability for vector bundles, and relate it to weighted slope stability.

\subsection{Moment map property}

The moment map derivation of the Hermite--Einstein equation was originally given by Atiyah--Bott \cite{AB83}. Here in the weighted setting we will give a proof along the lines of \cite[Section 6]{DH23}, by integrating equivariant forms over the fibres of a universal family. Let \((E, h)\) be a \(\T\)-equivariant Hermitian vector bundle on \((X,\omega)\). Denote by \(\A(E,h)^{\T}\) the space of \(\T\)-invariant unitary connections on \(E\), which is an affine space modelled on \(\Omega^1(X,\mathrm{End}_{SH}E)^\T\), the space of \(\T\)-invariant 1-forms with values in the skew-Hermitian endomorphisms of \(E\). The gauge group \(\mathcal{G}^\T\) of \(\T\)-commuting unitary endomorphisms of \(E\) acts on \(\mathcal{A}(E,h)^\T\).

We consider \(X\times\A(E,h)^\T\) as a family over \(\A(E, h)^\T\). Denote by \(\E\) the pullback of \(E\) from \(X\) to \(X\times\A(E,h)^\T\). We define a universal connection \(D_{\E}\) on \(\E\) as follows. For \((u,v)\in T_{(x, \nabla)}(X\times \A(E, h)^\T)\) and \(s\in\Gamma(\E)\), define \[(D_\E s)(u,v):= (\nabla s)(u) + (d_{\A(E,h)^\T}s)(v).\] That is, in the vertical direction along a fibre \(X\times\{\nabla\}\), the connection is simply the basepoint \(\nabla\in\A(E,h)^\T\), and in the horizontal direction, \(\E\) is trivial so we may extend the usual exterior derivative on \(\A(E,h)^\T\) to \(\E\). This connection has curvature \(F_{\E}\), whose pure vertical component over the fibre \(X\times\{\nabla\}\) coincides with the curvature of \(\nabla\), and whose mixed horizontal-vertical component maps a horizontal tangent vector to the associated vertical 1-form with values in \(\mathrm{End}_{SH}(\E)\) under the identification \(T_\nabla\A(E,h)^\T\cong\Omega^1(X,\End_{SH}E)^\T\). The purely horizontal component of the curvature vanishes. Identifying the Lie algebra of \(\mathcal{G}^\T\) with the \(\T\)-commuting skew-adjoint endomorphisms of \(E\), a moment map \(\Psi_\E\) for the curvature \(F_\E\) is given by \[\langle\Psi_\E,e\rangle = -e.\] In particular \(F_\E+\Psi_\E\) is a \(\mathcal{G}^{\T}\)-equivariant \(\End(\E)\)-valued form on \(X\times\A(E,h)^\T\).

To produce a moment map on \(\A(E,h)^\T\), we use the methods of Section \ref{sec:eq2forms} to generate an equivariant intersection 2-form, beginning with choices of \(\T\times\mathcal{G}^\T\)-equivariant 2-forms on the total space \(X\times\A(E,h)^\T\). A representative for the class \(\alpha_{\T\times\mathcal{G}^\T}\) will be the pullback of the equivariant form \(\omega+\mu\) from \(X\) to \(X\times\A(E,h)^\T\); note that the \(\mathcal{G}^\T\)-action on \(X\) is trivial, so this is an equivariant \(\T\times\mathcal{G}^\T\)-form on \(X\times\A(E,h)^\T\). As an auxiliary equivariant form, note the curvature \(F_\E\) has two moment maps: the map \(\Psi_\E\) for the action of \(\mathcal{G}^\T\) described above, and on each fibre \(X\times\{\nabla\}\) a moment map \(\Phi_\nabla\) for the \(\T\)-action. Combining these two moment maps produces the \(\T\times\mathcal{G}^\T\)-equivariant \(\End(\E)\)-valued form \(F_\E + \Phi_\nabla + \Psi_\E\); note Proposition \ref{prop:canonical_moment_map} implies \(F_\E+\Phi_\nabla\) is indeed \(\T\)-equivariantly closed. To represent the auxiliary class \(\beta_{\T\times \mathcal{G}^\T}=\mathrm{ch}_2(\E)_{\T\times \mathcal{G}^\T}\), we take \[\frac{1}{2!}\tr\left(\left(\frac{i}{2\pi}(F_\E+\Phi_\nabla+\Psi_\E)\right)^2\right).\]

It now remains to compute the equivariant representative of the (formal) degree 2 equivariant class \((\beta_{\T\times \mathcal{G}^\T}\cdot v(\alpha_{\T\times \mathcal{G}^\T}))_{\A(E,h)^\T}\) on \(\A(E,h)^\T\), \begin{align*}
	&-\frac{1}{8\pi^2}\left[\int_{\tor}\left(\int_{X} \tr((F_\E+\Psi_\E+\langle\Phi_\nabla,i\xi\rangle)^2) e^{\omega+i\langle\mu,\xi\rangle}\right)\hat{v}(\xi)\underline{d\xi}\right]_{\deg = 2} \\
	=&-\frac{1}{8\pi^2}\int_{\tor}\int_{X} \tr((F_\E+\Psi_\E)^2)\wedge\omega^{[n-1]}e^{i\langle\mu,\xi\rangle}\hat{v}(\xi)\underline{d\xi} \\
	&-\frac{1}{4\pi^2}\int_{\tor}\int_{X} \tr((F_\E+\Psi_\E)\langle\Phi_\nabla,i\xi\rangle)\wedge\omega^{[n]} e^{i\langle\mu,\xi\rangle}\hat{v}(\xi)\underline{d\xi} \\
	=&-\frac{1}{8\pi^2}\int_{X} v(\mu)\tr((F_\E+\Psi_\E)^2)\wedge\omega^{[n-1]} \\
	&-\frac{1}{4\pi^2}\int_{X} \tr\left((F_\E+\Psi_\E)\int_\tor e^{i\langle\mu,\xi\rangle}\langle\Phi_\nabla,i\xi\rangle \hat{v}(\xi)\underline{d\xi}\right)\wedge\omega^{[n]} \\
	=&-\frac{1}{8\pi^2}\int_{X} v(\mu)\tr(F_\E^2)\wedge\omega^{[n-1]} \\
	&-\frac{1}{4\pi^2}\int_{X} v(\mu)\tr(F_\E\Psi_\E)\wedge\omega^{[n-1]}-\frac{1}{4\pi^2}\int_{X} \tr\left((F_\E+\Psi_\E)\langle\Phi_\nabla,dv(\mu)\rangle\right)\wedge\omega^{[n]} \\
	=&-\frac{1}{8\pi^2}\int_{X} v(\mu)\tr(F_\E^2)\wedge\omega^{[n-1]} \\
	&-\frac{1}{4\pi^2}\int_{X} \tr\left((v(\mu)\Lambda_\omega F_\E+\langle\Phi_\nabla,dv(\mu)\rangle)\Psi_\E\right)\omega^{[n]}.
\end{align*}

We also normalise by subtracting the appropriate multiple \(c_{v,w}\) of the corresponding representative for \((c_1(\E)_\T\cdot w(\alpha_\T))_{\A(E,h)^\T}\), whose 2-form part vanishes, \begin{align*}
	&\left[\int_\tor\int_X\frac{i}{2\pi}\tr(F_{\E}+\Psi_\E+\langle\Phi_\nabla,i\xi\rangle)\wedge e^{\omega+i\langle\mu,\xi\rangle}\hat{w}(\xi)\underline{d\xi}\right]_{\deg=2} \\
	=& \int_X\frac{i}{2\pi}\tr(\Psi_\E)w(\mu)\omega^{[n]},
\end{align*} producing \begin{align*}
	&\frac{-1}{8\pi^2}\int_{X} v(\mu)\tr(F_\E^2)\wedge\omega^{[n-1]} \\
	&+\int_{X} \tr\left(\left(\frac{i}{2\pi}\Lambda_{\omega,v}(F_\E+\Phi_\nabla)-c_{v,w}w(\mu)\Id_\E\right)\frac{i}{2\pi}\Psi_\E\right)\omega^{[n]}.
\end{align*} Note that in the 2-form component of this equivariant form, the only part of \(F_\E\) that contributes to the integral is the mixed component, which takes a tangent vector to \(\A(E,h)^\T\) and converts it to the corresponding vertical \(\End(\E)\)-valued 1-form on \(X\times\A(E,h)^\T\), under the isomorphism \(T_\nabla\A(E,h)^\T \cong \Omega^1(\End_{SH}(E))^\T\). In particular, for tangent vectors \(a,b\in T\A(E,h)^\T\), this 2-form may be written \[\Omega_v(a,b) := \frac{-1}{8\pi^2}\int_X v(\mu)\tr(a\wedge b)\wedge\omega^{[n-1]},\] which is a weighted modification of the usual Atiyah--Bott symplectic form. We thus conclude:

\begin{prop}\label{prop:weighted_AtBo}
	The map \(\sigma_{v,w}:\A(E,h)^\T\to\Gamma(\End_{SH}(E,h)^\T)^*\), \[\langle\sigma_{v,w}(\nabla), e\rangle := -\int_{X} \tr\left(\left(\frac{i}{2\pi}\Lambda_{\omega,v}(F_E+\Phi_\nabla)-c_{v,w}w(\mu)\Id_E\right)\frac{i}{2\pi}e\right)\omega^{[n]}\] is a moment map for the \(\mathcal{G}^\T\)-action on \(\A(E,h)^\T\) with respect to the symplectic form
	\[\Omega_v(a,b):=\frac{-1}{8\pi^2}\int_X v(\mu)\tr(a\wedge b)\wedge\omega^{[n-1]}.\]
\end{prop}

\begin{rem}
	The proof above extends to the finite-dimensional case, where one has a \(\T\times K\)-equivariant vector bundle \(E\to X\) along with a \(\T\times K\)-equivariant proper holomorphic submersion \(X\to B\), where the \(\T\)-action is trivial on \(B\). In particular, this allows not only for the complex structure of the vector bundle to vary over the fibres of \(X\to B\), but the complex structures of the fibres \(X_b\) themselves may also vary with \(b\in B\). The caveat is that the differential 2-form on \(B\) is no longer guaranteed to be non-degenerate in general.
\end{rem}

\subsection{Weighted slope stability}

Let \(\F\subset E\) be a subsheaf, associating to each open subset \(U\subset X\) an \(\mathcal{O}(U)\)-submodule \(\mathcal{O}(U,\F)\subset\mathcal{O}(U,E)\). Note that for each \(t\in\T\) and \(s\in\mathcal{O}(U,E)\), the \(\T\)-action on \(E\) induces a section \(t\cdot s\in\mathcal{O}(t\cdot U, E)\). The subsheaf \(\F\subset E\) is called \emph{\(\T\)-equivariant} if for each \(t\in\T\) and \(s\in\mathcal{O}(U,\F)\), \[t\cdot s \in \mathcal{O}(t\cdot U, \F).\]

Recall that for a torsion-free coherent sheaf \(\F\), the \emph{determinant line bundle} of \(\F\) may be defined as \[\det(\F) := (\Lambda^{\rk(\F)}\F)^{**}.\] Canonical operations on sheaves such as tensor powers, exterior powers and duals all preserve the property of being \(\T\)-equivariant, so that if \(\F\) is a \(\T\)-equivariant torsion-free coherent sheaf, the line bundle \(\det(\F)\) is naturally \(\T\)-equivariant. Given a \(\T\)-equivariant torsion-free coherent sheaf \(\F\), we therefore define \[c_1(\F)_{\T}:=c_1(\det(\F))_{\T}.\] Note that any \(\T\)-equivariant coherent subsheaf \(\F\subset E\) is automatically torsion-free, and so \(c_1(\F)_\T\) is defined. We lastly recall that a coherent subsheaf \(\F\subset E\) is called \emph{saturated} if the quotient \(E/\F\) is torsion-free. In particular, a saturated sheaf is reflexive \cite[Proposition 5.5.22]{Kob14}, and has singular locus having complex codimension at least 2 in \(X\) \cite[Corollary 5.5.20]{Kob14}.

\begin{defn}
	Let \(\F\) be a \(\T\)-equivariant torsion-free coherent sheaf on \((X,\alpha_\T)\). The \emph{\(v\)-weighted slope} of \(\F\) is \[\mu_v(\F):=\frac{(c_1(\F)_{\T}\cdot v(\alpha_{\T}))}{\mathrm{rk}(\F)}.\] A \(\T\)-equivariant vector bundle \(E\) on \((X,\alpha_T)\) is: \begin{enumerate}
		\item \emph{\(v\)-weighted slope semistable} if for all proper non-zero \(\T\)-equivariant coherent saturated subsheaves \(\F\subset E\), \[\mu_v(\F)\leq\mu_v(E);\]
		\item \emph{\(v\)-weighted slope stable} if for all proper non-zero \(\T\)-equivariant coherent saturated subsheaves \(\F\subset E\), \[\mu_v(\F)<\mu_v(E);\]
		\item \emph{\(v\)-weighted slope polystable} if \(E\) is a direct sum of \(\T\)-equivariant stable vector bundles of the same weighted slope;
		\item \emph{\(v\)-weighted slope unstable} if \(E\) is not \(v\)-weighted slope semistable.
	\end{enumerate}
\end{defn}

\begin{rem}
	We could similarly define the \emph{\((v,w)\)-weighted slope} of a \(\T\)-equivariant sheaf \(\F\) as \[\mu_{v,w}(\F) := \frac{(c_1(\F)_{\T}\cdot v(\alpha_{\T}))}{\mathrm{rk}(\F)(w(\alpha_{\T}))},\] along with various analogues of \((v,w)\)-weighted slope stability, etc. Of course, the definitions of \((v,w)\)-weighted slope stability and \(v\)-weighted slope stability coincide, which is the algebro-geometric analogue of Lemma \ref{vw-weighted-is-v-weighted}.
\end{rem}

The following lemma will be needed in the proof of the weighted Kobayashi--Hitchin correspondence.

\begin{lemma}\label{lem:codim2integral}
	Let \(\mathcal{F}\subset E\) be a \(\T\)-equivariant coherent subsheaf whose singular locus \(V\) has codimension at least 2 in \(X\). Let \(h\) be a \(\T\)-invariant Hermitian metric on \(E\), and denote by \(F_\mathcal{F}+\Phi_\mathcal{F}\) the curvature and moment map associated to the restricted metric \(h|_\F\) over \(X\backslash V\). Then \[(c_1(\mathcal{F})_{\T}\cdot v(\alpha_{\T})) = \int_{X\backslash V}\frac{i}{2\pi}\tr\Lambda_{\omega, v}(F_\F+\Phi_\F)\omega^{[n]}.\]
\end{lemma}

In other words, the equivariant intersection number \((c_1(\F)\cdot v(\alpha_{\T}))\) can be computed by differential geometric means on the smooth locus of the subsheaf.

\begin{proof}
	We follow the approach of Kobayashi \cite[pp. 181-182]{Kob14}. First, consider the map \[j:\det(\F):=(\Lambda^{\rk(\F)}\F)^{**}\to\Lambda^{\rk(\F)}E.\] Away from \(V\) this map is injective on the fibres. Let \(\tau\) be a local holomorphic non-vanishing section of \(\det(\F)\), and let \(e_1,\ldots, e_r\) be a local holomorphic frame for \(E\) over the same open subset. Then \[j(\tau) = \sum_{I}\tau^Is_I,\] where \(I:=(i_1,\ldots,i_{\rk(\F)})\) for \(i_1<i_2<\cdots<i_{\rk(\F)}\), \(s_I:=e_{i_1}\wedge\cdots\wedge e_{i_{\rk(\F)}}\), and the \(\tau^I\) are local holomorphic functions. Denote by \(u:=j^*\Lambda^{\rk(\F)}h\) the (degenerate along \(V\)) metric induced by \(h\). Fix an arbitrary smooth \(\T\)-invariant Hermitian metric \(\tilde{u}\) on \(\det(\F)\), and define \[f:=u(\tau,\tau)/\tilde{u}(\tau,\tau).\] In particular, \(f\) is a smooth non-negative \(\T\)-invariant function that is independent of the choice of local section \(\tau\), so is globally defined on \(X\). Furthermore, \(f\) vanishes if and only if each of the \(\tau^I\) vanish, since the metric on \(\Lambda^{\rk(\F)}E\) is non-degenerate. Denote by \(W\subset V\) this vanishing locus, which is a \(\T\)-invariant subvariety of \(X\) of codimension at least 2.
	
	Next, choose a \(\T\)-invariant holomorphic map \(\pi:\tilde{X}\to X\) so that \(D:=\pi^{-1}(W)\) has the following properties: \(\mathrm{codim}D = 1\), \(\pi:\tilde{X}\backslash D\to X\backslash W\) is a biholomorphism, and \(\pi^*\mathcal{I}_W = \mathcal{O}_{\tilde{X}}(-mD)\) for some \(m>0\), where \(\mathcal{I}_W\) denotes the ideal sheaf of \(W\). Let \(\zeta\) be a local holomorphic function defining \(D\), so that for each \(I\), \(\pi^*\tau^I / \zeta^m\) is holomorphic and at least one of these functions is non-vanishing. Then we may further write \[\pi^*f = a|\zeta|^{2m}\] for some local positive smooth function \(a\).
	
	We are now in a position to compute \((c_1(\F)_\T\cdot v(\alpha_{\T}))\). Denote by \(\tilde{c}_1(\det(\F), \tilde{u})_\T\) (resp. \(\tilde{c}_1(\det(\F), u)_\T\)) the equivariant curvature of the metric \(\tilde{u}\) on \(\det(\F)\) (resp. \(u\) on \(\det(\F)|_{X\backslash W}\)). For any \(k\geq n-1\), we begin with \begin{align*}
		&\int_X \tilde{c}_1(\det(\F),\tilde{u})_{\T}\wedge(\omega+\mu)^{[k]} \\
		=& \int_{\tilde{X}\backslash D} \pi^*\tilde{c}_1(\det(\F),\tilde{u})_{\T}\wedge\pi^*(\omega+\mu)^{[k]} \\
		=& \int_{\tilde{X}\backslash D} \pi^*\left(\tilde{c}_1(\det(\F),u)_{\T}-\frac{i}{2\pi} \db^{\eq}\d \log f\right)\wedge\pi^*(\omega+\mu)^{[k]},
	\end{align*} where we used \eqref{eq:change-of-curvature} in the final line. Note the integral of \(\db^{\eq}\d\log\pi^*f\wedge\pi^*(\omega+\mu)^{[k]}\) over \(\tilde{X}\backslash D\) converges absolutely, since near \(D\) we have \[\db^{\eq}\d\log\pi^*f = \db^{\eq}\d\log a + m\frac{d\zeta}{\zeta},\] where the second term is interpreted as a map \(\tilde{X}\backslash D\to\tor^*\) in the usual way. Moreover, we see that for \(\xi\in\tor\), \begin{equation}\label{eq:integral-bound}
		\int_{\tilde{X}\backslash D}|\langle\db^{\eq}\d\pi^*\log f,\xi\rangle\wedge \pi^*(e^\omega)| \leq C_1|\xi|+C_2,
	\end{equation} for certain constants \(C_1,C_2>0\); here the absolute value in the integral is interpreted as taking the degree \(2n\) component (as a differential form), writing it as a multiple \(g\omega^{[n]}\) for some function \(g\), and replacing with \(|g|\omega^{[n]}\). For all \(\epsilon > 0\) sufficiently small, let \(N_\epsilon\) denote a \(\T\)-invariant tubular neighbourhood of \(D\) of radius \(\epsilon\) for a choice of smooth \(\T\)-invariant metric on \(\tilde{X}\). Then \begin{align*}
		& -\int_{\tilde{X}\backslash D} \frac{i}{2\pi} \db^{\eq}\d \pi^*\log f\wedge\pi^*(\omega+\mu)^{[k]} \\
		=& -\lim_{\epsilon\to 0}\int_{\tilde{X}\backslash N_\epsilon} \frac{i}{2\pi} \db^{\eq}\d \pi^*\log f\wedge\pi^*(\omega+\mu)^{[k]} \\
		=& -\lim_{\epsilon\to 0}\int_{\tilde{X}\backslash N_\epsilon} \frac{i}{2\pi} d^{\eq}\left(\d \pi^*\log f\wedge\pi^*(\omega+\mu)^{[k]}\right) \\
		=& \lim_{\epsilon\to0}\int_{\d N_\epsilon}\frac{i}{2\pi} \d \pi^*\log f\wedge\pi^*(\omega+\mu)^{[k]},
	\end{align*} where in the final line we used equivariant Stokes theorem (see \cite[Lemma 3.14]{Ino21}, for example), and the minus sign is lost due to the orientation of \(\d N_\epsilon\). Now, on a small open set where \(\pi^*f= a|\zeta|^{2m}\), the contributions from \[\frac{i}{2\pi} \d \pi^*\log a\wedge\pi^*(\omega+\mu)^{[k]}\] to the integral become increasingly small as \(\epsilon\to0\), since this form is bounded and \(\d N_\epsilon\) has volume approaching 0. Thus only the local term \[\frac{i}{2\pi} \d \pi^*\log|\zeta|^{2m}\wedge\pi^*(\omega+\mu)^{[k]}\] may contribute in the limit. By the usual (i.e. non-equivariant) Poincar{\'e}--Lelong formula, the corresponding contributions tend towards \[m\int_D \mu^{[k-n-1]}\omega^{[n-1]}\] as \(\epsilon\to0\). Since \(\omega|_D\) is pulled back from \(W\) which has codimension at least 2, this integral vanishes, and we have proven \begin{align*}
		\int_X \tilde{c}_1(\det(\F),\tilde{u})_{\T}\wedge(\omega+\mu)^{[k]} =&\int_{X\backslash W} \tilde{c}_1(\det(\F),u)_{\T}\wedge(\omega+\mu)^{[k]} \\
		=&\int_{X\backslash V} \tilde{c}_1(\det(\F),u)_{\T}\wedge(\omega+\mu)^{[k]}.
	\end{align*} Summing from \(k=n-1\) to \(\infty\), we see that for \(\xi\in\mathfrak{t}\oplus i\mathfrak{t}\), the equivariant intersection number \((c_1(\F)_\T\cdot e^{\langle\alpha_{\T},\xi\rangle})\) can be computed as \[\int_{X\backslash V}\langle\tilde{c}_1(\det(\F),u)_{\T},\xi\rangle\wedge e^{\omega+\langle\mu,\xi\rangle}.\] Applying the inverse Fourier transform, we have \[(c_1(\F)_\T\cdot v(\alpha_\T)) = \int_\tor\int_{X\backslash V}\langle \tilde{c}_1(\det(\F),u)_\T, i\xi\rangle \wedge e^{\omega+i\langle\mu,\xi\rangle} \hat{v}(\xi)\underline{d\xi}.\] It remains to observe we may exchange the order of integration. To see this, recall that \[\tilde{c}_1(\det(\F),u)_\T = \tilde{c}_1(\det(\F),\tilde{u})_\T+\frac{i}{2\pi}\db^{\eq}\d\log f;\] clearly the corresponding integral with \(\tilde{u}\) in place of \(u\) converges absolutely, since \(\tilde{u}\) extends smoothly over \(\tilde{X}\). Finally the inequality \eqref{eq:integral-bound} shows that the integral \[\int_\tor\int_{X\backslash V}\langle i\db^\eq\d \log \pi^* f, i\xi\rangle \wedge e^{\omega+i\langle\mu,\xi\rangle}\hat{v}(\xi)\underline{d\xi}\] also converges absolutely, since \(\hat{v}\) is Schwartz class. Applying Fubini's theorem, we may exchange the order of integration, and we are done.
\end{proof}

\subsection{Existence of a weighted Hermite--Einstein metric implies weighted slope polystability}

Let $E$ be a $\T$-equivariant holomorphic vector bundle over a compact K\"ahler $\T$-manifold $(X,\omega)$, and let $\F\subset E$ be a $\T$-equivariant holomorphic subbundle of $E$. A $\T$-invariant Hermitian metric $h_E$ on $E$ induces a $\T$-invariant Hermitian metric $h_\F$ on $\F$ by restriction. The orthogonal complement of $\F$ in $(E,h_E)$ defines a smooth complex subbundle $\F^\perp$ of $E$ that is $\T$-equivariant. As a smooth complex vector bundle, $\F^\perp$ is isomorphic to the $\T$-equivariant quotient bundle $Q:=E/\F$. We denote by $h_{Q}$ the induced Hermitian metric on $Q$.

Let $\nabla^E$, $\nabla^\F$, and $\nabla^Q$ be the Chern connections of $(E,h_E)$, $(\F,h_\F)$, and $(Q,h_Q)$ respectively. For all $s\in \Gamma(\F)\subset\Gamma(E)$, we have the decomposition
\[
\nabla^E s=\nabla^\F s+A(s),
\]
relative to $E=\F\oplus \F^{\perp}$ where $A\in\Omega^{1,0}({\rm Hom}( \F,\F^{\perp}))$ is the second fundamental form of $\F$ in $(E,h_E)$. It follows that the bundle moment map $\Phi_E$ for the action of $\T$ on $E$ decomposes as follows
\[
\langle\Phi_E(s),\xi\rangle=\langle\Phi_\F(s),\xi\rangle+A(s)(\xi)
\]
for all $\xi\in\tor$ and $s\in\Gamma(\F)$. 
Using the smooth isomorphism $Q\to \F^\perp$, for all $s\in\Gamma(\F^\perp)$ we have the decomposition
\[
\nabla^Es=\nabla^{Q}s-A^\dagger(s)
\]
where $A^\dagger$ is the adjoint of $A$ relative to $h_E$. It follows that the bundle moment map $\Phi_E$ for the action of $\T$ on $E$ also decomposes as
\[
\langle\Phi_E(s),\xi\rangle=\langle\Phi_Q(s),\xi\rangle-A^\dagger(s)(\xi)
\]
for all $\xi\in\tor$ and $s\in\Gamma(\F^\perp)$.

In summary, relative to the decomposition $E=\F\oplus \F^{\perp}$, the curvature $F_E$ of $\nabla^E$ and the bundle moment map admit the following decompositions
\begin{equation}\label{eq:decomposition}
	\begin{split}
		F_{\nabla^E}=\begin{pmatrix}
			F_{\nabla^\F}-A^\dagger\wedge A & \nabla^{1,0}A^\dagger\\
			-\db A & F_{\nabla^Q}-A\wedge A^\dagger
		\end{pmatrix}, \quad
		\Phi_{E}=\begin{pmatrix}
			\Phi_\F & -A^\dagger\\
			A&\Phi_Q
		\end{pmatrix},
	\end{split}
\end{equation} where we use the isomorphism \(Q\cong E/\F\) to identify \(A\) with an element of \(\Omega^{1,0}(\Hom(\F,Q))\) and \(A^\dagger\) with an element of \(\Omega^{0,1}(\Hom(Q,\F))\).

Of course, in the case where \(\F\subset E\) is a \(\T\)-equivariant coherent saturated subsheaf, the above discussion goes through in an identical manner over the open locus where \(\F\) is a vector bundle.

\begin{lemma}\label{lem:pointwise-ineq}
	Let $(E,h_E)$ be a $\T$-equivariant $v$-weighted Hermite--Einstein bundle on $(X,\omega)$ and $\F$ a $\T$-equivariant coherent saturated subsheaf of $E$ with singular locus \(V\). Then over \(X\backslash V\), the inequality
	\[
	\frac{\tr_\F\,i\Lambda_{\omega,v}(F_{\nabla^\F}+\Phi_\F)}{\rk(\F)}\leq \frac{\tr_E\,i\Lambda_{\omega,v}(F_{\nabla^E}+\Phi_E)}{\rk(E)}
	\]
	holds pointwise. Moreover, if equality holds on all of \(X\backslash V\), then $E$
	splits as $E=\F\oplus E/\F$ over \(X\backslash V\), where each component is $v$-weighted Hermite--Einstein with the same $v$-weighted slope.
\end{lemma}

\begin{proof}
	By the weighted Hermite--Einstein condition, on \(X\backslash V\) we have that
	\[
	\frac{\tr_E\,i\Lambda_{\omega,v}(F_{\nabla^E}+\Phi_E)}{\rk(E)}=2\pi c_v.
	\]
	On the other hand, using \eqref{eq:decomposition} we get
	\[
	\begin{split}
		\frac{\tr_\F\,i\Lambda_{\omega,v}(F_{\nabla^\F}+\Phi_\F)}{\rk(\F)}=& \frac{\tr_\F\,i\Lambda_{\omega,v}(F_{\nabla^E}+\Phi_E)|_\F}{\rk(\F)} + \frac{\tr_\F\big(iv(\mu)\Lambda_\omega (A^\dagger\wedge A)\big)}{\rk(\F)}\\
		=& 2\pi c_v+ \frac{\tr_\F\big(iv(\mu)\Lambda_\omega (A^\dagger\wedge A)\big)}{\rk(\F)}.
	\end{split}
	\]
	We infer that
	\[
	\frac{\tr_\F\,i\Lambda_{\omega,v}(F_{\nabla^\F}+\Phi_\F)}{\rk(\F)}- \frac{\tr_E\,i\Lambda_{\omega,v}(F_{\nabla^E}+\Phi_E)}{\rk(E)}=\frac{\tr_\F\big(iv(\mu)\Lambda_\omega (A^\dagger\wedge A)\big)}{\rk(\F)}\leq  0
	\]
	with equality if and only if $A=0$. In the case \(A=0\) on \(X\backslash V\), by \cite[Proposition 8.7]{Sek21} we have $E=\F\oplus E/\F$ over \(X\backslash V\), and using \eqref{eq:decomposition}, each component is $v$-weighted Hermite--Einstein. By Lemma \ref{lem:codim2integral}, we can compute the weighted slopes over the smooth locus, and both \(\F\) and \(E/\F\) have the same $v$-weighted slope equal to $\mu_v(E)$.
\end{proof}

\begin{prop}\label{prop:polystability}
	If \(E\) admits a \(v\)-weighted Hermite--Einstein metric, then \(E\) is \(v\)-weighted slope polystable.
\end{prop}

\begin{proof}
	Let \(\F\subset E\) be a \(\T\)-equivariant coherent saturated subsheaf. By integrating the inequality of Lemma  \ref{lem:pointwise-ineq} over the smooth locus of \(\F\) and applying Lemma \ref{lem:codim2integral}, we see that \[\mu_v(\F)\leq\mu_v(E),\] and so \(E\) is \(v\)-weighted slope semistable.
	
	Suppose \(\F\) is a \(\T\)-equivariant proper non-zero coherent saturated subsheaf such that \(\mu_v(\F) = \mu_v(E)\). Our aim is to show that \(\F\) is a vector bundle, and \(E\) splits as the direct sum \(E = \F\oplus E/\F\). The equality \(\mu_v(\F)=\mu_v(E)\) implies that \(A=0\) on the smooth locus of \(\F\), so that the restriction of \(E\) to \(X\backslash V\) splits as \(E|_{X\backslash V} = \F|_{X\backslash V}\oplus (E/\F)|_{X\backslash V}\) by Lemma \ref{lem:pointwise-ineq}. The exact argument in the proof of \cite[Theorem 8.1]{Sek21} gives that the splitting extends to \(E=\F\oplus E/\F\) over all of \(X\). Since \(E\) is a vector bundle, this implies \(E\) and \(E/\F\) are also vector bundles, and they have the same slope by Lemma \ref{lem:pointwise-ineq}. Both \(\F\) and \(E/\F\) admit \(v\)-weighted Hermite--Einstein metrics, and applying the argument recursively on \(\F\) and \(E/\F\), we get that \(E\) is a direct sum of \(v\)-slope stable vector bundles of the same slope.
\end{proof}

\subsection{Application to transversally Hermite--Einstein metrics and polynomially weighted Hermite--Einstein metrics}

Although we have not yet proven the full weighted Kobayashi--Hitchin correspondence (see Section \ref{sec:DUY}), we take the opportunity here to apply the full correspondence to transversally Hermite--Einstein metrics and polynomially weighted Hermite--Einstein metrics.

Using the notation and conventions of Section \ref{sec:transverse}, we prove the following:

\begin{prop}\label{prop:transverse-converse}
	The bundle \(\tilde{E}\to S\) admits a transversally Hermite--Einstein connection if and only if \(E\to X\) admits a \((v,w)\)-weighted Hermite--Einstein metric.
\end{prop}

\begin{proof}
	Without loss of generality, assume that \(E\) is simple. Choose an arbitrary \(\T\)-invariant metric \(h\) on \(E\) and denote by \(\tilde{\theta}\) the induced connection on \(\tilde{E}\to S\) as per Section \ref{sec:transverse}. Let \(\F\subset E\) be a proper non-zero \(\T\)-equivariant coherent saturated subsheaf of \(E\), inducing a proper non-zero transverse coherent saturated subsheaf \(\tilde{\F}\subset \tilde{E}\), in the sense of \cite[Section 3.3]{BH22}. As per \cite[Definition 3.19, Proposition 3.21]{BH22}, the transverse slope of \(\tilde{\F}\) is \begin{align*}
		\mu(\tilde{\F}) =& \frac{i}{2\pi}\frac{\int_{S\backslash \pi^{-1}(V)}\tr_\F(F_{\tilde{\theta}|_\F})\wedge(d A_\xi)^{[n-1]}\wedge A_\xi}{\rk(\tilde{\F})} \\
		=& \frac{i}{2\pi}\frac{\int_{S\backslash \pi^{-1}(V)}\pi^*\tr_\F(f\Lambda_\omega F_{h|_\F}- n\Phi_{h|_\F})\wedge f^{-n}\pi^*\omega^{[n]}\wedge f^{-1}A
		}{\rk(\F)} \\
		=& \frac{i}{2\pi}\frac{\int_{X\backslash V}\tr_\F(f^{-n}\Lambda_\omega F_{h|_\F} - nf^{-n-1}\Phi_{h|_\F}) \,\omega^{[n]}
		}{\rk(\F)} \\
		=& \mu_{v}(\F),
	\end{align*} where \(V\) denotes the singular locus of \(\F\). Hence \(E\) is \(v\)-slope stable if and only if \(\tilde{E}\) is transversally slope stable. The weighted Kobayashi--Hitchin correspondence, the foliated Kobayashi--Hitchin correspondence of \cite{BH22}, and Lemma \ref{vw-weighted-is-v-weighted} immediately imply the result.
\end{proof}

Next, following the set-up of Section \ref{sec:poly-weights}, we prove:

\begin{prop}\label{prop:semi-simple-fibration}
	The bundle $E_Q\to (X_Q,\omega_Q)$ admits a Hermite--Einstein structure if and only if $E\to (X,\omega)$ admits a $(p,p)$-weighted Hermite--Einstein structure.
\end{prop}
\begin{proof}
	Without loss of generality, we can suppose that $E\to (X,\omega)$ is simple. Let $\F\subset E$ be a proper non-zero $\T$-equivariant coherent saturated subsheaf of $E$. Then, $\F_Q:=Q\times_\T\F$ is a coherent saturated subsheaf of $E_Q$. Since $E_Q\to (X_Q,\omega_Q)$ admits a Hermite--Einstein structure, then by the Hitchin--Kobayashi correspondence, $E_Q\to (X_Q,\omega_Q)$ is slope stable,
	\[
	\mu(\F_Q)<\mu(E_Q).
	\]
	Let $\nabla$ be an arbitrary $\T$-invariant connection on $E\to (X,\omega)$ and $\nabla^Q$, $(\nabla^{\F})^Q$ the induced bundle compatible connections on $E_Q\to (X_Q,\omega_Q)$ and $\F_Q\to (X_Q,\omega_Q)$ respectively. By \eqref{eq:p-MeanCurv}, we have
	\[
	\begin{split}
		\mu(E_Q)=&\frac{i}{2\pi}\frac{\int_{X_Q}\tr\left(\Lambda_{\omega_Q}F_{\nabla^Q}\right)\omega_Q^{[n+m]}}{\rk(E_Q)([\omega_Q])^{[n+m]}}\\
		=&\frac{i}{2\pi}\frac{\int_{X_Q}\tr\left(\Lambda_{\omega_Q}F_{\nabla^Q}\right)\omega_Q^{[n+m]}}{\rk(E)\int_{X_Q}\omega_Q^{[n+m]}}\\
		=&\frac{i}{2\pi} \frac{\int_{Q\times X}\tr\left(\frac{\Lambda_{\omega,p}(F_\nabla+\Phi_\nabla)}{p(\mu)}\right) p(\mu) \omega^{[n]}\wedge  \bigwedge_{\alpha=1}^{N} \pi^{*}_B\omega_\alpha^{[n_\alpha]}\wedge\eta^\ell}{\rk(E)\int_{Q\times X} p(\mu) \omega^{[n]}\wedge  \bigwedge_{\alpha=1}^{N} \pi^{*}_B\omega_\alpha^{[n_\alpha]}\wedge\eta^\ell}\\
		=&\frac{i}{2\pi} \frac{\int_{X}\tr\left(\Lambda_{\omega,p}(F_\nabla+\Phi_\nabla)\right)  \omega^{[n]}}{\rk(E)\int_{X}p(\mu) \omega^{[n]}}\\
		=&\mu_{p,p}(E).
	\end{split}
	\]
	A similar computation using the induced connection $(\nabla^{\F})^Q$ on $\F\to (X,\omega)$, gives
	\[
	\mu(\F_Q)=\mu_{p,p}(\F).
	\]
	Hence, $\mu_{p,p}(\F)<\mu_{p,p}(E)$ showing that $E\to (X,\omega)$ is $(p,p)$-weighted slope stable. By Theorem \ref{thm:main_wHK}, $E\to(X,\omega)$ admits a $p$-weighted Hermite--Einstein structure.
\end{proof}

Note that by uniqueness of transverse Hermite--Einstein metrics (resp. Hermite--Einstein metrics) the metric on \(\tilde{E}\to S\) (resp. the metric on \(E_Q\to X_Q\)) must be induced from a \((v,w)\)-weighted (resp. \((p,p)\)-weighted) Hermite--Einstein metric on \((X,\omega)\).

\subsection{The weighted Euler characteristic and weighted Gieseker stability}

Let $(X,L)$ be a projective K\"ahler manifold polarized by an ample holomorphic line bundle $L$ and let $\T\subset\Aut(X,L)$ be a real torus acting on the total space of $L$ and covering a torus action on \(X\). Let $E$ be a $\T$-equivariant vector bundle over $X$ with $\mathcal{L}^E$ the induced infinitesimal action on the space of sections. If \(s\in H^0(E)\) is a holomorphic section, so too is \(\exp(-t\xi)s\) for any \(\xi\in\tor\) and \(t\in\mathbb{R}\). Hence \[0 = \frac{d}{dt}(\db \exp(-t\xi)s) = \db\left(\frac{d}{dt}\exp(-t\xi)s\right).\] Evaluating at \(t=0\) we see \(\db\mathcal{L}_\xi^Es = 0\), so \(\mathcal{L}_\xi^E\) acts on \(H^0(E)\).
For a $\T$-invariant Hermitian metric $h$ on $E$ we let $\langle -,-\rangle_h$ be the $L^2$-inner product on the space of sections:
\[
\langle s,s'\rangle_h:=\int_X h(s,s')\omega^{[n]}.
\]
Since \(\mathcal{L}_\xi^E h = 0\), we deduce \(\langle\mathcal{L}_\xi^E s,s'\rangle_h = -\langle s,\mathcal{L}_\xi^E s'\rangle_h\).
It follows that for all $\xi\in\tor$ the operator $ \widetilde{\mathcal{L}}^E_\xi:=-i\mathcal{L}^E_\xi:\Omega^0(E)\to\Omega^0(E)$ is Hermitian. Furthermore by the general property of Lie derivatives  $[\mathcal{L}^E_{\xi},\mathcal{L}^E_{\xi'}]=\mathcal{L}_{[\xi,\xi']}^E$ it follows that $\left[\widetilde{\mathcal{L}}^E_{\xi},\widetilde{\mathcal{L}}^E_{\xi'}\right]=0$ for all $\xi,\xi'\in\tor$. In summary, \(\xi\mapsto\tilde{\mathcal{L}}_\xi^E\) defines a Lie algebra representation of \(\tor\) on \(\Herm(H^0(E), \langle-,-\rangle_h)\).

We consider the $\T$-equivariant vector bundle $E_k:=E\otimes L^k$ over $X$ with the induced infinitesimal action $\mathcal{L}^{(k)}:=\mathcal{L}^E+\Id_E\otimes\mathcal{L}^{L^k}$ on $\Omega^0(E_k)$ and we denote by $\widetilde{\mathcal{L}}^{(k)}:=-\frac{i}{k}\mathcal{L}^{(k)}$ the corresponding Hermitian operators. For $\xi\in\tor$, the spectrum (counted with multiplicity) of the Hermitian operator $\widetilde{\mathcal{L}}^{(k)}_\xi: H^0(E_k)\to H^0(E_k)$ is given by $\{\langle\lambda^{(k)}_j,\xi\rangle \mid \lambda_j^{(k)}\in W_k\}$ where $W_k=\{\lambda_j^{(k)} \mid j=1,\cdots, h^0(E_k)\}\subset \Lambda^*$ is the finite set of weights of the complexified action of $\T$ on $H^0(E_k)$ and $\Lambda^*$ is the dual of the lattice $\Lambda$ of circle subgroups of $\T$.

Let $\Pol_L\subset\mathfrak{t}^*$ be the momentum polytope associated to $(X,L,\T)$. We set $h^k:=h\otimes (h_{L})^{\otimes k}$ where $h$ is a $\T$-invariant metric on $E$ and $h_L$ is a $\T$-invariant Hermitian metric on $L$ whose curvature 2-form (times \(i/2\pi\)) is $\omega$, let $\lambda_j^{(k)}\in W_k$, $\xi\in\tor$, and take $s_j\in H^0(E_k)$ an eigensection of $\widetilde{\mathcal{L}}^{(k)}_\xi $ associated to the eigenvalue $\langle\lambda_j^{(k)},\xi\rangle$. Noting that the moment map for \(F_{h^k}=F_h\otimes\Id_{L^k}+\Id_E\otimes F_{L^k}\) is \(\Phi_{h^k} = \Phi_h\otimes\Id_{L^k} + \Id_E\otimes k\Phi_{h_L}\), we compute
\[
\begin{split}
	\langle\lambda_j^{(k)},\xi\rangle |s_j|^2_{h^k} =&  h^k\left(\widetilde{\mathcal{L}}^{(k)}_\xi s_j,s_j\right) \\
	=&-h^k\left(\frac{i}{k}(\nabla^{h^k}_\xi-\Phi^\xi_h)(s_j)-\mu^\xi s_j,s_j\right)\\
	=&-\frac{i}{k}(\partial |s_j|_{h^k}^2 )(\xi)+\mu^\xi|s_j|^2_{h^k} + \frac{i}{k} h^k(\Phi^\xi_h(s_j),s_j)
\end{split}
\]
where we use $s_j\in H^0(E_k)$ in the last equality to conclude \(h^k(\nabla^{h^k}s_j,s_j)=\d|s_j|_{h^k}^2\). At a global maximum $x_{j,k}\in X$ of the smooth function $x\mapsto |s_j|^2_{h^k}(x)$, we get
\[
\langle\lambda_j^{(k)},\xi\rangle=\mu^\xi(x_{j,k}) +\frac{1}{k}\frac{ h^k_{x_{j,k}}\big([i\Phi^\xi_h(s_j)](x_{j,k}),s_j(x_{j,k})\big)}{|s_j|^2_{h^k}(x_{j,k})}.
\]
By the Cauchy--Schwarz inequality we obtain
\[
|\langle\lambda_j^{(k)}-\mu(x_{j,k}),\xi\rangle| \leq \frac{1}{k}\frac{ |i\Phi^\xi_h(s_j)|_{h_k}(x_{j,k})}{|s_j|_{h^k}(x_{j,k})}\leq \frac{1}{k}|\Phi^\xi_h|_h(x_{j,k})
\]
where $|\Phi^\xi_h|_h$ is the pointwise endomorphism norm. Given a basis $(\xi_a)_{a=1,\cdots,\ell}$ of $\mathbb{S}^1$-generators of $\tor$, for all $k\geq 1$ we have
\[
\begin{split}
	||\lambda_j^{(k)}-\mu(x_{j,k})||_{\tor^*} &=\sum_{a=1}^\ell |\langle\lambda_j^{(k)}-\mu(x_{j,k}),\xi_a\rangle|\\
	&\leq \frac{1}{k}\sum_{a=1}^\ell |\Phi^{\xi_a}_h|_h(x_{j,k})\\
	&\leq \frac{1}{k}\sum_{a=1}^\ell \max_X|\Phi^{\xi_a}_h|_h\\
	&\to 0
\end{split}
\] as \(k\to\infty\).
Thus, for any compact set \(Q\subset\tor^*\) whose interior contains \(P_L\), there exists an integer \(K>0\) such that for all $k\geq K$ we have
\[
W_k\subset Q.
\] In particular, for \(k\gg0\), the sets \(W_k\) are almost contained in \(P_L\). In fact, one can show that the convex hulls of the \(W_k\) approach \(P_L\) in a suitable sense, by considering for each vertex \(\lambda\in P_L\) a corresponding weight \(s\in H^0(L)\), and tensoring an arbitrary eigensection of \(E_\ell\) for some \(\ell\) with \(s^{\otimes k}\) to get an eigensection of \(E_{\ell+k}\), then sending \(k\to\infty\). The corresponding weights of the eigensections of \(E_{\ell+k}\) tend towards \(\lambda\) in the limit.

Let $v\in C^\infty(\Pol_{L})$ a weight function. Given an extension of $v$ to a compactly supported smooth function on $\tor^*$, also denoted by $v$, one can use the weight decomposition of $H^0(E_k)$ as
a direct sum of common eigenspaces for the commuting operators $\widetilde{\mathcal{L}}^{(k)}_\xi$, $\xi\in\tor$,
\[
H^0(E_k)=\bigoplus_{\lambda_j^{(k)}\in W_k}H^0(E_k)_{\lambda_j^{(k)}},\quad H^0(E_k)_{\lambda_j^{(k)}}:=\{s\in H^0(E_k)|\widetilde{\mathcal{L}}^{(k)} s=\lambda_j^{(k)}\otimes s\},
\]
to define an operator $v(\widetilde{\mathcal{L}}^{(k)}):H^0(E_k)\to H^0(E_k)$ by
\[
v(\widetilde{\mathcal{L}}^{(k)})_{\mid H^0(E_k)_{\lambda_j^{(k)}}}=v(\lambda_j^{(k)})\Id_{ H^0(E_k)_{\lambda_j^{(k)}}}.
\]
Let $\hat{v}$ denote the Fourier transform of $v$. Using the Fourier inversion formula, we get an alternative formulation of the above operator $v(\widetilde{\mathcal{L}}^{(k)}):H^0(E_k)\to H^0(E_k)$,
\[
v(\widetilde{\mathcal{L}}^{(k)})=\int_\tor\hat{v}(\xi) e^{i\widetilde{\mathcal{L}}^{(k)}_\xi}\underline{d\xi}.
\]

By, \cite[Theorem 8.2]{BGV04} (see also \cite[Remark 8.9.1]{Gau10} for the case $E=\mathcal{O}$) the equivariant Hirzebruch--Riemann--Roch formula reads
\begin{equation}\label{eq:RiemRoch-Equiv}
	\sum_{j=0}^n (-1)^j\tr_{\mid H^{j}(E_k)}\left(e^{k\widetilde{\mathcal{L}}_{\bullet}^{(k)}}\right)=\int_X\ch^{\T}(E_k)\Todd^\T(X).
\end{equation}
By the Kodaira--Nakano--Akizuki vanishing theorem (see e.g. \cite{Kob14}) there is an integer $k_0$ such that for all $k\geq k_0$ we have
\[
H^j(E_k)=\{0\},\quad \forall j\geq 1.
\]
Thus, for \(k\geq k_0\),
\begin{equation}\label{eq:RiemRoch-Equiv-1}
	\tr_{\mid H^{0}(E_k)}\left(e^{k\widetilde{\mathcal{L}}_{\bullet}^{(k)}}\right)=\int_X\ch^{\T}(E_k)\Todd^\T(X).
\end{equation}

\begin{defn}
	The $v$-weighted Euler characteristic of $E_k$ is defined by
	\[\chi_v(E_k):=\tr_{\mid H^{0}(E_k)}\left(v(\widetilde{\mathcal{L}}_{\bullet}^{(k)})\right),\quad\forall k\geq k_0.
	\]
\end{defn}
Unfortunately, this definition depends on the choice of extension of \(v\) to all of \(\mathfrak{t}^*\). However, we will now show that there is an asymptotic expansion of this invariant as \(k\to\infty\), which does not depend on the choice of extension---the intuition for this comes from the discussion above, where we note the convex hull of \(W_k\) converges to \(P_L\) as \(k\to\infty\). To compute the expansion, we use the equivariant Hirzebruch--Riemann--Roch theorem. By \eqref{eq:RiemRoch-Equiv-1}, for all $\xi\in \tor$ we have
\[
\tr_{\mid H^{0}(E_k)}\left(e^{i\widetilde{\mathcal{L}}_{\xi}^{(k)}}\right)=\int_X(\ch^{\T}(E_k)\Todd^\T(X))\left(\frac{i}{k}\xi \right)
\]
Integrating both sides against the measure $\hat{v}(\xi)\underline{d\xi}$ on $\tor$, we obtain (using \(F_E+\Phi_E:=\frac{i}{2\pi}(F_h+\Phi_h)\) for notational simplicity),
\[
\begin{split}
	\chi_v(E_k)=&\int_\tor\int_X(\ch^{\T}(E_k) \Todd^\T(X))\left(\frac{i}{k}\xi \right)\hat{v}(\xi)\underline{d\xi}\\
	=&\int_\tor\int_X \tr\left( e^{F_{E}+\frac{i}{k}\Phi^\xi_{E}+(k\omega+i \mu^\xi)\Id_E}\right)\left(1+\frac{1}{2}(\Ric(\omega)+\frac{i}{k}\Delta(\mu^\xi))+\cdots\right)\hat{v}(\xi)\underline{d\xi}\\
	=&\int_\tor\int_X \tr\left( e^{F_{E}+\frac{i}{k}\Phi^\xi_{E}}\right)e^{k\omega+i \mu^\xi}\left(1+\frac{1}{2}(\Ric(\omega)+\frac{i}{k}\Delta(\mu^\xi))+\cdots\right)\hat{v}(\xi)\underline{d\xi}\\
	=&\int_\tor\int_X \tr\left( e^{F_{E}+\frac{i}{k}\Phi^\xi_{E}}\right)\left(1+\frac{1}{2}(\Ric(\omega)+\frac{i}{k}\Delta(\mu^\xi))+\cdots\right)\left(\sum_{j=0}^n k^j\omega^{[j]}\right)e^{i \mu^\xi}\hat{v}(\xi)\underline{d\xi}\\
	=&k^n\int_\tor\int_X \tr(\Id_E)\hat{v}(\xi)e^{i\mu^\xi}\omega^{[n]}\underline{d\xi}\\
	&+k^{n-1}\int_\tor\int_X \left[\vphantom{\frac{\tr(\Id_E)}{2}}\tr( F_E\wedge\omega^{[n-1]}+ i\Phi^\xi_E\omega^{[n]})\right.\\
	& \left. \quad\quad\quad\quad\quad\quad\quad\quad\quad\quad\quad+ \frac{\tr(\Id_E)}{2}(\Ric(\omega)\wedge\omega^{[n-1]}+ i\Delta_\omega(\mu^\xi)\omega^{[n]} )  \right]\hat{v}(\xi)e^{i\mu^\xi}\underline{d\xi}\\
	&+k^{n-2}(\cdots)\\
	=&k^n\rk(E)\int_X v(\mu)\omega^{[n]}\\
	&+k^{n-1}\int_X\left[\tr(v(\mu)\Lambda_\omega F_E+\langle \Phi_E, dv(\mu)\rangle)+\frac{\rk(E)}{2}(v(\mu)\Scal_\omega+\langle \Delta_\omega(\mu), dv(\mu)\rangle)\right]\omega^{[n]}\\
	&+k^{n-2}(\cdots)\\
	=&k^n\rk(E)\int_X v(\mu)\omega^{[n]}
	+k^{n-1}\int_X \left(\tr(\Lambda_{\omega,v}(F_E+\Phi_E))+\frac{\rk(E)}{2}\Scal_v(\omega)\right)\omega^{[n]}+\cdots\\
	=&\rk(E)\left(v(c_1(L)_\T)\right)k^n+ \left(\left(c_1(E)_\T\cdot v(c_1(L)_\T)\right)+\frac{\rk(E)}{2}\left(c_1(X)_\T\cdot v(c_1(L)_\T)\right)\right)k^{n-1}+\cdots
\end{split}
\]
where $\Scal_v(\omega)$ is the weighted scalar curvature of $\omega$ (see \cite[Definition 1]{Lah19}) and we use that
\[
\int_X\Scal_v(\omega)\omega^{[n]}=\int_X\left[v(\mu)\Scal_\omega+\langle \Delta_\omega(\mu),  dv(\mu)\rangle\right]\omega^{[n]}=\left(c_1(X)_\T\cdot v(c_1(L)_\T)\right).
\]
To recapitulate,
\begin{prop}\label{prop:Euler-polynom}
	The $v$-weighted Euler characteristic of $E_k$ admits an asymptotic expansion for $k\gg 0$,
	\[
	\frac{\chi_v(E_k)}{\rk(E)}=\left(v(c_1(L)_\T)\right)k^n+\left(\mu_v(E)+\frac{1}{2}\left(c_1(X)_\T\cdot v(c_1(L)_\T)\right)\right)k^{n-1}+\cdots
	\]
	where $\mu_v(E)$ is the $v$-weighted slope of $E$.
\end{prop}
For any $\T$-equivariant coherent saturated subsheaf $\F\subset E$, the $v$-weighted Euler characteristic of $\F$ is similarly defined by
\[
\chi_v(\F_k):=\tr_{|H^0(\F_k)}\left(v(\widetilde{\mathcal{L}}^{(k)}_{\bullet})\right),\quad \forall k\gg 0
\]
where $\widetilde{\mathcal{L}}^{(k)}_{\xi}:H^0(\F_k)\to H^0(\F_k)$ is the induced infinitesimal action of elements $\xi\in\tor$. As in Proposition \ref{prop:Euler-polynom}, the $\T$-equivariant Hirzebruch--Riemann--Roch formula for $\F_k$ gives the asymptotic expansion,
\[
\frac{\chi_v(\F_k)}{\rk(\F)}=\left(v(c_1(L)_\T)\right)k^n+ \left(\mu_v(\F)+\frac{1}{2}\left(c_1(X)_\T\cdot v(c_1(L)_\T)\right)\right)k^{n-1}+\cdots
\]
which holds for all $k\gg 0$.
\begin{defn}
	Let $E$ be a $\T$-equivariant holomorphic vector bundle over a polarized K\"ahler $\T$-manifold $(X,L)$ endowed with the action of a torus $\T\subset\Aut(X,L)$. We say that $E$ is:
	\begin{enumerate}
		\item \emph{\(v\)-Gieseker $L$-semistable} if for all non-zero \(\T\)-equivariant coherent saturated subsheaves \(\F\subset E\), \[\frac{\chi_v(\F_k)}{\rk(\F)}\leq \frac{\chi_v(E_k)}{\rk(E)}\quad \forall k\gg 0;\]
		\item \emph{\(v\)-Gieseker $L$-stable} if for all proper non-zero \(\T\)-equivariant coherent saturated subsheaves \(\F\subset E\), \[\frac{\chi_v(\F_k)}{\rk(\F)}< \frac{\chi_v(E_k)}{\rk(E)}\quad \forall k\gg 0.\]
	\end{enumerate}
\end{defn}
The following is a straightforward consequence of the polynomial expansions of the weighted Euler characteristic.

\begin{prop}
	If $E$ is $v$-slope stable, then $E$ is \(v\)-Gieseker stable. If \(E\) is \(v\)-Gieseker semistable, then \(E\) is \(v\)-slope semistable.
\end{prop}
\begin{proof}
	Indeed, for any proper non-zero \(\T\)-equivariant coherent saturated subsheaf \(\F\subset E\) we have
	\[
	\frac{\chi_v(E_k)}{\rk(E)}-\frac{\chi_v(\F_k)}{\rk(\F)}=\left(\mu_v(E)-\mu_v(\F) \right)k^{n-1}+\mathrm{O}(k^{n-2}). \qedhere
	\]
\end{proof}
When the weight function $v=1$ is constant, the weighted Euler characteristic \(\chi_v(E_k)=\chi(E_k)\) reduces to the usual Euler characteristic, and weighted Gieseker stability reduces to usual Gieseker stability. Gieseker stability is known to be equivalent to the existence of balanced metrics induced from a sequence of Grassmannian embeddings of the polarised manifold $(X,L)$ (see \cite{wang02, wang05, hashi22}). We leave the differential geometric counterpart of weighted Gieseker stability for future work.

\section{Weighted Kobayashi--L{\"u}bke inequality}\label{sec:Lubke}

In this section we prove a weighted analogue of the Kobayashi--L{\"u}bke inequality \cite{Bog78, Gie79, Lub82} for vector bundles admitting weighted Hermite--Einstein metrics. We recall the usual Kobayashi--L{\"u}bke inequality is stated as follows: if \(E\) admits a Hermite--Einstein metric, then \[(r-1)(c_1(E)^2\alpha^{n-2}) \leq 2r(c_2(E)\alpha^{n-2}),\] where \(r\) is the rank of \(E\), \(\alpha = [\omega]\) denotes the K{\"a}hler class, and the expressions \((c_1(E)^2\alpha^{n-2})\) and \((c_2(E)\alpha^{n-2})\) denote ordinary (i.e. non-equivariant) intersection numbers.

Given the equivariant intersection theory, it is fairly straightforward to guess the correct equivariant analogue of this inequality---indeed it is simply \[(r-1)(c_1(E)_\T^2\cdot v(\alpha_\T)) \leq 2r(c_2(E)_\T\cdot v(\alpha_\T)).\] However, we will also require the weight function \(v\) to satisfy a certain Hessian inequality.

\begin{thm}\label{weighted-Lubke}
	Suppose that \(E\) admits a \(v\)-weighted Hermite--Einstein metric, and that \(v\) satisfies the inequality \begin{equation}
		\Hess(v) - \frac{n+1}{n}\frac{dv\otimes dv}{v} \leq 0.
	\end{equation} Then \[(r-1)(c_1(E)_\T^2\cdot v(\alpha_\T)) \leq 2r(c_2(E)_\T\cdot v(\alpha_\T)).\] Furthermore, equality holds if and only if \(E\) is projectively flat.
\end{thm}

Before launching into the proof, we shall verify that the inequality is satisfied by all of the important examples of weighted Hermite--Einstein metrics. Many of these examples stipulate that \(v\) is of the form \(v(\mu) = \tilde{v}(\langle\mu,\xi\rangle)\) for some \(\xi\in\tor\). One observes that the inequality is then equivalent to \[\tilde{v}''(x) - \frac{n+1}{n}\frac{\tilde{v}'(x)^2}{\tilde{v}(x)}\leq 0.\]

\begin{enumerate}
	\item When \(\tilde{v}=1\) (Hermite--Einstein case), the inequality reduces to \(0\leq0\);
	\item When \(\tilde{v}(x) = e^x\) (KR-soliton case), the inequality reduces to \(1\leq \frac{n+1}{n}\);
	\item When \(\tilde{v}(x) = x^{-n}\) (Sasaki case), the inequality reduces to \(0\leq0\).
\end{enumerate}

Lastly, we consider the polynomial weight \[p(\mu):=\prod_{\alpha=1}^{N}(c_\alpha+\langle\mu, p_\alpha\rangle)^{n_\alpha}\] from Section \ref{sec:poly-weights}. A straightforward computation gives
\[
{\rm Hess}(\log p)(\mu) = -\sum_{\alpha=1}^{N}\left(\frac{n_\alpha}{(\langle\mu,p_\alpha\rangle+c_\alpha)^2}\right)p_\alpha\otimes p_\alpha\leq 0.
\] Since \[\Hess(p)-\frac{n+1}{n}\frac{dp\otimes dp}{p}\leq \Hess(p)-\frac{dp\otimes dp}{p} = \Hess(\log p)\leq 0,\] we see that \(p\) also satisfies the desired inequality.

\begin{proof}[Proof of Theorem \ref{weighted-Lubke}]
	Using the identity ${\rm ch}_2(E)_{\T}=\frac{1}{2}(c_1(E)^2_{\T}-2c_2(E)_{\T})$, the weighted Kobayashi--L{\"u}bke inequality is equivalent to
	\begin{equation}\label{eq:Lubk}
		2r\left({\rm ch}_2(E)_{\T}\cdot v(\alpha_{\T})\right)\leq \left(c_1(E)_{\T}^2\cdot v(\alpha_{\T})\right).
	\end{equation}
	
	We first compute these equivariant intersection numbers using equivariant representatives. Letting \(h\) be the weighted Hermite--Einstein metric, we define \(F_E:=\frac{i}{2\pi}F_h\) and \(\Phi_E:=\frac{i}{2\pi}\Phi_h\), and compute
	\[
	\begin{split}
		&\left(c_1(E)_{\T}^2\cdot v(\alpha_{\T})\right) \\
		=&\int_{\tor}\left(c_1(E)_{\T}^2\cdot e^{\alpha_{\T}}\right)(i\xi)\hat{v}(\xi) \underline{d\xi}\\
		=& \int_X\int_{\tor}\left(\left(\tr(F_E+\Phi_E)\right)^2 \cdot e^{\omega+\mu}\right)(i\xi) \hat{v}(\xi) \underline{d\xi}\\
		=& \int_X\int_{\tor}\left( (\tr(F_E))^2+(\tr(i\langle\Phi_E,\xi\rangle))^2+2\tr(F_E)\tr(i\langle\Phi_E,\xi\rangle)\right)e^{\omega
			+i\langle\mu,\xi\rangle} \hat{v}(\xi) \underline{d\xi}\\
		=& \int_X\left(\int_{\tor}  \hat{v}(\xi) e^{i\langle\mu,\xi\rangle}\underline{d\xi}\right) (\tr(F_E))^2  \wedge \omega^{[n-2]}\\
		&+2\int_X\left(\int_{\tor}\tr(i\langle\Phi_E,\xi\rangle)\hat{v}(\xi) e^{i\langle\mu,\xi\rangle}  \underline{d\xi}\right)\tr(F_E) \wedge \omega^{[n-1]}\\
		&+ \int_X\left(\int_{\tor} (\tr(i\langle\Phi_E,\xi\rangle))^2 \hat{v}(\xi) e^{i\langle\mu,\xi\rangle}\underline{d\xi}\right)\omega^{[n]}\\
		=& \int_X  v(\mu)  \tr(F_E)^2  \wedge \omega^{[n-2]}+2\int_X  \langle\tr(\Phi_E),dv(\mu)\rangle \tr(F_E) \wedge \omega^{[n-1]}\\
		&+ \int_X\langle \tr(\Phi_E)\otimes \tr(\Phi_E), \Hess(v)(\mu)\rangle \omega^{[n]}.
	\end{split}
	\]
	In the last equality we use \eqref{eq:FourierD}, which in particular gives
	\[
	\begin{split}
		&\int_X\left(\int_{\tor} (\tr(i\langle\Phi_E,\xi\rangle))^2 \hat{v}(\xi) e^{i\langle\mu,\xi\rangle}\underline{d\xi}\right)\omega^{[n]}\\
		=& \int_X\left(\int_{\tor} \langle\tr(i\Phi_E)\otimes \tr(i\Phi_E), \xi\otimes\xi\rangle \hat{v}(\xi) e^{i\langle\mu,\xi\rangle}\underline{d\xi}\right)\omega^{[n]}\\
		=& \int_X \left\langle\tr(\Phi_E)\otimes \tr(\Phi_E),i^2\int_{\tor} ( \xi\otimes\xi ) \hat{v}(\xi) e^{i\langle\mu,\xi\rangle}\underline{d\xi}\right\rangle \omega^{[n]}\\
		=& \int_X\langle \tr(\Phi_E)\otimes \tr(\Phi_E), \Hess(v)(\mu)\rangle \omega^{[n]}.
	\end{split}
	\]
	Thus we obtain
	\begin{equation}\label{eq:A}
		\begin{split}
			\left(c_1(E)_{\T}^2\cdot v(\alpha_{\T})\right)
			=& \int_X  v(\mu)  \tr(F_E)^2  \wedge \omega^{[n-2]} +2\int_X  \langle\tr(\Phi_E),dv(\mu)\rangle \tr(F_E)\wedge \omega^{[n-1]}\\
			&+ \int_X\langle\tr(\Phi_E)\otimes \tr(\Phi_E), \Hess(v)(\mu)\rangle \omega^{[n]}.
		\end{split}
	\end{equation}
	A similar calculation gives
	\begin{equation}
		\begin{split}\label{eq:B}
			2\left({\rm ch}_2(E)_{\T}\cdot v(\alpha_{\T})\right)=& \int_X  v(\mu)  \tr(F_E^2)  \wedge \omega^{[n-2]}+ 2\int_X  \langle\tr(F_E\Phi_E),dv(\mu)\rangle \wedge \omega^{[n-1]}\\
			&+ \int_X\langle \tr(\Phi_E\circ \Phi_E), \Hess(v)(\mu)\rangle \omega^{[n]}.\\
		\end{split}
	\end{equation}
	
	We reformulate the following term from \eqref{eq:A}:
	\[
	\begin{split}
		&2\int_X  \langle\tr(\Phi_E),dv(\mu)\rangle \tr(F_E)\wedge \omega^{[n-1]}\\
		=&2\int_X  \langle\tr(\Phi_E),dv(\mu)\rangle \tr(\Lambda_\omega F_E)\omega^{[n]}\\
		=&\int_X  \tr\left(c_v\Id_E-v(\mu)\Lambda_\omega F_E \right)\tr(\Lambda_\omega F_E) \omega^{[n]}\\
		&+ \int_X  \langle\tr(\Phi_E),dv(\mu)\rangle \tr\left( \frac{c_v}{v(\mu)}\Id_E-\langle \Phi_E,d(\log v)(\mu)\rangle\right)\omega^{[n]}\\
		=&-\int_X (\tr(\Lambda_\omega F_E))^2v(\mu) \omega^{[n]}
		+rc_v\int_X\tr(\Lambda_\omega F_E) \omega^{[n]}\\
		&+r c_v\int_X  \langle\tr(\Phi_E),d(\log v)(\mu)\rangle  \omega^{[n]}
		-\int_X  \langle\tr(\Phi_E),d(\log v)(\mu)\rangle^2 v(\mu)\omega^{[n]}.
	\end{split}
	\]
	The analogous term from \eqref{eq:B} is given by
	\[
	\begin{split}
		&2\int_X  \langle\tr(F_E\Phi_E),dv(\mu)\rangle \wedge \omega^{[n-1]}\\
		=&2\int_X  \tr((\Lambda_\omega F_E)\langle\Phi_E,d(\log v)(\mu)\rangle) v(\mu)\omega^{[n]}\\
		=&-\int_X  \tr((\Lambda_\omega F_E)^2) v(\mu)\omega^{[n]}  + c_v\int_X  \tr(\Lambda_\omega F_E) \omega^{[n]} \\
		&+c_v\int_X  \langle\tr(\Phi_E),d(\log v)(\mu)\rangle \omega^{[n]} - \int_X \tr(\langle \Phi_E, d(\log v)(\mu)\rangle^2) v(\mu)\omega^{[n]}.
	\end{split}
	\]
	The terms involving $c_v$ cancel out in \eqref{eq:Lubk}. Hence the inequality \eqref{eq:Lubk} is equivalent to
	\begin{equation}\label{eq:Lubk1}
		\begin{split}
			&r\int_Xv(\mu)\tr(F_E^2)\wedge\omega^{[n-2]} - r\int_X  \tr((\Lambda_\omega F_E)^2) v(\mu)\omega^{[n]} \\
			&+r\int_X\langle \tr(\Phi_E\circ \Phi_E), {\rm Hess}(v)(\mu)\rangle \omega^{[n]} - r\int_X \tr(\langle \Phi_E, d(\log v)(\mu)\rangle^2) v(\mu)\omega^{[n]}\\
			\leq& \int_Xv(\mu)\tr(F_E)^2\wedge\omega^{[n-2]} - \int_X  \tr((\Lambda_\omega F_E))^2 v(\mu)\omega^{[n]} \\
			&+\int_X\langle \tr(\Phi_E)\otimes \tr(\Phi_E), {\rm Hess}(v)(\mu)\rangle \omega^{[n]} -\int_X  \langle\tr(\Phi_E),d(\log v)(\mu)\rangle^2 v(\mu)\omega^{[n]}
		\end{split}
	\end{equation}
	Using the following identities
	\begin{equation}\label{eq:dvdv}
		\begin{split}
			\langle\tr(\Phi_E),d(\log v)(\mu)\rangle^2=& \langle\tr(\Phi_E)\otimes \tr(\Phi_E),d(\log v)(\mu)\otimes d(\log v)(\mu)\rangle\\
			\langle \Phi_E, d(\log v)(\mu)\rangle^2=& \langle\Phi_E\circ\Phi_E,d(\log v)(\mu)\otimes d(\log v)(\mu)\rangle,
		\end{split}
	\end{equation}
	the inequality \eqref{eq:Lubk1} is equivalent to
	\begin{equation}
		\begin{split}\label{Eq:Lubk2}
			0\leq &\int_X\left\langle\tr(\Phi_E)\otimes \tr(\Phi_E)-r\,\tr(\Phi_E\circ \Phi_E), \left[ {\rm Hess}(v)(\mu)-\frac{dv(\mu)\otimes dv(\mu)}{v(\mu)} \right]\right\rangle \omega^{[n]} \\
			&+\int_Xv(\mu)\left([(\tr\,F_E)^2-r\tr( F_E^2)]\wedge\omega^{[n-2]}+[r\tr((\Lambda_\omega F_E)^2)-(\tr\,\Lambda_\omega F_E)^2]\omega^{[n]}\right),
		\end{split}
	\end{equation}
	At this point we introduce the endomorphism valued 2-form $\tilde{F}_E$ given by
	\begin{equation}\label{eq:tilde-FE}
		\tilde{F}_E:=F_E+\frac{1}{n}\langle \Phi_E, d(\log v)(\mu)\rangle\otimes\omega
	\end{equation}
	and notice that the $v$-weighted Hermite--Einstein equation is the same as $v(\mu)\Lambda_\omega\tilde{F}_E=c_v\Id_E$. We write each term in the integral from the second line of \eqref{Eq:Lubk2} in terms of $\tilde{F}_E$:
	\[
	\begin{split}
		(\tr\,F_E)^2\wedge\omega^{[n-2]}=&(\tr\,\tilde{F}_E)^2\wedge\omega^{[n-2]}-\frac{2r(n-1)c_v}{nv(\mu)}\langle \tr(\Phi_E), d(\log v)(\mu)\rangle\omega^{[n]}\\
		&+\frac{n-1}{n} \langle\tr(\Phi_E)\otimes\tr(\Phi_E), d(\log v)(\mu)\otimes d(\log v)(\mu)\rangle\omega^{[n]},\\
		\tr(F_E^2)\wedge\omega^{[n-2]}=& \tr(\tilde{F}_E^2)\wedge\omega^{[n-2]}-\frac{2(n-1)c_v}{nv(\mu)}\tr\left(\langle \Phi_E,  d(\log v)(\mu)\rangle\right)\omega^{[n]}\\
		&+\frac{n-1}{n}\tr(\langle\Phi_E\circ\Phi_E, d(\log v)(\mu)\otimes  d(\log v)(\mu)\rangle )\omega^{[n]},\\
		(\tr \Lambda_\omega F_E)^2=& (\tr(\Lambda_\omega \tilde{F}_E))^2-\frac{2r c_v}{v(\mu)} \langle \tr(\Phi_E),  d(\log v)(\mu)\rangle\\
		&+\langle\tr(\Phi_E)\otimes \tr(\Phi_E),  d(\log v)(\mu)\otimes  d(\log v)(\mu)\rangle,\\
		\tr((\Lambda_\omega F_E)^2)=&\tr((\Lambda_\omega \tilde{F}_E)^2)-\frac{2c_v}{v(\mu)}\tr\left(\langle \Phi_E,  d(\log v)(\mu)\rangle\right)\\
		&+ \tr\left(\langle\Phi_E\circ\Phi_E, d(\log v)(\mu)\otimes  d(\log v)(\mu)\rangle\right),
	\end{split}
	\]
	where we use the $v$-weighted Hermite--Einstein condition $v(\mu)\Lambda_\omega\tilde{F}_E=c_v\Id_E$ and \eqref{eq:dvdv} to simplify the quadratic terms involving $\Phi_E$. Substituting back in \eqref{Eq:Lubk2} we get the equivalent inequality
	\begin{equation}
		\begin{split}\label{Eq:Lubk3}
			0\leq &\int_X\left\langle\tr(\Phi_E)\otimes \tr(\Phi_E)-r\,\tr(\Phi_E\circ \Phi_E), \left[ {\rm Hess}(v)(\mu)-\frac{n+1}{n}\frac{dv(\mu)\otimes dv(\mu)}{v(\mu)} \right]\right\rangle \omega^{[n]} \\
			&+\int_Xv(\mu)\left([(\tr\,\tilde{F}_E)^2-r\tr(\tilde{F}_E^2)]\wedge\omega^{[n-2]}+[r\tr((\Lambda_\omega \tilde{F}_E)^2)-(\tr\,\Lambda_\omega \tilde{F}_E)^2]\omega^{[n]}\right).
		\end{split}
	\end{equation}
	We consider the integral from the second line of \eqref{Eq:Lubk3}. Using the following formula
	\begin{equation}\label{eq:classic}
		\alpha\wedge\beta\wedge\omega^{[n-2]}=(\Lambda_\omega\alpha\Lambda_\omega\beta-\langle\alpha,\beta\rangle_\omega)\omega^{[n]}
	\end{equation}
	which holds for any real $(1,1)$-forms $\alpha, \beta$, we get
	\[
	\tilde{F}_E^2\wedge\omega^{[n-2]}=\left((\Lambda_\omega \tilde{F}_E)\circ(\Lambda_\omega \tilde{F}_E)-|\tilde{F}_E|_\omega^2\right)\otimes\omega^{[n]}
	\]
	where $|\tilde{F}_E|^2_\omega$ is the endomorphism given by the product $\langle\cdot,\cdot\rangle_\omega$ on the differential form component of $\tilde{F}_E$ and composition on the endomorphism component of $\tilde{F}_E$. Taking the trace of the endomorphism component of both sides gives
	\begin{equation}\label{eq:trF1}
		\tr(\tilde{F}_E^2)\wedge\omega^{[n-2]}=\left(\tr((\Lambda_\omega \tilde{F}_E)^2)-\tr(|\tilde{F}_E|_\omega^2)\right)\omega^{[n]}.
	\end{equation}
	Applying \eqref{eq:classic} for $\alpha=\beta=\tr(\tilde{F}_E)$ gives
	\begin{equation}\label{eq:trF2}
		(\tr\tilde{F}_E)^2\wedge\omega^{[n-2]}=\left( (\Lambda_\omega \tr \tilde{F}_E )^2-|\tr(\tilde{F}_E)|_\omega^2\right)\omega^{[n]}.
	\end{equation}
	combining \eqref{eq:trF1} and \eqref{eq:trF2} we obtain
	\[
	\begin{split}
		&r\left(\tr(( \tilde{F}_E)^{2})\wedge\omega^{[n-2]}-\tr((\Lambda_\omega \tilde{F}_E)^2) \omega^{[n]}\right)- \left((\tr(\tilde{F}_E))^2\wedge\omega^{[n-2]}-(\tr(\Lambda_\omega \tilde{F}_E))^2\omega^{[n]}\right)\\
		=&-r\left(\tr(|\tilde{F}_E|_\omega^2)-\frac{1}{r}|\tr(\tilde{F}_E)|_\omega^2\right)\omega^{[n]}\\
		=&-r\tr\left( |\tilde{F}_E|_\omega^2 -\frac{2}{r}\langle \tilde{F}_E,\tr(\tilde{F}_E)\rangle+ \frac{1}{r^2}|\tr(\tilde{F}_E)|_\omega^2 \otimes\Id_E\right)\\
		=&-r\tr \left|\tilde{F}_E-\frac{1}{r}\tr(\tilde{F}_E)\otimes\Id_E\right|_\omega^2.
	\end{split}
	\]
	So the inequality \eqref{Eq:Lubk3} has become \[
	\begin{split}
		0\leq &\int_X\left\langle\tr(\Phi_E)\otimes \tr(\Phi_E)-r\,\tr(\Phi_E\circ \Phi_E), \left[ {\rm Hess}(v)(\mu)-\frac{n+1}{n}\frac{dv(\mu)\otimes dv(\mu)}{v(\mu)} \right]\right\rangle \omega^{[n]} \\
		&+r\int_Xv(\mu)\tr \left|\tilde{F}_E-\frac{1}{r}\tr(\tilde{F}_E)\otimes\Id_E\right|_\omega^2 \omega^{[n]}.
	\end{split}
	\]
	We have assumed that
	\begin{equation}\label{v-ineq}
		{\rm Hess}(v)(\mu) - \frac{n+1}{n}\frac{dv(\mu)\otimes dv(\mu)}{v(\mu)}\leq 0.
	\end{equation} 
	Choosing (pointwise on \(X\)) a basis  $(\xi_a)_{1\leq a\leq \ell}$ such that ${\rm Hess}(v)(\mu)-\frac{n+1}{n}\frac{dv(\mu)\otimes dv(\mu)}{v(\mu)}$ is diagonal, it is enough to show that
	\[
	\tr(\Phi_E^{\xi_a})^2-  r\tr(\Phi_E^{\xi_a}\circ\Phi_E^{\xi_a})\leq 0,
	\]
	Since $\tr(H)^2\leq r\tr(H^2)$ holds for arbitrary Hermitian matrices, noting that  $\Phi_E^{\xi_a}$ is Hermitian gives the desired inequality. Hence we have proven the inequality \eqref{eq:Lubk}.
	
	Note that the weighted Kobayashi--L\"ubke inequality is an equality if and only if $\tilde{F}_E=\frac{1}{r}\tr(\tilde{F}_E)\otimes\Id_E$ and $\Phi_E=\lambda\Id_E$, in which case $E$ is projectively flat.
\end{proof}

\begin{rem}
	In the above proof, if one does not introduce the endomorphism valued 2-form $\tilde{F}_E$ given by \eqref{eq:tilde-FE}, then one can easily deduce from \eqref{Eq:Lubk2} the inequality
	\[
	\begin{split}
		0\leq &\int_X\left\langle\tr(\Phi_E)\otimes \tr(\Phi_E)-r\,\tr(\Phi_E\circ \Phi_E),  {\rm Hess}(\log v)(\mu)  \right\rangle v(\mu)\omega^{[n]} \\
		&+r\int_Xv(\mu)\tr \left|F_E-\frac{1}{r}\tr(F_E)\otimes\Id_E\right|_\omega^2 \omega^{[n]},
	\end{split}
	\]
	which shows that the weighted Kobayashi--L\"ubke inequality also holds for the smaller class of weight functions $v\in C^\infty(\Pol,\mathbb{R}_{>0})$ satisfying the inequality
	\begin{equation}\label{eq:logv-concave}
		\Hess(\log v)\leq 0.
	\end{equation}
	Notice that the condition \eqref{eq:logv-concave} is clearly satisfied by the weight  $v(\mu)=e^{\langle\mu,\xi\rangle}$ and also by the polynomial weights $p(\mu)=\prod_{\alpha=1}^{N}(c_\alpha+\langle\mu, p_\alpha\rangle)^{n_\alpha}$ from Section \ref{sec:poly-weights}. 
\end{rem}
\begin{rem}
	Let $v\in C^\infty(\Pol,\mathbb{R}_{>0})$ be a weight function and $\Delta(v)$ the intersection number
	\[
	\Delta(v):=\left(c_1(E)_{\T}^2\cdot v(\alpha_{\T})\right)-2r\left({\rm ch}_2(E)_{\T}\cdot v(\alpha_{\T})\right).
	\]
	Then, from the calculations in the proof of Theorem \ref{weighted-Lubke} one can easily see that for any $\T$-invariant Hermitian metric $h$ on $E$ the following equality holds
	\[
	\begin{split}
		\Delta(v)= &\int_X\left\langle\tr(\Phi_E)\otimes \tr(\Phi_E)-r\,\tr(\Phi_E\circ \Phi_E),  {\rm Hess}(\log v)(\mu)\right\rangle v(\mu)\omega^{[n]} \\
		&+\frac{r}{4\pi^2}\int_X v(\mu)\tr \left|F_h^\circ\right|_\omega^2 \omega^{[n]}- r\int_X|K_v(h)^{\circ}|^2 \frac{\omega^{[n]}}{v(\mu)}.
	\end{split}
	\]
	where $K_v(h):=\frac{i}{2\pi}\Lambda_{\omega,v}(F_h+\Phi_h)$ is the $v$-weighted mean curvature of $(E,h)$ and $T^\circ$ is the endomorphism trace free part, defined by $T^\circ=T-\frac{1}{r}\tr_E(T)\otimes\Id_E$ for all $T\in\Omega^*(\End(E))$. Thus, for weight functions such that $\Hess(\log v)\leq 0$, the following inequality holds for any $\T$-invariant Hermitian metric $h$,
	\begin{equation}\label{ineq:YM}
		\int_X|K_v(h)^{\circ}|^2 \frac{\omega^{[n]}}{v(\mu)} +\frac{\Delta(v)}{r} \geq \frac{1}{4\pi^2}{\rm YM}_v(h).
	\end{equation}
	where ${\rm YM}_v(h)$ is the $v$-weighted Yang--Mills energy of a $\T$-invariant Hermitian metric $h$, given by
	\[
	{\rm YM}_v(h):=\int_Xv(\mu)\tr \left|F_h^\circ\right|_\omega^2 \omega^{[n]}
	\]
	For the special class of weights $\{e^{\langle-,\xi\rangle}\mid\xi\in\tor\}$ the inequality \eqref{ineq:YM} becomes an equality
	\[
	\int_X|K_{e^{\langle-,\xi\rangle}}(h)^{\circ}|^2 e^{-\langle\mu, \xi\rangle}\omega^{[n]} +\frac{\Delta(\xi)}{r} = \frac{1}{4\pi^2}{\rm YM}_{e^{\langle-,\xi\rangle}}(h),
	\]
	which holds for all $\T$-invariant Hermitian metrics $h$. Consequently,   the exponentially weighted Yang--Mills energy ${\rm YM}_{e^{\langle-,\xi\rangle}}$ admits a topological lower bound,
	\[
	\frac{4\pi^2}{r}\Delta(\xi) \leq {\rm YM}_{e^{\langle\xi,-\rangle}}(h).
	\]
	Furthermore, if $h$ is a $e^{\langle-,\xi\rangle}$-weighted Hermite--Einstein metric, then the weighted Yang--Mills energy of $h$ attains the topological minimum, namely
	${\rm YM}_{e^{\langle-,\xi\rangle}}(h)
	=\frac{4\pi^2}{r}\Delta(\xi)$.
\end{rem}

As the holomorphic tangent bundle of a K\"ahler--Ricci soliton is, in particular, weighted Hermite--Einstein (see section \ref{sec3.3}), then the corresponding equivariant first and second Chern classes are constrained by the exponentially weighted Kobayashi--L\"ubke inequality.

\begin{cor}
	Let $X$ be a Fano $\T$-manifold of complex dimension $n$ admitting a K\"ahler--Ricci soliton in $c_1(X)$ with optimal vector field $\xi\in\tor$. Then,
	\[
	(n-1)\left(c_1(X)_\T^2\cdot e^{c_1(X)_\T}\right)(\xi)\leq 2n \left(c_2(X)_
	\T\cdot e^{c_1(X)_\T}\right)(\xi).
	\]
\end{cor}

\section{Extensions of tangent bundles of K{\"a}hler--Ricci solitons and bounds on weighted Ricci curvature}\label{sec:extensions}

In \cite{Tia92}, Tian showed that
the stability of certain natural extensions of the holomorphic tangent bundle can be used to bound the maximal possible positive lower bound of the Ricci curvature of K\"ahler metrics in the anticanonical polarization. In this section, we first show in Theorem \ref{thm:ExtSolitons} that certain equivariant extensions of the tangent bundle of a K{\"a}hler--Ricci soliton admit natural metrics. These metrics, despite not being weighted Hermite--Einstein, are then used to generalise Tian's observation to lower bounds of the ``weighted" Ricci curvature ${\rm Ric}(\omega)-\frac{i}{2\pi}\d\db\mu^{\xi}$, in Theorem \ref{thm:weighted-Ricci-bound}.

\subsection{Equivariant extensions of equivariant vector bundles}

We start by recalling some natural constructions on extensions of equivariant vector bundles. Let $E_1$ and $E_2$ be $\T$-equivariant holomorphic vector bundles over a K\"ahler $\T$-manifold $(X,\omega)$ and $\Psi\in \Omega^{0,1}(\Hom(E_1,E_2))^\T$ a $\T$-invariant $\db_{\Hom(E_1,E_2)}$-closed $(0,1)$-form with values in $\Hom(E_1,E_2)$. Then the cohomology class $[\Psi]\in H^{0,1}(X,\Hom(E_1,E_2))\simeq H^{1}(X,\Hom(E_1,E_2))$ determines a $\T$-equivariant extension $E_{[\Psi]}$ of $E_1$ by $E_2$,
\[
0\to E_2\to E_{[\Psi]}\to E_1\to 0.
\]
In what follows we give the construction of a holomorphic structure on $E_{[\Psi]}$ together with a natural lift of the $\T$-action making the extension $E_{[\Psi]}$ a $\T$-equivariant bundle. Let $h_1$, $h_2$ be two $\T$-invariant metrics on $E_1$, $E_2$ respectively and $\nabla_1$, $\nabla_2$ the corresponding Chern connections. We consider the smooth complex vector bundle $E_1\oplus E_2$ endowed with the Hermitian structure $h:=h_1\oplus h_2$ and the diagonal lift of the $\T$-action with induced infinitesimal action $\begin{pmatrix}
	\mathcal{L}^{E_1}&0\\
	0&\mathcal{L}^{E_2}
\end{pmatrix}$. On $E_1\oplus E_2$ we have the $\T$-invariant connection $\nabla^\Psi$ given by
\begin{equation}\label{eq:nabla-split}
	\nabla^{\Psi}=\begin{pmatrix}
		\nabla_1 & -\Psi^{\dagger}\\
		\Psi & \nabla_2    \end{pmatrix}
\end{equation}
where $\Psi^\dagger\in\Omega^{1,0}(\Hom(E_2,E_1))^{\T}$ is defined by
\begin{equation}\label{eq:phi^dagger}
	h_{2}(\Psi(s_1),s_2)=h_1(s_1,\Psi^{\dagger}(s_2))
\end{equation}
for all $s_1\in \Gamma(E_1)$ and $s_2\in \Gamma(E_2)$. One can easily check that $\nabla^\Psi$ is a Hermitian connection. The $(0,1)$-part of $\nabla^\Psi$ gives the Cauchy--Riemann operator
\[
(\nabla^\Psi)^{(0,1)}=\begin{pmatrix}
	\db^{E_1} & 0\\
	\Psi & \db^{E_2}   \end{pmatrix}.
\]
By $\T$-invariance and $\db_{\Hom(E_1,E_2)}$-closedness of $\Psi$, the Cauchy--Riemann operator $(\nabla^\Psi)^{(0,1)}$ is $\T$-invariant and integrable: $((\nabla^\Psi)^{(0,1)})^2=0$. The $\T$-equivariant extension $E_{[\Psi]}$ of $E_1$ by $E_2$ is the $\T$-equivariant holomorphic vector bundle $(E_1\oplus E_2,h,\nabla^\Psi)$. The curvature $F_{\nabla^\Psi}$ of $(E_{[\Psi]},h)$ is given by
\[
\begin{split}
	F_{\nabla^\Psi}=\begin{pmatrix}
		F_{h_1}-\Psi^{\dagger}\wedge \Psi &-\Hom(\nabla_2,\nabla_1)\Psi^{\dagger}\\
		\Hom(\nabla_1,\nabla_2) \Psi & F_{h_2}-\Psi\wedge \Psi^{\dagger}
	\end{pmatrix},
\end{split}
\]
where $\Hom(\nabla_1,\nabla_2)$ is the connection on $\Hom(E_1,E_2)$ induced by the connections $\nabla_1$, $\nabla_2$, so that $\Hom(\nabla_1,\nabla_2)\Psi$ is given by
\[
\Hom(\nabla_1,\nabla_2)\Psi=\nabla_2\circ\Psi+\Psi\circ\nabla_1,
\] and similarly for \(\Hom(\nabla_2,\nabla_1)\).
By Proposition \ref{prop:canonical_moment_map}, the corresponding bundle momentum map $\Phi_{\nabla^{\Psi}}$ is given by
\[
\Phi_{\nabla^\Psi}=\begin{pmatrix}
	\Phi_{h_1} & -\Psi^{\dagger}\\
	\Psi &\Phi_{h_2}
\end{pmatrix}.
\]
Given weight functions $v\in C^{\infty}(\Pol,\R_{>0})$ and $w\in C^\infty(\Pol,\R)$ the $\T$-equivariant extension $(E_{[\Psi]},h,\nabla^\Psi)$ is a $(v,w)$-weighted Hermite--Einstein structure if and only if
\begin{equation}\label{eq:wHE-extension}
	\begin{cases}
		v(\mu)\Lambda_\omega\left(\Hom(\nabla_1,\nabla_2)\Psi\right)+\langle \Psi, dv(\mu)\rangle=0\\
		\frac{i}{2\pi}\Lambda_{\omega,v}\left( (F_{h_1}-\Psi^{\dagger}\wedge \Psi)+\Phi_{h_1}\right)=w(\mu)\Id_{E_1}\\
		\frac{i}{2\pi}\Lambda_{\omega,v}\left(  (F_{h_2}-\Psi\wedge \Psi^{\dagger})+\Phi_{h_2}\right)=w(\mu)\Id_{E_2}.
	\end{cases}
\end{equation}

Let $X$ be a Fano $\T$-manifold with $\omega\in c_1(X)$ a $\T$-invariant representative of the anticanonical polarization and denote by $\mathcal{O}$ the trivial line bundle over $X$ with \(\T\) acting on \(\mathcal{O}=X\times\mathbb{C}\) only through \(X\). Following \cite{Tia92}, the metric $\omega\in c_1(X)$ can be seen as section in $\Psi^\omega\in\Omega^{0,1}(\Hom(TX^{1,0},\mathcal{O}))$, indeed for any $W\in T^{0,1}X$ we define $\Psi^\omega_W=\omega(-,W)\in \Gamma(\Hom(TX^{1,0},\mathcal{O}))$. By the $\T$-invariance and the closedness of $\omega$, the section $\Psi^\omega$ is $\T$-invariant and $\db_{\Hom(TX^{1,0},\mathcal{O})}$-closed. This follows from the computation
\[
\begin{split}
	(\db_{\Hom(TX^{1,0},\mathcal{O})}\Psi^\omega)(V)=&(\Psi^\omega\circ \db_{TX^{1,0}} +\db_{\mathcal{O}}\circ \Psi^\omega)(V)\\
	=&-\omega(\db_{TX^{1,0}}V,\cdot)+\db(\omega(V,\cdot))\\
	=&-\omega(\db_{TX^{1,0}}V,\cdot)+(\db\omega)(V,\cdot) +\omega(\db_{TX^{1,0}}V,\cdot)=0.
\end{split}
\]

While we cannot solve \eqref{eq:wHE-extension}, we instead prove the following:
\begin{thm}\label{thm:ExtSolitons}
	Let $(X,\xi)$ be a K{\"a}hler--Ricci soliton with canonically normalized momentum polytope $\Pol$. Then for any \(\gamma\in\mathbb{R}\) there exists a Hermitian metric $h_E$ on the extension bundle $E_{[\gamma\Psi^\omega]}$ of $TX^{1,0}$ by the trivial line bundle with the extension class $\gamma c_1(X)$, solving the equation
	\[
	\frac{i}{2\pi}\left(\Lambda_\omega F_{h_E}+\Phi_{h_E}^\xi\right)=\begin{pmatrix}
		\left(1+\frac{1}{\pi}\mu_\xi - \frac{\gamma^2}{2\pi}e^{-\mu_\xi}\right)\Id_{TX^{1,0}}&0\\
		0& \frac{1}{2\pi}\mu_\xi+\frac{n\gamma^2}{2\pi}e^{-\mu_\xi}
	\end{pmatrix}.
	\]
	Furthermore, the $e^{\mu_\xi}$-slope of $E_{[\gamma\Psi^\omega]}$ is
	\[
	\mu_{e^{\mu_\xi}}(E_{[\gamma\Psi^\omega]})=\frac{n}{n+1}\left(e^{c_1(X)_\T}\right)(\xi).
	\]
\end{thm}

We will later choose $\gamma^2:=\frac{2\pi}{n+1}\frac{\left(e^{c_1(X)_\T}\right)(\xi)}{(c_1(X))^{[n]}}$ for the application to bounds on the weighted Ricci curvature.

\begin{proof}
	Suppose that $(X,\omega,\xi)$ is a K\"ahler--Ricci soliton and let $f,u\in C^{\infty}(X,\R)^{\T}$ be two $\T$-invariant functions. We consider the holomorphic tangent bundle $(TX^{1,0},e^f h_\omega)$ endowed with the $\T$-invariant Hermitian metric $e^f h_\omega$ conformally equivalent to the one induced by $\omega$, and the trivial line bundle $(\mathcal{O},h_\mathcal{O})$ endowed with the $\T$-invariant metric $h_{\mathcal{O}}=e^{u}$. We want to find functions $f,u\in C^{\infty}(X,\R)^{\T}$ and $\gamma\in\R$ such that the extension  $(E_{[\gamma\Psi^\omega]},e^f h_\omega\oplus h_{\mathcal{O}},\nabla^{\gamma\Psi^\omega})$ of $(TX^{1,0},e^f h_\omega)$ by $(\mathcal{O},h_\mathcal{O})$ with the extension class $[\gamma\Psi^\omega]\in H^{0,1}(X,\Hom(TX^{1,0},\mathcal{O}))$ satisfies
	\[
	\frac{i}{2\pi}\left(\Lambda_\omega F_{h_E}+\Phi_{h_E}^\xi\right)=\begin{pmatrix}
		w_1(\mu)\Id_{TX^{1,0}}&0\\
		0&w_2(\mu)
	\end{pmatrix}
	\]
	where $w_i\in C^{\infty}(\Pol,\R)$ are weight functions to be specified. By \eqref{eq:wHE-extension}, we need to find $f,u\in C^{\infty}(X,\R)^{\T}$ and $\gamma\in\R$ such that
	\begin{equation}
		\begin{cases}\label{eq:exp-sol-ext}
			\Lambda_\omega\left(\Hom(\nabla^f,\nabla^{u})\Psi^\omega\right)+\langle \Psi^\omega, \xi\rangle=0\\
			\frac{i}{2\pi}\left(\Lambda_{\omega}( F_{e^f h_\omega}-\gamma^2 (\Psi^\omega)^{\dagger}\wedge \Psi^\omega)+\Phi_{e^fh_\omega}^\xi\right)=w_1(\mu) \Id_{TX^{1,0}}\\
			\frac{i}{2\pi}\left( \Lambda_{\omega}( F_{h_\mathcal{O}}-\gamma^2\Psi^\omega\wedge (\Psi^\omega)^{\dagger})+\Phi_{h_\mathcal{O}}^\xi\right)=w_2(\mu)
		\end{cases}
	\end{equation}
	where $\nabla^f$, $\nabla^u$ are the Chern connections of $(TX^{1,0},e^fh_\omega)$ and $(\mathcal{O},h_{\mathcal{O}}=e^u)$ respectively.
	
	For $\Psi\in\Omega^{0,1}(\Hom(E_1,E_2)) =\Gamma(\Lambda^{0,1}\otimes E_1^*\otimes E_2)$ we denote the evaluation of $\Psi$ at $W\in T^{0,1}X$, $e_1\in E_1$ by
	\[
	\Psi_W(e_1)\in E_2,
	\]
	so that $\Psi_W\in \Hom(E_1,E_2)$ and $\Psi(e_1)\in\Omega^{0,1}(E_2)$. A straightforward computation in an orthonormal local frame of $(TX,J,\omega)$ gives
	\begin{equation}\label{eq:tr(phi-phi^dagger)}
		\Lambda_\omega(\Psi^{\dagger}\wedge \Psi_\omega)=-ie^{u-f}\Id_{TX^{1,0}},\quad
		\Lambda_{\omega}(\Psi\wedge \Psi^{\dagger})
		=ine^{u-f}.
	\end{equation}
	
	Plugging \eqref{eq:tr(phi-phi^dagger)} in \eqref{eq:exp-sol-ext} we get
	\begin{equation}
		\begin{cases}\label{eq:exp-sol-ext1}
			\Lambda_\omega\left(\Hom(\nabla^f,\nabla^u)\Psi\right)+\langle \Psi, \xi\rangle=0\\
			\frac{i}{2\pi}\left(\Lambda_{\omega}F_{e^fh_\omega}+\Phi_{e^fh_\omega}^\xi\right)=\left(w_1(\mu)+\frac{\gamma^2}{2\pi}e^{u-f}\right) \Id_{TX^{1,0}}\\
			\frac{i}{2\pi}\left( \Lambda_{\omega} F_{h_\mathcal{O}}+\Phi_{h_\mathcal{O}}^\xi\right)=w_2(\mu)-\frac{n\gamma^2}{2\pi}e^{u-f}.
		\end{cases}
	\end{equation}
	We consider the first equation in \eqref{eq:exp-sol-ext1}: for all $V\in TX^{1,0}$ we have
	\[
	\begin{split}
		(\Hom(\nabla^f,\nabla^u)\Psi)V=&\Psi( \nabla^f V)  +\nabla^u(\Psi(V))\\
		=&\Psi\left(\nabla^\omega V+ \d f\otimes V\right)+ d(\omega(V,\cdot)) +\partial u\wedge (\omega(V,\cdot))\\
		=&\Psi\left(\nabla^\omega V\right) - \partial f\wedge\Psi(V)+ d(\omega(V,\cdot)) +\partial u\wedge (\omega(V,\cdot))\\
		=&-\omega(\nabla^\omega V,\cdot)+ (\nabla^\omega\omega)(V,\cdot)+\omega(\nabla^\omega V,\cdot) +\partial(u-f) \wedge \omega(V,\cdot)\\
		=&\partial (u-f)\wedge\omega(V,\cdot)
	\end{split}
	\]
	where $\nabla^{u}=d+\partial u$ and $\nabla^{f}=\nabla^\omega+\partial f$ are the Chern connections of $h_\mathcal{O}=e^{u}$ of $e^fh_\omega$ respectively. It follows that \begin{align*}
		(\Lambda_\omega\left(\Hom(\nabla^f,\nabla^{u})\Psi\right)+\langle \Psi, \xi\rangle)V &= \Lambda_\omega(\d(u-f)\wedge\omega(V,\cdot))+\omega(V,\xi'') \\
		&= \d(u-f)(V)+\d\mu^\xi(V).
	\end{align*}
	Hence, the first equation in \eqref{eq:exp-sol-ext1} is equivalent to
	\begin{equation}\label{eq:u+f}
		u-f=-\mu_\xi,
	\end{equation} where without loss of generality we choose the normalising constant \(0\).
	For the second equation in \eqref{eq:exp-sol-ext1}, since $\omega$ is a K{\"a}hler--Ricci soliton then we have
	\[
	\begin{split}
		\frac{i}{2\pi}\left(\Lambda_{\omega} F_{e^fh_\omega}+\Phi_{e^fh_\omega}^\xi\right)=&\frac{i}{2\pi}\left(\Lambda_{\omega} F_{h_\omega}+\Phi_{h_\omega}^\xi\right)+\frac{1}{2\pi}e^{-\mu_\xi}\Delta_{e^{\mu_\xi}}(f)\Id_{TX^{1,0}}\\
		=&\left(1+\frac{1}{2\pi}e^{-\mu_\xi}\Delta_{e^{\mu_\xi}}(f)\right) \Id_{TX^{1,0}},
	\end{split}
	\] where \(\Delta_{e^{\mu_\xi}}(f):=\d^*(e^{\mu_\xi}\d f)\) (see the proof of Lemma \ref{vw-weighted-is-v-weighted}).
	Thus the second equation in \eqref{eq:exp-sol-ext1} is equivalent to
	\[
	1+\frac{1}{2\pi}e^{-\mu_\xi}\Delta_{e^{\mu_\xi}}(f)=w_1(\mu)+\frac{\gamma^2}{2\pi}e^{u-f}.
	\]
	For the last equation in \eqref{eq:exp-sol-ext1}, we have $F_{h_{\mathcal{O}}}=\frac{i}{2}dd^cu$ with a bundle momentum map $\Phi^\mathcal{O}$ given by
	\[
	\begin{split}
		\Phi_{\mathcal{O}}^\xi(s)=&\nabla^{\mathcal{O}}_\xi s-\mathcal{L}^\mathcal{O}_\xi s\\
		=& (ds)(\xi)  +\frac{1}{2}(du(\xi)+i(d^cu)(\xi))s- (ds)(\xi)\\
		=&\frac{i}{2}(d^cu)(\xi)s=\frac{i}{2}\langle d\mu_\xi,du\rangle s.
	\end{split}
	\]
	where $du(\xi)=0$ by $\T$-invariance of $u$ and $\mathcal{L}^\mathcal{O}_\xi s=\begin{pmatrix}
		\xi&0\\
		0&\partial_\theta
	\end{pmatrix}s=(ds)(\xi)$ since a section $s\in \Gamma(\mathcal{O})$ is constant along the fibers. Thus, the last equation of \eqref{eq:exp-sol-ext1} is equivalent to
	\[
	\frac{1}{2\pi}e^{-\mu_\xi}\Delta_{e^{\mu_\xi}}(u)= w_2(\mu)-\frac{n\gamma^2}{2\pi}e^{u-f}.
	\]
	Thus, the system \eqref{eq:exp-sol-ext1} is equivalent to
	\begin{equation}
		\begin{cases}\label{eq:exp-sol-ext2}
			u-f=-\mu_\xi\\
			1+\frac{1}{2\pi}e^{-\mu_\xi}\Delta_{e^{\mu_\xi}}(f)=w_1(\mu)+\frac{\gamma^2}{2\pi}e^{u-f}\\
			\frac{1}{2\pi}e^{-\mu_\xi}\Delta_{e^{\mu_\xi}}(u)= w_2(\mu)-\frac{n\gamma^2}{2\pi}e^{u-f}.
		\end{cases}
	\end{equation}
	
	We suppose that the momentum polytope is canonically normalized so that for all $\zeta\in\tor$,
	\[
	e^{-\mu_\xi}\Delta_{e^{\mu_\xi}}(\mu_\zeta)=\mu_\zeta.
	\]
	For the third equation, taking $u=\mu_\xi$ gives
	\[
	w_2(\mu)=\frac{1}{2\pi}\mu_\xi+\frac{n\gamma^2}{2\pi}e^{-\mu_\xi} ,\quad f=2\mu_\xi.
	\]
	Substituting in the second equation we get
	\[
	w_1(\mu)=1+\frac{1}{\pi}\mu_\xi - \frac{\gamma^2}{2\pi}e^{-\mu_\xi}.
	\]
	
	To recapitulate, for all $\gamma\in\R$ the Hermitian metric $e^{2\mu_\xi}h_\omega\oplus e^{\mu_\xi}$ on the extension $E_{[\gamma\Psi^\omega]}TX^{1,0}$ solves the equation
	\begin{equation}\label{eq:exp_wHE_vortex}
		\frac{i}{2\pi}\left(\Lambda_\omega F_{h_E}+\Phi_{h_E}^\xi\right)=\begin{pmatrix}
			\left(1+\frac{1}{\pi}\mu_\xi - \frac{\gamma^2}{2\pi}e^{-\mu_\xi}\right)\Id_{TX^{1,0}}&0\\
			0& \frac{1}{2\pi}\mu_\xi+\frac{n\gamma^2}{2\pi}e^{-\mu_\xi}
		\end{pmatrix}.
	\end{equation}
	Since \(E_{[\gamma\Psi^\omega]}\) splits smoothly as a direct sum \(TX^{1,0}\oplus\mathcal{O}\) of complex vector bundles,
	\begin{equation}
		\begin{split}\label{eq:SlopExt}
			\mu_{e^{\mu_\xi}}(E_{[\gamma\Psi^\omega]})=&\frac{\left(c_1(E_{[\gamma\Psi^{\omega}]})_{\T}\cdot e^{c_1(X)_\T}\right)(\xi)}{n+1}\\
			=&\frac{\left(c_1(X)_{\T}\cdot e^{c_1(X)_\T}\right)(\xi)}{n+1} \\
			=&\frac{n}{n+1}\left(e^{ c_1(X)_\T}\right)(\xi),
		\end{split}
	\end{equation} as may be verified directly from the definitions of equivariant intersection numbers.
\end{proof}

\subsection{Bounds on the weighted Ricci curvature} Following Tian \cite{Tia92}, we consider the problem: given a Fano $\T$-manifold $X$ and $\xi\in \tor$, is there an obstruction to the existence of a $\T$-invariant K\"ahler metric $\omega\in c_1(X)$ such that $\Ric(\omega)-\frac{i}{2\pi}\d\db\mu_\xi\geq t\omega$ for some $t\in [0,1]$?

We define the greatest $\xi$-Ricci lower bound by
\begin{equation}
	\begin{split}\label{eq:beta_xi}
		R_\xi(X):=\sup\big\{t\geq 0\mid  \exists \,\omega\in \mathcal{K}^\T(c_1(X)),\,\,\,
		\Ric(\omega)-\frac{i}{2\pi}\d\db\mu_\xi\geq t\omega\big\},
	\end{split}
\end{equation}
where $\mathcal{K}^\T(c_1(X))$ is the space of $\T$-invariant K\"ahler metrics in the K\"ahler class $c_1(X)$. Notice that $ R_\xi (X)\leq 1$\footnote{If $\Ric(\omega)-\frac{i}{2\pi}\d\db\mu_\xi\geq t\omega$ then $(c_1(X))^{[n]}\geq t^n (c_1(X))^{[n]}$ hence $t\leq 1$, showing $R_\xi(X)\leq 1$.} and $R_\xi(X)= 1$ is a necessary condition for $(X,\xi)$ to be a K{\"a}hler--Ricci soliton.

To construct a K\"ahler metrics $\omega_t\in \mathcal{K}^\T(c_1(X))$ such that $\Ric(\omega_t)-\frac{i}{2\pi}\d\db\mu_{\omega_t}^\xi\geq t\omega_t$, one can use the continuity path of Tian--Zhu \cite{TZ} (see also the work \cite{Sze11} for the K{\"a}hler--Einstein case), defined for $t\in [0,1]$ by
\begin{equation}\label{eq:C0_TZ}
	\Ric(\omega_\varphi)-\frac{i}{2\pi}\d\db\mu^\xi_\varphi=t\omega_\varphi+(1-t)\omega;
\end{equation} here \(\varphi\) is a \(\T\)-invariant K{\"a}hler potential with respect to \(\omega\), and \(\omega_\varphi:=\omega+i\d\db\varphi\). Introducing the Ricci potential of $\omega$,
\[
\Ric(\omega)-\omega=\frac{i}{2\pi}\d\db p_\omega,\quad\int_X e^{p_\omega}\omega^{[n]}=(c_1(X))^{[n]},
\]
the equation \eqref{eq:C0_TZ} is equivalent to the complex Monge--Amp{\`e}re equation 
\[
\begin{cases}
	\omega_\varphi^{[n]}=e^{p_\omega - \mu^\xi_\varphi-2\pi t\varphi}\omega^{[n]}\\
	\omega_\varphi\in\mathcal{K}^\T(c_1(X)).
\end{cases}
\]
A solution $\omega_{\varphi_t}$ of \eqref{eq:C0_TZ} for some $0\leq t<1$ clearly satisfies
\[
\Ric(\omega_{\varphi_t})-\frac{i}{2\pi}\d\db \mu_{\varphi_t}^\xi\geq t\omega_{\varphi_t},
\]
thus, for $t> R_\xi(X)$ the equation \eqref{eq:C0_TZ} is not solvable.

Now, we define the exponentially weighted version of Tian's invariant: let $E$ be the extension of $TX^{1,0}$ by the trivial line bundle with extension class $\gamma c_1(X)$, where \(\gamma>0\) is defined by $$\gamma^2=\frac{2\pi}{n+1}\frac{\left(e^{c_1(X)_\T}\right)(\xi)}{(c_1(X))^{[n]}}.$$ For any $\T$-equivariant coherent saturated subsheaf $\F$ of $TX^{1,0}$ we define $\beta_\xi(\F)$ by
\[
\beta_\xi(\F)=\frac{n+1}{n+1-\rk(\F)}\left(1-\frac{\left(c_1(\F)_{\T}\cdot e^{ c_1(X)_\T}\right)(\xi)}{\left(c_1(X)_\T\cdot e^{c_1(X)_\T}\right)(\xi)}\right),
\]
if $\F$ can be lifted as a $\T$-equivariant subsheaf \(\F'\) of the extension $E$ with $\ker(\F'\to F)=\mathcal{O}$, and by
\[
\beta_\xi(\F)=\frac{n}{n-\rk(\F)}\left(1-\frac{\left(c_1(\F)_{\T}\cdot e^{c_1(X)_\T}\right)(\xi)}{\left( c_1(X)_\T \cdot e^{c_1(X)_\T}\right)(\xi)}\right),
\]
if $\F$ cannot be lifted to a $\T$-equivariant holomorphic subsheaf of the extension $E$. Note the above two case are a dichotomy---if \(\F\) admits a lift \(\F'\subset E\), we may always take \(\F'':=\F'+\mathcal{O}\) as a lift of \(\F\) such that \(\ker(\F''\to \F)=\mathcal{O}\). We define the invariant $\beta_\xi(X)$ as
\begin{equation}\label{eq:beta-invariant}
	\beta_\xi(X):=\inf\{1,\beta_\xi(\F)\}
\end{equation}
where $\F$ runs over all $\T$-equivariant coherent saturated subsheaves of $TX^{1,0}$.
\begin{thm}\label{thm:weighted-Ricci-bound}
	Let $X$ be a Fano manifold $\T$-manifold and $\xi\in \tor$. Then,
	\[
	R_\xi(X)\leq \beta_\xi(X)
	\]
\end{thm}

\begin{proof}
	Suppose that there exists $\omega\in\mathcal{K}^\T(c_1(X))$ such that
	\begin{equation}\label{eq:Ricxi>omega}
		\Ric(\omega)-\frac{i}{2\pi}\d\db\mu_\xi\geq t\omega.
	\end{equation}
	We will show that $t\leq \beta_\xi(\F)$ for all $\T$-equivariant coherent saturated subsheaves $\F\subset TX^{1,0}$. We denote by $h_\omega$ the $\T$-invariant Hermitian metric on $TX^{1,0}$ corresponding to $\omega$ and $h_\F$ the induced metric on $\F$. Using \eqref{eq:decomposition} and \eqref{eq:Ricxi>omega}, we get
	\[
	\begin{split}
		\tr_\F\,i\Lambda_{\omega,e^{\mu_\xi}}(F_{h_\F}+\Phi_{h_\F})=& \tr_\F\,i\Lambda_{\omega,e^{\mu_\xi}}(F_{h_\omega}+\Phi_{h_\omega})_{\mid \F}  + e^{\mu_\xi}\tr_\F\big(i \Lambda_\omega (A^\dagger\wedge A)\big)\\
		\leq& \tr_{TX^{1,0}}\,i\Lambda_{\omega,e^{\mu_\xi}}(F_{h_\omega}+\Phi_{h_\omega})- \tr_{\F^{\perp}}\,i\Lambda_{\omega,e^{\mu_\xi }}(F_{h_\omega}+\Phi_{h_\omega})_{\mid \F^{\perp}}\\
		\leq& \tr_{TX^{1,0}}\,i\Lambda_{\omega,e^{\mu_\xi}}(F_{h_\omega}+\Phi_{h_\omega})-2\pi e^{\mu_\xi}(n-\rk(\F))t
	\end{split}
	\]
	where $A$ is second fundamental form of $\F$ in $(TX^{1,0},h_\omega)$. Integrating both sides relative to the measure $\omega^{[n]}$ over $X$ and using Lemma \ref{constants-in-wHE} we obtain
	\[
	\begin{split}
		\left(c_1(\F)_{\T}\cdot e^{c_1(X)_\T}\right)(\xi)\leq& \left(c_1(TX^{1,0})_{\T}\cdot e^{c_1(X)_\T}\right)(\xi)-t\,(n-\rk(\F))\left( e^{c_1(X)_\T}\right)(\xi)\\
		=&\left(1-\frac{t\,(n-\rk(\F))}{n}\right) \left( c_1(X)_\T\cdot e^{c_1(X)_\T}\right)(\xi).
	\end{split}
	\]
	Consequently,
	\[
	t\leq \frac{n}{n-\rk(\F)}\left(1-\frac{\left(c_1(\F)_{\T}\cdot e^{c_1(X)_\T}\right)(\xi)}{\left( c_1(X)_\T\cdot e^{c_1(X)_\T}\right)(\xi)}\right).
	\]
	
	Now, as in the proof of Theorem \ref{thm:ExtSolitons}, let $E_t$ be the extension of $TX^{1,0}$ by the trivial line bundle $\mathcal{O}$ with the extension representative $\gamma_t\Psi^\omega$, where $$\gamma^2_t:=\frac{2\pi t}{n+1}\frac{\left(e^{c_1(X)_\T}\right)(\xi)}{(c_1(X))^{[n]}}.$$ By adapting the proof of Theorem \ref{thm:ExtSolitons}, using instead the K\"ahler metric $\omega\in \mathcal{K}^\T(c_1(X))$ satisfying \eqref{eq:Ricxi>omega}, one may show that there exists a metric \(h_{E_t}\) on \(E_t\) such that
	\begin{equation}\label{eq:wHE>Id}
		\frac{i}{2\pi}\left(\Lambda_\omega F_{h_{E_t}}+\Phi_{h_{E_t}}^\xi\right)\geq
		\begin{pmatrix}
			\left(t+\frac{1}{\pi}\mu_\xi - \frac{\gamma_t^2}{2\pi}e^{-\mu_\xi}\right)\Id_{TX^{1,0}}&0\\
			0& \frac{1}{2\pi}\mu_\xi+\frac{n\gamma_t^2}{2\pi}e^{-\mu_\xi}
		\end{pmatrix}.
	\end{equation}
	Let $\F\subset TX^{1,0}$ be a $\T$-equivariant coherent saturated subsheaf which can be lifted to a subsheaf of $E_t$. We compute as above
	\[
	\begin{split}
		&\int_X\tr_\F\,i\Lambda_{\omega,e^{\mu_\xi}}(F_{h_\F}+\Phi_\F)\omega^{[n]}\\
		=& \int_X\tr_\F\,i\Lambda_{\omega,e^{\mu_\xi}}(F_{h_E}+\Phi_{h_E})_{\mid \F}  \omega^{[n]}+ \int_X\tr_\F\big(i \Lambda_\omega (A^\dagger\wedge A)\big)\omega^{[n]}\\
		\leq& \int_X\tr_{E_t}\,i\Lambda_{\omega,e^{\mu_\xi}}(F_{h_E}+\Phi_{h_E})\omega^{[n]}-\int_X \tr_{\F^{\perp}}\,i\Lambda_{\omega,e^{\mu_\xi }}(F_{h_E}+\Phi_{h_E})_{\mid \F^{\perp}} \omega^{[n]}.
	\end{split}
	\]
	Since $\F\subset TX^{1,0}$ can be lifted to a subsheaf \(\F'\) of $E$ containing \(\mathcal{O}\), there is a \(C^\infty\) decomposition \[E = \mathcal{O}\oplus TX^{1,0}=\mathcal{O}\oplus\F''\oplus (\F')^\perp,\] over the smooth locus of \(\F'\), where \(\F''\) denotes the orthogonal complement of \(\mathcal{O}\) in \(\F'\). Using \eqref{eq:wHE>Id} and the fact that \((\F')^\perp\subset TX^{1,0}\) over the smooth locus of \(\F\), we get
	\[
	\begin{split}
		&\frac{1}{2\pi}\int_X\tr_{(\F')^{\perp}}\,i\Lambda_{\omega,e^{\mu_\xi }}(F_{h_E}+\Phi_{h_E})_{\mid \F^{\perp}} \omega^{[n]}\\
		\geq& \int_X\tr_{(\F')^{\perp}}\begin{pmatrix}
			w_1\Id_{TX^{1,0}}&0\\
			0& w_2
		\end{pmatrix} e^{\mu_\xi}\omega^{[n]}\\
		=& \int_X\tr_{(\F')^{\perp}}\begin{pmatrix}
			w_1\Id_{\F''}&0&0\\
			0&w_1\Id_{(\F')^{\perp}}&0\\
			0&0& w_2
		\end{pmatrix} e^{\mu_\xi}\omega^{[n]}\\
		=&\int_X\rk((\F)'^{\perp}) \left(t+\frac{1}{\pi}\mu_\xi - \frac{\gamma_t^2}{2\pi}e^{-\mu_\xi}\right) e^{\mu_\xi}\omega^{[n]}\\
		=&(n-\rk(\F)+1)\left(t\left(e^{c_1(X)_\T}\right)(\xi) -\frac{\gamma_t^2}{2\pi}(c_1(X))^{[n]}\right)\\
		=&\frac{(n-\rk(\F)+1)t}{n+1}\left(c_1(X)_\T \cdot e^{c_1(X)_\T}\right)(\xi).
	\end{split}
	\]
	Substituting back we get
	\[
	\begin{split}
		\left(c_1(\F)_{\T}\cdot e^{c_1(X)_\T}\right)\leq &\left(c_1(TX^{1,0}\oplus \mathcal{O})_{\T}\cdot e^{c_1(X)_\T}\right)(\xi) -\frac{(n-\rk(\F)+1)t}{n+1}\left(c_1(X)_\T\cdot e^{c_1(X)_\T}\right)(\xi)\\
		=&\left(1-\frac{(n+1-\rk(\F))t}{n+1}\right)\left(c_1(X)_\T\cdot e^{c_1(X)_\T}\right)(\xi).
	\end{split}
	\]
	Hence,
	\[
	t\leq \frac{n+1}{n+1-\rk(\F)}\left(1-\frac{ \left(c_1(\F)_{\T}\cdot e^{c_1(X)_\T}\right)(\xi)}{\left(c_1(X)_\T\cdot  e^{c_1(X)_\T}\right)(\xi)}\right).
	\]
	It follows that $t\leq \beta_\xi(\F)$ for all $\T$-equivariant coherent saturated subsheaves $\F\subset TX^{1,0}$, that is $t\leq \inf\{1,\beta_\xi(\F)\}=\beta_\xi(X)$. Showing that $R_\xi(X)\leq \beta_\xi(X)$. 
\end{proof}

\begin{cor}
	Let $X$ be a Fano manifold $\T$-manifold and $\xi\in \tor$. There is no K\"ahler metric $\omega\in\mathcal{K}^\T(c_1(X))$ with $\Ric(\omega)-\frac{i}{2\pi}\d\db\mu_\xi\geq t\omega$ and $t>\beta_\xi(X)$. In particular, if \(\beta_\xi(X)<1\), \(X\) does not admit a K{\"a}hler--Ricci soliton with field \(\xi\).
\end{cor}

\section{The weighted Donaldson functional and uniqueness}\label{sec:uniqueness}

In this section, we introduce the weighted Donaldson functional on the space of \(\T\)-invariant Hermitian metrics. We then show this functional is convex along geodesics, which will imply uniqueness of weighted Hermite--Einstein metrics up to scalar multiplication on simple vector bundles.

\subsection{Equivariant Bott--Chern forms (after \cite[section 1]{Tia05})}

Let $(X,\omega)$ be a K\"ahler $\T$-manifold with a fixed momentum image $\mu:X\to\Pol\subset \tor^*$ and $E$ a $\T$-equivariant holomorphic vector bundle over $X$. The equivariant Chern character is represented by
\begin{equation}\label{eq:Ch}
	{\rm ch}_{\T}(E,h):=\tr (e^{\frac{i}{2\pi}(F_h+\Phi_h)})
\end{equation}
this is an equivariantly closed form depending on the $\T$-invariant Hermitian metric $h$. The equivariant Bott--Chern forms measure the difference between these forms under the change of the $\T$-invariant Hermitian metric $h$.
\begin{lemma}[\cite{Tia05}]
	There exists a map
	\[
	{\rm BC}:\mathcal{H}^{\T}(E)\times\mathcal{H}^{\T}(E)\to \left(\bigoplus_{p\geq 0}\Omega_{\T}^{p,p}(X)\right)\Big/\left({\rm Im}(\partial^\eq)\oplus{\rm Im}(\db^\eq)\right),
	\]
	called the \emph{equivariant Bott--Chern form}, defined by
	\begin{equation}\label{eq:BCh}
		{\rm BC}_{\T}(h,h'):=\int_0^1\tr\left(e^{\frac{i}{2\pi}(F_{h_t}+\Phi_{h_t})}h_t^{-1}\dot{h}_t\right)dt+{\rm Im}(\d^\eq)\oplus{\rm Im}(\db^\eq).
	\end{equation}
	where $(h_t)_{t\in [0,1]}$ is a path of $\T$-invariant metrics joining $h_0=h$ to $h_1=h'$ and $\dot{h}_t$ is the velocity of the path. The Bott--Chern forms satisfy the following properties:
	\begin{enumerate}
		\item ${\rm BC}_{\T}(h,h)=0$ and ${\rm BC}_{\T}(h,h')+{\rm BC}_{\T}(h',h'')={\rm BC}_{\T}(h,h'')$.
		\item $\frac{i}{2\pi}\d^\eq\db^\eq{\rm BC}_{\T}(h,h')={\rm ch}_{\T}(h)-{\rm ch}_{\T}(h')$.
	\end{enumerate}
\end{lemma}
\begin{proof}
	The Lemma is established in \cite{Tia05} for circle actions $\T=\mathbb{S}^1$. We give the proof for the case of higher-dimensional torus actions for convenience. For the equivariant Bott--Chern form to be well defined we need to show that \eqref{eq:BCh} is independent from the choice of the path joining $h$ and $h'$. To this end, we define the $\bigoplus_{p\geq 0}\Omega_{\T}^{p,p}(X)$-valued 1-form $\Theta$ on the space of metrics $\mathcal{H}^{\T}(E)$ by
	\[
	\Theta_h(\dot{h}):=\tr\left(e^{\frac{i}{2\pi}(F_{h}+\Phi_{h})} h^{-1}\dot{h}\right).
	\]
	We have to check that the exterior derivative of $\Theta$ lies in ${\rm Im}(\d^\eq)\oplus{\rm Im}(\db^\eq)$. From \eqref{eq:diff-F} below, for any reference metric $h_{\rm ref}\in\mathcal{H}^{\T}(E)$ we have
	\[
	F_h+\Phi_h=F_{h_{\rm ref}}+\Phi_{h_{\rm ref}}+\db^{\eq}_E\left( (h^{-1}h_{\rm ref})\partial_{h_{\rm ref}}(h_{\rm ref}^{-1}h)\right).
	\]
	The variation of $h\mapsto F_h+\Phi_h$ is given by
	\[
	\begin{split} \dot{F}_h+\dot{\Phi}_h=&\db^{\eq}_E\left(-h^{-1}\dot{h}h^{-1}h_{\rm ref}\partial_{h_{\rm ref}} (h_{\rm ref}^{-1}h)+ (h^{-1}h_{\rm ref})\partial_{h_{\rm ref}}(h_{\rm ref}^{-1}\dot{h})\right).
	\end{split}
	\]
	Since the above equation holds for any $\T$-invariant reference metric $h_{\rm ref}$, then by taking $h_{\rm ref}=h$ we get
	\begin{equation}\label{eq:Diff-Curvat}
		\dot{F}_h+\dot{\Phi}_h=\db^{\eq}_E\partial_h(h^{-1}\dot{h}).
	\end{equation}
	Let $\dot{h}_{1},\dot{h}_{2}$ be two vector fields on $\mathcal{H}^{\T}(E)$. Using the above formula, we compute the variation of the function $h\mapsto  \Theta_h(\dot{h}_1)$ at $h$ in the direction of $\dot{h}_2$,
	\[
	\begin{split}
		d(\Theta(\dot{h}_1))_h(\dot{h}_2)=& -\tr\left(e^{\frac{i}{2\pi}(F_{h}+\Phi_{h})} \dot{H}_2\dot{H}_1\right)\\
		&+\tr(e^{\frac{i}{2\pi}(F_h+\Phi_h)}h^{-1}\boldsymbol{\nabla}_{\dot{h}_2}\dot{h}_1) \\
		&+\frac{i}{2\pi}\int_0^1 \tr\left(e^{\tau\frac{i}{2\pi}(F_{h}+\Phi_{h})} (\db^{\eq}_E\partial_h\dot{H}_2) e^{(1-\tau)\frac{i}{2\pi}(F_{h}+\Phi_{h})} \dot{H}_1\right)d\tau
	\end{split}
	\]
	where we denote $\dot{H}_1=h^{-1}\dot{h}_1$, $\dot{H}_2=h^{-1}\dot{h}_2$, \(\boldsymbol{\nabla}\) is the canonical connection on the infinite-dimensional symmetric manifold \(\mathcal{H}^\T(E)\), and we use the following well-known expression for the variation of the exponential $e^{T}$ of an endomorphism $T$,
	\begin{equation}\label{eq:diff-exp}
		\delta(e^{T})=\int_0^1e^{\tau T} (\delta T) e^{(1-\tau) T} d\tau.
	\end{equation}
	It follows that
	\begin{equation}
		\begin{split}\label{eq:DiffTheta}
			(d\Theta)_h(\dot{h}_{1},\dot{h}_2)=& d(\Theta(\dot{h}_2))_h(\dot{h}_1)-d(\Theta(\dot{h}_1))_h(\dot{h}_2) - \Theta_h([\dot{h}_1,\dot{h}_2])\\
			=&\tr\left(e^{\frac{i}{2\pi}(F_{h}+\Phi_{h})} [\dot{H}_2,\dot{H}_1]\right)\\
			&+\frac{i}{2\pi}\int_0^1 \tr\left( e^{\tau\frac{i}{2\pi}(F_{h}+\Phi_{h})} (\db^{\eq}_E\partial_h\dot{H}_1) e^{(1-\tau)\frac{i}{2\pi}(F_{h}+\Phi_{h})} \dot{H}_2\right)d\tau\\
			&-\frac{i}{2\pi}\int_0^1 \tr\left( e^{\tau\frac{i}{2\pi}(F_{h}+\Phi_{h})} (\db^{\eq}_E\partial_h\dot{H}_2) e^{(1-\tau)\frac{i}{2\pi}(F_{h}+\Phi_{h})} \dot{H}_1\right)d\tau,
		\end{split}
	\end{equation} where we use that \(\boldsymbol{\nabla}\) is torsion-free. For all $\varepsilon\in\R$ we have
	\[
	\begin{split}
		\tr\left(e^{\frac{i}{2\pi}(F_{h}+\Phi_{h})} \dot{H}_1\right)=&\tr\left(e^{\varepsilon\dot{H}_2}e^{\frac{i}{2\pi}(F_{h}+\Phi_{h})} e^{-\varepsilon\dot{H}_2}\cdot e^{\varepsilon\dot{H}_2}\dot{H}_1e^{-\varepsilon\dot{H}_2}\right)\\
		=&\tr\left(\exp\left( e^{\varepsilon\dot{H}_2}\frac{i}{2\pi}(F_{h}+\Phi_{h}) e^{-\varepsilon\dot{H}_2}\right) \cdot e^{\varepsilon\dot{H}_2}\dot{H}_1e^{-\varepsilon\dot{H}_2}\right).
	\end{split}
	\]
	Using \eqref{eq:diff-exp} to differentiate at $\varepsilon=0$ gives
	\[
	\begin{split}
		\tr\left(e^{\frac{i}{2\pi}(F_{h}+\Phi_{h})} [\dot{H}_1,\dot{H}_2]\right)=&-\frac{i}{2\pi}\int_0^1\tr\left(e^{\tau\frac{i}{2\pi}(F_{h}+\Phi_{h})} [F_h+\Phi_h,\dot{H}_2] e^{(1-\tau) \frac{i}{2\pi}(F_{h}+\Phi_{h})}\dot{H}_1\right) d\tau\\
		=&-\frac{i}{2\pi}\int_0^1\tr\left(e^{\tau\frac{i}{2\pi}(F_{h}+\Phi_{h})} (\db^{\eq}\partial^{\eq}_h+\partial^{\eq}_h\db^{\eq})\dot{H}_2 e^{(1-\tau) \frac{i}{2\pi}(F_{h}+\Phi_{h})}\dot{H}_1\right) d\tau,
	\end{split}
	\]
	where in the second equality we use the identity $F_{h}+\Phi_{h}=\db^{\eq}\partial^{\eq}_h+\partial^{\eq}_h\db^{\eq}$. Plugging the above equation in \eqref{eq:DiffTheta} gives
	\[
	\begin{split}
		(d\Theta)_h(\dot{h}_{1},\dot{h}_2)=& \frac{i}{2\pi}\int_0^1 \tr\left( e^{\tau\frac{i}{2\pi}(F_{h}+\Phi_{h})} (\db^{\eq}_E\partial_h\dot{H}_1) e^{(1-\tau)\frac{i}{2\pi}(F_{h}+\Phi_{h})} \dot{H}_2\right)d\tau\\
		&+\frac{i}{2\pi}\int_0^1\tr\left(e^{\tau\frac{i}{2\pi}(F_{h}+\Phi_{h})} (\partial^{\eq}_h\db_{E}\dot{H}_2) e^{(1-\tau) \frac{i}{2\pi}(F_{h}+\Phi_{h})}\dot{H}_1\right) d\tau.
	\end{split}
	\]
	By the equivariant Bianchi identity,
	\[
	\begin{split}
		&\d^\eq\db^\eq \tr\left(e^{\tau\frac{i}{2\pi}(F_{h}+\Phi_{h})} \dot{H}_2 e^{(1-\tau) \frac{i}{2\pi}(F_{h}+\Phi_{h})}\dot{H}_1\right)\\
		=& \tr\left(e^{\tau\frac{i}{2\pi}(F_{h}+\Phi_{h})} (\partial^{\eq}_h\db_E\dot{H}_2) e^{(1-\tau) \frac{i}{2\pi}(F_{h}+\Phi_{h})}\dot{H}_1\right)+\tr\left(e^{\tau\frac{i}{2\pi}(F_{h}+\Phi_{h})} \dot{H}_2 e^{(1-\tau) \frac{i}{2\pi}(F_{h}+\Phi_{h})}(\partial^{\eq}_h\db_E\dot{H}_1)\right)\\
		&-\tr\left(e^{\tau\frac{i}{2\pi}(F_{h}+\Phi_{h})} \db_E\dot{H}_2 e^{(1-\tau) \frac{i}{2\pi}(F_{h}+\Phi_{h})} \partial_h\dot{H}_1\right)+\tr\left(e^{\tau\frac{i}{2\pi}(F_{h}+\Phi_{h})} \partial_h\dot{H}_2 e^{(1-\tau) \frac{i}{2\pi}(F_{h}+\Phi_{h})} \db_E\dot{H}_1\right),
	\end{split}
	\] and \[
	\begin{split}
		&\db^\eq \tr\left(e^{\tau\frac{i}{2\pi}(F_{h}+\Phi_{h})} \dot{H}_2 e^{(1-\tau) \frac{i}{2\pi}(F_{h}+\Phi_{h})}\partial_h\dot{H}_1\right)-\d^\eq \tr\left(e^{\tau\frac{i}{2\pi}(F_{h}+\Phi_{h})} \dot{H}_2 e^{(1-\tau) \frac{i}{2\pi}(F_{h}+\Phi_{h})}\db_E\dot{H}_1\right)\\
		=&\tr\left(e^{\tau\frac{i}{2\pi}(F_{h}+\Phi_{h})} \db_E\dot{H}_2 e^{(1-\tau) \frac{i}{2\pi}(F_{h}+\Phi_{h})}\partial_h\dot{H}_1\right)+\tr\left(e^{\tau\frac{i}{2\pi}(F_{h}+\Phi_{h})}\dot{H}_2 e^{(1-\tau) \frac{i}{2\pi}(F_{h}+\Phi_{h})}\db^{\eq}_E\partial_h\dot{H}_1\right)\\
		&-\tr\left(e^{\tau\frac{i}{2\pi}(F_{h}+\Phi_{h})} \partial_h\dot{H}_2 e^{(1-\tau) \frac{i}{2\pi}(F_{h}+\Phi_{h})}\db_E\dot{H}_1\right)-\tr\left(e^{\tau\frac{i}{2\pi}(F_{h}+\Phi_{h})}\dot{H}_2 e^{(1-\tau) \frac{i}{2\pi}(F_{h}+\Phi_{h})}\partial^{\eq}_h\db_E\dot{H}_1\right).
	\end{split}
	\]
	Summing the above two equations and integrating the result $\frac{i}{2\pi}\int_0^1$ relative to $\tau$, we get an element in ${\rm Im}(\d^\eq)\oplus{\rm Im}(\db^\eq)$
	which is equal to
	\begin{equation}\label{eq:ddH-ddH}
		\begin{split}
			&\frac{i}{2\pi}\int_0^1\tr\left(e^{\tau\frac{i}{2\pi}(F_{h}+\Phi_{h})} (\partial^{\eq}_h\db_{E}\dot{H}_2) e^{(1-\tau) \frac{i}{2\pi}(F_{h}+\Phi_{h})}\dot{H}_1\right) d\tau\\
			&+ \frac{i}{2\pi}\int_0^1 \tr\left(e^{\tau\frac{i}{2\pi}(F_{h}+\Phi_{h})}\dot{H}_2 e^{(1-\tau) \frac{i}{2\pi}(F_{h}+\Phi_{h})}\db^{\eq}_E\partial_h\dot{H}_1\right)d\tau.
		\end{split}    
	\end{equation}
	In the second integral of \eqref{eq:ddH-ddH}, permuting $e^{(1-\tau) \frac{i}{2\pi}(F_{h}+\Phi_{h})}\db^{\eq}_E\partial_h\dot{H}_1$ and $e^{\tau\frac{i}{2\pi}(F_{h}+\Phi_{h})}\dot{H}_2$ under the trace and making the change of variable $\tau\mapsto 1-\tau$, shows that the expression \eqref{eq:ddH-ddH} is exactly $(d\Theta)_h(\dot{h}_{1},\dot{h}_2)$. This completes the proof that the expression defining \(\mathrm{BC}_\T(h,h')\) is independent from the choice of paths joining $h$ and $h'$ up to elements in \(\mathrm{Im}(\d^\eq)\oplus\mathrm{Im}(\db^\eq)\).
	
	The cocycle property (1) follows from \eqref{eq:BCh}. To verify the property (2), let $h_t\in\mathcal{H}^{\T}(E)$ be a smooth path such that $h_0=h'$ and $h_1=h$. Using \eqref{eq:Diff-Curvat}, we compute
	\[
	\begin{split}
		\frac{d}{dt}{\rm ch}_{\T}(h_t)=&\frac{d}{dt}\tr (e^{\frac{i}{2\pi}(F_{h_t}+\Phi_{h_t})})\\
		=&\frac{i}{2\pi}\tr\left(e^{\frac{i}{2\pi}(F_{h_t}+\Phi_{h_t})}\db^{\eq}_E\partial_{h_t}(h^{-1}_t\dot{h}_t)\right)\\
		=&\frac{i}{2\pi}\d^\eq\db^\eq\tr\left(e^{\frac{i}{2\pi}(F_{h_t}+\Phi_{h_t})}(h^{-1}_t\dot{h}_t)\right)\\
		=&\frac{i}{2\pi}\d^\eq\db^\eq \frac{d}{dt} {\rm BC}_{\T}(h_t,h').
	\end{split}
	\]
	Integrating both sides relative to $t\in[0,1]$ gives the result.
\end{proof}

Now we define the weighted Donaldson functional. Letting $v\in C^{\infty}(\Pol,\R_{>0})$ be a weight function, we set
\begin{equation}
	\begin{split}\label{eq:Don-Funct-0}
		\mathcal{M}_v(h,h'):=c_v  \int_\tor& \left({\rm BC}_{\T}(h,h')^{(0,0)}\cdot e^{\alpha_\T}\right)
		(i\xi)\hat{v}(\xi)\underline{d\xi}\\
		&-\int_\tor\left({\rm BC}_{\T}(h,h')^{(1,1)}\cdot e^{\alpha_\T}\right)(i\xi)\hat{v}(\xi)\underline{d\xi}
	\end{split}
\end{equation}
where ${\rm BC}_{\T}(h,h')^{(p,p)}$ stands for the $(p,p)$-component of the equivariant Bott--Chern form and $c_v:=\frac{(c_1(E)_{\T}\cdot v(\alpha_\T))}{(\alpha^{[n]})}$. For a path $h_t\in\mathcal{H}^\T(E)$ we have
\[
\begin{split}
	\frac{d}{dt} {\rm BC}_{\T}(h_t,h')^{(0,0)}=& \tr(h_t^{-1}\dot{h}_t)dt\\
	\frac{d}{dt}{\rm BC}_{\T}(h_t,h')^{(1,1)}=& \tr\left(\frac{i}{2\pi}(F_{h_t}+\Phi_{h_t})h_t^{-1}\dot{h}_t\right).
\end{split}
\]
Substituting back in \eqref{eq:Don-Funct-0} gives the expression
\begin{equation}
	\begin{split}\label{eq:Don-Funct-1}
		\frac{d}{dt}\mathcal{M}_v(h_t,h')=\int_X \tr\left(h_t^{-1}\dot{h}_t\left( c_v\Id_E-\frac{i}{2\pi}\Lambda_{\omega,v}(F_{h_t}+\Phi_{h_t}) \right)\right)\omega^{[n]}.
	\end{split}
\end{equation}
\begin{lemma}\label{lem:ConvexDonaldson}
	Let $h\in\mathcal{H}^\T(E)$ and $\Psi$ a $\T$-invariant endomorphism of $E$. Then the weighted Donaldson function $t\mapsto \mathcal{M}_v(h_t,h)$ is convex along the $\T$-invariant geodesic path $h_t:=e^{t\Psi}h$.
\end{lemma}

\begin{proof}
	We show that the second derivative of \(\mathcal{M}_v(h_t, h')\) along \(h_t\) is non-negative. First note \(h_t^{-1}\dot{h}_t\) is constant along a geodesic path, so we need only compute the derivative of the curvature term in \eqref{eq:Don-Funct-1}. By \cite[(4.2.12), (4.2.16)]{Kob14}, \[\d_t\Phi_{h_t} = \d(h_t^{-1}\dot{h}_t),\quad\quad \d_t\Lambda_\omega F_{h_t} = \db\d(h_t^{-1}\dot{h}_t),\] where \(\d\) denotes the \((1,0)\)-part of the connection determined by \(h_t\). One may summarise this conveniently as \[\d_t(F_{h_t}+\Phi_t) = \db^{\eq}\d(h_t^{-1}\dot{h}_t).\] It follows that \begin{align*}
		\frac{d^2}{dt^2}\mathcal{M}_v(h_t,h') &= \int_X\tr\left(h_t^{-1}\dot{h}_t\left(-\frac{i}{2\pi}\Lambda_{\omega, v}\db^{\eq}\d(h_t^{-1}\dot{h}_t)\right)\right)\omega^{[n]} \\
		&= \frac{1}{2\pi}\int_X\tr(h_t^{-1}\dot{h}_t\Delta_v(h_t^{-1}\dot{h}_t))\omega^{[n]} \\
		&= \frac{1}{2\pi}\int_X|\d(h_t^{-1}\dot{h}_t)|^2v(\mu)\omega^{[n]}\\
		&\geq 0,
	\end{align*} where \(\Delta_v\psi = \d^*(v(\mu)\d\psi)\) denotes the weighted del-Laplacian on endomorphisms of \(E\).
\end{proof}
As a consequence of the above Lemma, we deduce the uniqueness modulo scaling of weighted Hermite--Einstein metrics.
\begin{thm}
	A weighted Hermite--Einstein metric on a $\T$-equivariant simple vector bundle is unique up to scaling.
\end{thm}

\begin{proof}
	Let \(h_0\) and \(h_1\) be two weighted Hermite--Einstein metrics, and let \(h_t\) for \(t\in[0,1]\) be the unique geodesic joining them. Since \(\frac{d}{dt}\mathcal{M}_v(h_t, h_0)|_{t=0} = \frac{d}{dt}\mathcal{M}_v(h_t,h_0)|_{t=1} = 0\), by convexity we must have \(\frac{d}{dt}\mathcal{M}_v(h_t, h_0) = 0\) for all \(t\in[0,1]\). From the proof of Lemma \ref{lem:ConvexDonaldson}, it follows that \(\d(h_t^{-1}\dot{h}_t)=0\); as \(h_t^{-1}\dot{h}_t\) is Hermitian, it further follows that \(\db(h_t^{-1}\dot{h}_t) = 0\). Since \(E\) is simple, \(h_t^{-1}\dot{h}_t\) must be a constant multiple of the identity, and so \(h_1\) is a constant multiple of \(h_0\).
\end{proof}

\section{The weighted Kobayashi--Hitchin correspondence}\label{sec:DUY}

In this section we prove the following.

\begin{thm}\label{thm:WHK}
	Let \(E\) be a simple vector bundle. If \(E\) is \(v\)-weighted slope stable, then \(E\) admits a \(v\)-weighted Hermite--Einstein metric.
\end{thm}

The proof will follow closely the method of Uhlenbeck--Yau \cite{UY86}, namely using a continuity path. In particular, we follow the exposition in the book by L{\"u}bke--Teleman \cite{LT95}. The weighted setting introduces few new challenges; we shall touch on the points that require explanation, but otherwise state results without proof where the adaptation is straightforward, referring to \cite{LT95} for more details.

\subsection{Set-up of the problem}

We are searching for a weighted Hermite--Einstein metric in the space of \(\T\)-invariant Hermitian metrics on the vector bundle \(E\). Given a reference metric \(h_0\), any other \(\T\)-invariant metric \(h\) may then be identified with a \(\T\)-invariant positive self-adjoint section \(f\in\Herm^+(E, h_0)^\T\) as follows. The endomorphism \(f\) is the unique endomorphism satisfying \[h(e_1,e_2) = h_0(f e_1, e_2 )\] for all \(e_1, e_2\in\Gamma(E)\). With respect to a local holomorphic frame for \(E\), the above equation may be written \[h_{\alpha\bar{\beta}} = f_\alpha^\gamma h_{0, \gamma\bar{\beta}},\] or equivalently \[f_\alpha^\beta = h_0^{\beta\bar{\gamma}}h_{\bar{\gamma}\alpha},\] which we sometimes abbreviate as \(f = h_0^{-1}h\). Conversely, given any \(f\in\Herm^+(E, h_0)^{\T}\), the sesquilinear form \(h(e_1,e_2):=h_0(f e_1, e_2)\) is a \(\T\)-invariant Hermitian metric on \(E\).

Fixing a local holomorphic frame for \(E\), and writing \(A_0\) and \(A\) for the connection matrices of \(h_0\) and \(h\) respectively, one computes \[A = A_0 + f^{-1}\d_0f,\] where \(\d_0 f := (\nabla^{h_0}f)^{1,0}\), with \(\nabla^{h_0}\) denoting the canonical extension of the Chern connection of \(h_0\) to \(\End E\). As per \eqref{eq:change-of-curvature}, the difference between both the curvature and the moment map can be expressed succinctly as \begin{equation}\label{eq:diff-F}
	F+\Phi = F_0+\Phi_0 + \db^{\eq}(f^{-1}\d_0f).
\end{equation}

We also have the useful formula \[\Lambda_{\omega, v}(\db^{\eq}\theta)\omega^{[n]} = \db(v(\mu)\theta)\wedge\omega^{[n-1]}\] for \(\theta\in\Omega^1(X)^\T\) or \(\theta\in\Omega^1(\End(E))^\T\), so that \[\Lambda_{\omega,v}(\db^{\eq}\theta) = \Lambda_\omega(\db(v(\mu)\theta)).\]  In particular, the weighted \(\d_0\)-Laplacian of \(E\) (or \(\End E\)) is expressed as \[\Delta_v(f):=\d_0^*(v(\mu)\d_0 f) = i\Lambda_{\omega,v}\db^{\eq}\d_0 f=i\Lambda_{\omega,v}\db(v(\mu)\d_0 f).\] Similar formulae hold for the weighted \(\d\)-Laplacian on functions, which we also denote \(\Delta_v\).

Using this information, we can see that finding \(h\) which solves the weighted Hermite--Einstein equation is equivalent to finding \(f\) which satisfies \[\frac{i}{2\pi}\Lambda_{\omega, v}(F_0 + \Phi_0) + \frac{i}{2\pi}\Lambda_{\omega, v}(\db^{\eq}(f^{-1}\d_0f)) = c_v \Id_E.\]

To solve the weighted Hermite--Einstein equation we set up a continuity method following Uhlenbeck--Yau \cite{UY86}, perturbing the equation to a family of equations depending on a parameter \(\epsilon \geq 0\): \[L_\epsilon(f):=\frac{i}{2\pi}\Lambda_{\omega, v}(F_0 + \Phi_0) - c_v\Id_E + \frac{i}{2\pi}\Lambda_{\omega, v}(\db^{\eq}(f^{-1}\d_0f)) + \epsilon\log f = 0.\]

For convenience we introduce the notations \[K = K_{h_0} := \frac{i}{2\pi}\Lambda_{\omega, v}(F_0 + \Phi_0),\] and \[K^0 = K_{h_0}^0 = K - c_v\Id_E,\] so that \[L_{\epsilon}(f) = K^0 + \frac{i}{2\pi}\Lambda_{\omega, v}(\db^{\eq}(f^{-1}\d_0f))+ \epsilon\log f.\]

We lastly note one assumption that can be made. First, we may normalise our reference metric \(h_0\) so that \(\tr(K^0) = 0\), which follows from the linearisation of the trace of the weighted Hermite--Einstein operator giving the weighted Laplacian; see the proof of Lemma \ref{vw-weighted-is-v-weighted}. We then have: \begin{lemma}\label{det1}
	If \(\tr(K^0) = 0\), then a solution \(f\) to \(L_{\epsilon}(f) = 0\) satisfies \(\det f = 1\).
\end{lemma}

\begin{proof}
	Taking the trace of the equation \(L_{\epsilon}(f) = 0\) gives \[\frac{i}{2\pi}\tr\Lambda_{\omega, v}(\db^{\eq}(f^{-1}\d_0f))+ \epsilon\tr\log f = 0.\] Now, \begin{align*}
		i\tr\Lambda_{\omega, v}(\db^{\eq}(f^{-1}\d_0f)) 
		= & i\Lambda_{\omega,v}(\db^{\eq}(\tr(f^{-1}\d_0f))) \\
		=& i\Lambda_{\omega,v}(\db^{\eq}(\d(\tr\log f))) \\
		=& \Delta_v(\tr\log f).
	\end{align*} Here between the first and second lines, \(\db^{\eq}\) changes from the equivariant del-bar operator on \(\Omega^{1,0}(\End(E))\) to the equivariant del-bar operator on \(\Omega^{1,0}(X)\), defined by \(\langle\db^{\eq}\theta,\xi\rangle = \db\theta + \iota_\xi\theta\) for \(\theta\in\Omega^{1,0}(X)\). It follows that \(\Delta_v(\tr\log f) + 2\pi\epsilon\tr\log f = 0\). Since the eigenvalues of the weighted \(\d\)-Laplacian \(\Delta_v\) are non-negative, it follows that \(\tr\log f = 0\), so \(\det f = e^{\tr\log f} = 1\).
\end{proof}

\subsection{Solution for \(\epsilon=1\)}

In this section we prove the existence of a starting point for the continuity method, namely by showing that the equation \(L_1(f) = 0\) admits a solution, for a particular choice of reference metric \(h_0\) satisfying \(\tr(K^0)=0\).

\begin{prop}
	There exist a \(\T\)-invariant metric \(h_0\) as well as an endomorphism \(f_1\in{\Herm}^+(E, h_0)^\T\) such that \(\tr(K^0) = 0\) and \(L_1(f_1) = 0\).
\end{prop}

\begin{proof}
	We fix an arbitrary \(\T\)-invariant metric \(h'\), and note that \(K^0_{h'}\) is a self-adjoint endomorphism with respect to \(h'\). Consider the metric \[h_0(s,t):=h'(\exp(K^0_{h'})s,t).\] The curvature of \(h_0\) satisfies \begin{align*}
		\tr(K_{h_0}^0) &= \tr(K_{h'}^0) + \tr\Lambda_{\omega,v}\db^{\eq}(\exp(K_{h'}^0)^{-1}\d_0\exp(K_{h'}^0)) \\
		&= 0+\Lambda_{\omega,v}\db^{\eq}\d\tr\log\exp(K_{h'}^0) \\
		&= 0.
	\end{align*} Defining \(f_1:=\exp(-K_{h'}^0)\), we have \(f_1\in {\Herm}^+(E,h_0)^\T\), and \begin{align*}
		L_1(f_1) = K_{h'}^0 + \log(f_1)= K_{h'}^0 - K_{h'}^0 = 0,
	\end{align*} where \(L_1\) is defined with respect to the reference metric \(h_0\).
\end{proof}

\subsection{Openness in \((0,1]\)}

Following \cite{LT95}, we introduce the notation \[\hat{L}(\epsilon,f):=f\circ L_\epsilon(f)\in{\Herm}(E,h_0)^\T.\] Of course \(L_\epsilon(f)=0\) if and only if \(\hat{L}(\epsilon, f)=0\), the advantage of \(\hat{L}\) being that (unlike \(L_\epsilon\)) it takes values in \(\Herm(E,h_0)^\T\). Denote by \(d_2\hat{L}\) the linearisation of \(\hat{L}\) in the \(f\)-variable. One readily computes that \(d_2\hat{L}\) is a second order linear differential operator, with leading order term \(\Delta_v\), the weighted \(\d\)-Laplacian of \(\End(E)\). In particular \(d_2\hat{L}\) is elliptic with index \(0\).

The proof of \cite[Proposition 3.2.5]{LT95} can be easily adapted to the following:

\begin{lemma}\label{inequality-lemma}
	Suppose that \(\hat{L}(\epsilon,f) = 0\) and \(d_2\hat{L}(\varphi)+\alpha f\log f = 0\) for some \(\alpha\in\mathbb{R}\). Then \[\eta:=f^{-1/2}\varphi f^{1/2}\] satisfies \[\Delta_v(|\eta|^2)+2\epsilon|\eta|^2 + v(\mu)|d^f\eta|^2\leq -2\alpha h_0(\log f,\eta),\] where \[d^f:=\d_0^f+\db^f:=\mathrm{Ad}f^{-1/2}\circ\d_0\circ\mathrm{Ad}f^{1/2} + \mathrm{Ad}f^{1/2}\circ\db\circ\mathrm{Ad}f^{-1/2}.\]
\end{lemma}

The key to adapting such proofs is the formula \(\Delta_v(g) = i\Lambda_{\omega,v}\db(v(\mu)\d g)\) for the weighted Laplacian. For example, the line \[P(|\eta|^2) = \Lambda_{\omega}\db\d h_0(\eta,\eta)=h_0(\Lambda_\omega\db^f\d_0^f\eta,\eta)+h_0(\eta,-\Lambda_\omega\d_0^f\db^f\eta)-|d^f\eta|^2\] from \cite[p. 68]{LT95}\footnote{The del-Laplacian is denoted by \(P\) in \cite{LT95}.} becomes \[\Delta_v(|\eta|^2) = i\Lambda_{\omega}\db(v(\mu)\d h_0(\eta,\eta))=ih_0(\Lambda_\omega\db^f(v(\mu)\d_0^f\eta),\eta)+ih_0(\eta,-\Lambda_\omega\d_0^f(v(\mu)\db^f\eta))-v(\mu)|d^f\eta|^2.\]

The lemma implies:

\begin{cor}
	If \(\hat{L}(\epsilon, f) = 0\), then \[d_2\hat{L}(\epsilon,f): L_k^p\Herm(E,h_0)^\T\to L_{k-2}^p\Herm(E, h_0)^{\T}\] is an isomorphism.
\end{cor}

\begin{proof}
	Suppose that \(d_2\hat{L}{(\epsilon,f)}(\varphi)=0\). Then by the Lemma \ref{inequality-lemma} with \(\alpha=0\), we have \[\Delta_v(|\eta|^2)+2\epsilon|\eta|^2\leq0.\] We now note that since \(\Delta_v(|\eta|^2) = v(\mu)\Delta(|\eta|^2)+\langle \d(|\eta|^2),\d v(\mu)\rangle\), the operator \(\Delta_v\) satisfies the maximum principle. Indeed, suppose \(|\eta|^2\) attains its maximum at \(x\in X\). Then \(\Delta_v(|\eta|^2)(x) = v(\mu(x))\Delta(|\eta|^2)(x)\geq0\), since \(\d(|\eta|^2)(x)=0\). Hence \(2\epsilon|\eta|^2(x)\leq0\), implying \(|\eta|^2=0\) everywhere; in particular \(\varphi=f^{-1/2}\eta f^{-1/2}=0\). Since \(d_2\hat{L}(\epsilon,f)\) is an elliptic operator with index \(0\) and trivial kernel, it is an isomorphism.
\end{proof}

It immediately follows from the inverse function theorem that the set of \(\epsilon\in(0,1]\) for which \(L_\epsilon(f)=0\) admits a smooth solution, is open.

\subsection{Closedness in \((0,1]\)}

Suppose that for some \(\epsilon_0>0\), the interval \((\epsilon_0,1]\) is such that \(L_\epsilon(f)=0\) admits a solution for all \(\epsilon\in(\epsilon_0,1]\). We will show that \(L_{\epsilon_0}(f)=0\) also admits a solution. In particular, by openness and existence of a solution to \(L_1(f_1)=0\), the equation \(L_\epsilon(f)=0\) admits a solution for all \(\epsilon\in(0,1]\).

For each \(\epsilon\in(\epsilon_0,1]\) denote by \(f_\epsilon\) the solution to \(L_\epsilon(f_\epsilon)=0\); note each \(f_\epsilon\) is smooth and varies smoothly in \(\epsilon\) by the inverse function theorem. Furthermore, \(\det f_\epsilon=1\) for each \(\epsilon\). Define \[m_\epsilon:=\max_X|\log f_\epsilon|,\quad \varphi_\epsilon:=\frac{df_\epsilon}{d\epsilon},\quad \eta_\epsilon:=f_\epsilon^{-1/2}\varphi_\epsilon f_\epsilon^{1/2}.\] In what follows we will write \(C(m_\epsilon)\) for a positive constant depending only on \(m_\epsilon\) (or possibly an upper bound for \(m_\epsilon\)), which may vary from line to line.

Similarly to \cite[Lemma 3.3.1]{LT95}:

\begin{lemma}\label{lemma-l2-inequality} For each \(\epsilon\in(\epsilon_0,1]\),
	\begin{enumerate}
		\item \(\tr\eta_\epsilon=0\);
		\item \(\|v(\mu)^{1/2}d^{f_\epsilon}\eta_\epsilon\|_{L^2}^2\geq C(m_\epsilon)\|\eta_\epsilon\|_{L^2}^2\).
	\end{enumerate}
\end{lemma}

We note the proof of the first point uses Lemma \ref{det1}, namely the assumption that \(\det f_\epsilon=1\) for all \(\epsilon\).

We will require the following:\footnote{We could not locate the result \cite[Lemma 3.3.2]{LT95} in the book \cite{GT01}, which uses the \(L^1\)-norm in place of the \(L^2\)-norm. However, the inequality making use of the \(L^2\)-norm is sufficient for our needs, as we show in the proof of Proposition \ref{varphi-bound}.}

\begin{lemma}[{\cite[Theorem 8.17]{GT01}}]\label{gilbarg-trudinger}
	Suppose \(u\in C^2(X;\mathbb{R})\) satisfies \(u\geq0\), and \(a,b\in\mathbb{R}\) are real constants such that \(a\geq0\). If \[\Delta_v(u)\leq au + b,\] then there exists a constant \(C>0\) depending only on \(a\), and other fixed data on the manifold such as \(\omega\) and \(v(\mu)\), such that \[\max_X u\leq C(\|u\|_{L^2}+|b|).\]
\end{lemma}

We note that the conditions required of the operator \(L\) in \cite[Theorem 8.17]{GT01} are satisfied by \(-\Delta_v\). In particular \(-\Delta_v\) is a second order linear elliptic operator of divergence type with uniformly bounded coefficients, and the leading term \(a^{ij}(x)\d_i\d_j\) of \(-\Delta_v\) is such that \(a^{ij}(x)\) is positive definite.

\begin{prop}\label{varphi-bound}
	For each \(\epsilon\in(\epsilon_0,1]\), \[\max_X|\varphi_\epsilon|\leq C(m_\epsilon).\]
\end{prop}

\begin{proof}
	By Lemma \ref{inequality-lemma} applied with \(\alpha=1\), \begin{equation}\label{eq:ineq-P_v}
		\Delta_v(|\eta_\epsilon|^2)+2\epsilon|\eta_\epsilon|^2+v(\mu)|d^{f_\epsilon}\eta_\epsilon|^2\leq -2h_0(\log f_\epsilon,\eta_\epsilon)\leq 2|\log f_\epsilon|\cdot|\eta_\epsilon|.
	\end{equation} Integrating over \(X\) and noting that the integral of \(\Delta_v(|\eta_\epsilon|^2)\) vanishes, this inequality implies \[\|v(\mu)^{1/2}d^{f_\epsilon}\eta_\epsilon\|_{L^2}^2\leq 2\|\log f_\epsilon\|_{L^2}\|\eta_\epsilon\|_{L^2}\leq C(m_\epsilon)\|\eta_\epsilon\|_{L^2}.\] Combining this with the second point of Lemma \ref{lemma-l2-inequality}, \begin{equation}\label{eq:ineq-eta}
		\|\eta_\epsilon\|_{L^2}\leq C(m_\epsilon).
	\end{equation} From \eqref{eq:ineq-P_v}, it also follows that \[\Delta_v(|\eta_\epsilon|^2)\leq 2|\log f_\epsilon|\cdot|\eta_\epsilon|\leq m_\epsilon(|\eta_\epsilon|^2+1).\] Hence by Lemma \ref{gilbarg-trudinger}, \[\max_X|\eta_\epsilon|^2\leq C(m_\epsilon)(\||\eta_\epsilon|^2\|_{L^2}+m_\epsilon)\leq C(m_\epsilon)(\max_X|\eta_\epsilon|\cdot \|\eta_\epsilon\|_{L^2}+1).\] Note that if \(\max_X|\eta_\epsilon|\leq 1\) then we are done, and if \(\max_X|\eta_\epsilon|\geq1\), then dividing by \(\max_X|\eta_\epsilon|\) we have \[\max_X|\eta_\epsilon|\leq C(m_\epsilon)(\|\eta_\epsilon\|_{L^2}+1)\leq C(m_\epsilon),\] where we used \eqref{eq:ineq-eta} to get the final inequality.
\end{proof}

The proofs of the following results are similar to those in \cite{LT95}, making use of the results we have stated up to now.

\begin{lemma}[{\cite[Lemma 3.3.4]{LT95}}]\label{lem:bounds}
	Let \(\epsilon>0\) and suppose \(f\in\Herm^+(E,h)^\T\) satisfies \(L_\epsilon(f)=0\). Then \begin{enumerate}
		\item \(\frac{1}{2}\Delta_v(|\log f|^2)+\epsilon|\log f|^2\leq |K^0|\cdot|\log f|\);
		\item \(m:=\max_X|\log f|\leq \frac{1}{\epsilon}\max_X|K^0|\);
		\item  \(m\leq 1 + C(\|\log f\|_{L^2}+\max_X|K^0|^2)\), where \(C>0\) is a constant depending only on fixed data on the manifold and \(E\).
	\end{enumerate}
\end{lemma}

\begin{prop}[{\cite[Proposition 3.3.5]{LT95}}]
	Suppose that \(m_\epsilon\leq m\) for all \(\epsilon\in(\epsilon_0,1]\). Then for all \(p>1\) and \(\epsilon\in(\epsilon_0,1]\), \begin{enumerate}
		\item \(\|\varphi_\epsilon\|_{L_2^p}\leq C(m)(1+\|f_\epsilon\|_{L_2^p})\);
		\item \(\|f_\epsilon\|_{L_2^p}\leq e^{C(m)(1-\epsilon)}(1+\|f_1\|_{L_2^p})\).
	\end{enumerate}
\end{prop}

Using this last Proposition, one can prove closedness. Namely, one assumes the maximal interval for which \(L_\epsilon(f_\epsilon)=0\) admits a solution is \((\epsilon_0,1]\) for some \(\epsilon_0>0\). The second point of Lemma \ref{lem:bounds} and the last proposition give a uniform bound on \(\|f_\epsilon\|_{L_2^p}\) independent of \(\epsilon\), so one may extract a weakly converging subsequence \(f_{\epsilon_k}\) having a weak limit \(f_{\epsilon_0}\in L^p_2\End(E)^\T\). Proving that \(L_{\epsilon_0}(f_{\epsilon_0})=0\), elliptic regularity then implies that \(f_{\epsilon_0}\) is smooth, and we get a contradiction.

\begin{prop}\label{prop:more_bounds}
	\begin{enumerate}
		\item For each \(\epsilon\in(0,1]\), there exists a smooth solution \(f_\epsilon\) to \(L_\epsilon(f_\epsilon)=0\);
		\item If there exists a constant \(C>0\) such that \(\|f_\epsilon\|_{L^2}\leq C\) for all \(\epsilon\in(0,1]\), then there exists \(f_0\) solving the weighted Hermite--Einstein equation \(L_0(f_0)=0\).
	\end{enumerate}
\end{prop}

To get the solution \(f_0\) in the second point, one uses the third point of Lemma \ref{lem:bounds} to get a uniform bound on \(\|f_\epsilon\|_{L_2^p}\) independent of \(\epsilon\), and the same argument as for the first point gives a solution to \(L_0(f_0)=0\).

\subsection{The limit as \(\epsilon\to0\)}

We aim to prove:

\begin{prop}
	If \(\rk(E)\geq2\) and \[\limsup_{\epsilon\to0}\|\log f_\epsilon\|_{L^2}=\infty\] then \(E\) is not \(v\)-weighted stable.
\end{prop}

Following L{\"u}bke--Teleman with the appropriate weighted adaptions, we are provided the following:

\begin{prop}
	There exist sequences of positive numbers \(\epsilon_i\to 0\) and \(\rho(\epsilon_i)\to0\) such that \(f_i:=\rho(\epsilon_i)f_{\epsilon_i}\) satisfies: \begin{enumerate}
		\item The \(f_i\) converge weakly to a non-zero \(f_\infty\in L^2_1(\End E)^\T\);
		\item There is a sequence \((\sigma_i)\) such that \(0<\sigma_i\leq1\), \(\sigma_i\to0\), and \(f_\infty^{\sigma_i}\) converges weakly in \(L^2_1(\End E)^\T\) to an element \(f_\infty^0\);
		\item \(\pi:=\Id_E-f_\infty^0\) is a weakly holomorphic subbundle of \(E\).
	\end{enumerate}
\end{prop}

In particular, there exists a coherent analytic subsheaf \(\F\subset E\) whose singular set \(V\) has codimension at least \(2\), such that on \(X\backslash V\), \(\pi\) is the orthogonal projection to \(\F|_{X\backslash V}\) with respect to the metric \(h_0\). Furthermore, the rank of \(\F\) satisfies \(0<\rk(\F)<\rk(E)\). We claim that \(\F\) may be chosen to be \(\T\)-invariant. To see this, suppose \(\F_0:=\F\) is not \(\T\)-invariant, so there exists \(t_0\in \T\) such that \(\F\neq t_0\F\). Define \(\F_1:=\F+t_0\F\). If \(\F_1\) is not \(\T\)-invariant, choose \(t_1\in \T\) so that \(\F_1\neq t_1\F_1\) and set \(\F_2:= \F_1+t_1\F_1\). Continuing in this vein, we produce a sequence \(\F_0\subset \F_1\subset \F_2\subset\cdots\) of coherent analytic subsheaves of \(E\). By the strong Noetherian property for coherent analytic sheaves (see \cite[Chapter II, 3.22]{agbook}), this sequence must eventually terminate, meaning there is some \(k\) such that \(\F\subset \F_k\) and \(\F_k\) is \(\T\)-invariant. Note that since \(\pi\) is \(\T\)-invariant, the smooth locus of \(\F\) is \(\T\)-invariant, and hence the smooth locus of \(\F_k\) is contains the smooth locus of \(\F\).

By taking a modification of \(\F\), which preserves the smooth locus and \(\T\)-invariance, we may further assume that \(\F\) is saturated.

\begin{prop}
	\(\mu_v(\F)\geq\mu_v(E)\).
\end{prop}

\begin{proof}
	Denote by \(h_1\) the metric on \(\F\) induced by \(h_0\). On the smooth locus of \(\F\), \[K_{h_1} = \pi\circ K_{h_0}\circ \pi - v(\mu)i\Lambda_\omega (\d_0\pi\wedge\db\pi);\] here \(\d_0\pi=A\) is the second fundamental form of \(\F\) over the smooth locus. Since by Lemma \ref{lem:codim2integral} the weighted slope of \(\F\) may be computed on the smooth locus, we have \[\mu_v(\F) = \frac{1}{\rk(\F)}\int_{X\backslash V}(\tr(K^0\circ\pi)-v(\mu)|\d_0\pi|^2)\omega^{[n]}+\mu_v(E).\] The proof from L{\"u}bke--Teleman goes through in the same manner to show that \[\int_{X\backslash V}(\tr(K^0\circ\pi)-v(\mu)|\d_0\pi|^2)\omega^{[n]}\geq0. \qedhere\]
\end{proof}

\bibliographystyle{alpha}
\bibliography{references}

\end{document}